\pdfoutput=1 
\PassOptionsToPackage{linktocpage=true}{hyperref}

\documentclass[12pt]{alt2022}

\usepackage[utf8]{inputenc} 
\usepackage{cmap}
\usepackage[T1]{fontenc}

\hypersetup{colorlinks=true, urlcolor=blue}

\usepackage{algorithm}

\usepackage{algcompatible}
\usepackage{multicol}
\usepackage{silence}
\usepackage{amsmath,amsfonts,amssymb,thmtools}
\usepackage{mathtools}
\usepackage{xparse}
\usepackage{xargs}
\usepackage{enumitem} 
\usepackage{etoolbox}
\usepackage{hhline}
\usepackage{mathrsfs}	
\usepackage{dsfont}
\usepackage{tikz}			
\usepackage{mathtools}
\usepackage{complexity}
\usepackage{svg}
\usepackage{nag} 					
\usepackage{arydshln} 
\usepackage{booktabs} 
\usepackage{soul}

\definecolor{lightgrey}{rgb}{0.9,0.9,0.9}
\sethlcolor{lightgrey}

\usepackage[capitalise,nameinlink,noabbrev]{cleveref}

\newtheorem{example}{Example}
\newtheorem{theorem}{Theorem}
\newtheorem{lemma}[theorem]{Lemma}
\newtheorem{proposition}[theorem]{Proposition}
\newtheorem{remark}[theorem]{Remark}
\newtheorem{corollary}[theorem]{Corollary}
\newtheorem{definition}[theorem]{Definition}

\makeatletter
\newcommand*\ie{\textit{i.e.}\@ifnextchar.{\@gobble}{\relax}}
\newcommand*\vs{\textit{vs.}\@ifnextchar.{\@gobble}{\relax}}
\newcommand*\etc{\textit{etc.}\@ifnextchar.{\@gobble}{\relax}}
\newcommand*\eg{\textit{e.g.}\@ifnextchar.{\@gobble}{\relax}}
\newcommand*\cf{\textit{cf.}\@ifnextchar.{\@gobble}{\relax}}
\makeatother

\newcommand{\norm}[1]{\| #1 \|} 
\newcommand{\abs}[1]{\lvert #1 \rvert}

\NewDocumentCommand{\infnorm}{s O{} m }{%
  \IfBooleanTF{#1}{\norm*{#3}}{\norm[#2]{#3}}_{\infty}%
}
\NewDocumentCommand{\twonorm}{s O{} m }{%
  \IfBooleanTF{#1}{\norm*{#3}}{\norm[#2]{#3}}_2%
}
\NewDocumentCommand{\tvnorm}{s O{} m }{%
  \IfBooleanTF{#1}{\norm*{#3}}{\norm[#2]{#3}}_{\textup{TV}}%
}
\NewDocumentCommand{\onenorm}{s O{} m }{%
  \IfBooleanTF{#1}{\norm*{#3}}{\norm[#2]{#3}}_1%
}
\NewDocumentCommand{\frobnorm}{s O{} m }{%
  \IfBooleanTF{#1}{\norm*{#3}}{\norm[#2]{#3}}_F%
}

\NewDocumentCommand{\scalar}{s O{} >{\SplitArgument{1}{,}}m}{%
    \IfBooleanTF{#1}{\scalaraux*#3}{\scalaraux[#2]#3}%
}
\DeclarePairedDelimiterX{\scalaraux}[2]{\langle}{\rangle}{#1, #2}

\newcommand*\circledaux[1]{\tikz[baseline=(char.base)]{
    \node[shape=circle,draw,inner sep=0.8pt] (char) {#1};}}

\NewDocumentCommand{\circled}{m o }{%
    \IfNoValueTF{#2}{\circledaux{#1}}{\stackrel{\circledaux{#1}}{#2}}%
}

\newcommand{\defi}{\stackrel{\mathrm{\scriptscriptstyle def}}{=}}

\renewcommand*\R{\mathbb{R}}

\let\epsilon\varepsilon

\def\dodeffunction#1:#2->#3;#4\relax
    {\ifblank{#4}
    {#1\colon#2\to#3}
    {\!\begin{aligned}#1\colon#2&\to#3\\ \dodeffunctionaux#4\relax\end{aligned}}}
\def\dodeffunctionaux#1->#2\relax{#1&\mapsto#2}

\newcommand\dual[1]{#1^\ast}

\NewDocumentCommand{\enorm}{s O{} m}{%
    \IfBooleanTF{#1}{\norm*{#3}}{\norm[#2]{#3}}_{\E}%
}
\NewDocumentCommand{\denorm}{s O{} m}{%
    \dual{\IfBooleanTF{#1}{\enorm*{#3}}{\enorm[#2]{#3}}}%
}

\NewDocumentCommand{\commasaux}{m m}{%
    \IfNoValueTF{#2}{#1}{#1, #2}%
}
\NewDocumentCommand{\prodaux}{m m}{%
    \IfNoValueTF{#2}{#1 \times #1}{#1 \times #2}%
}

\renewcommand{\E}[1]{\mathbb{E}\left[ #1 \right]} 

\renewcommand{\L}{\mathscr{L}}

\newcommand{\varl}{\operatorname{left}} 
\newcommand{\varr}{\operatorname{right}}

\DeclareMathOperator*{\argmin}{arg\,min}

\let\svthefootnote\thefootnote
\newcommand\freefootnote[1]{%
  \let\thefootnote\relax%
  \footnotetext{#1}%
  \let\thefootnote\svthefootnote%
}

\makeatletter
\renewcommand\paragraph{\@startsection{paragraph}{4}{\z@}%
                                    {0ex \@plus0.5ex \@minus.2ex}%
                                    {-1em}%
                                    {\normalfont\normalsize\bfseries}}
\makeatother



\makeatletter
\newcounter{algorithmicH}
\let\oldalgorithmic\algorithmic
\renewcommand{\algorithmic}{%
  \stepcounter{algorithmicH}
  \oldalgorithmic}
\renewcommand{\theHALG@line}{ALG@line.\thealgorithmicH.\arabic{ALG@line}}
\makeatother

\usepackage{algcompatible}
\algnewcommand{\lst}{\texttt{lst}}
\algnewcommand{\slst}{\texttt{slst}}
\algnewcommand{\SEND}{\textbf{send}}

\newsavebox{\algleft}
\newsavebox{\algright}

\newcommand{\innp}[1]{\langle #1 \rangle}
\newcommand{\bigo}[1]{O( #1 )}

\newcommand{\absadj}[1]{\left| #1 \right|} 

\DeclareMathOperator*{\arctanh}{arctanh}
\DeclareMathOperator*{\arccosh}{arccosh}

\renewcommand{\dot}[1]{\langle #1\rangle} 
\let\dott\temp

\newtheorem{fact}[theorem]{Fact}

\title[Global Riemannian Acceleration in Hyperbolic and Spherical Spaces]{Global Riemannian Acceleration in Hyperbolic and Spherical Spaces}
\usepackage{times}

\altauthor{%
 \Name{David Martínez-Rubio} \Email{david.martinez@cs.ox.ac.uk}\\
 \addr Department of Computer Science, University of Oxford, Oxford, United Kingdom
}

\usepackage[natbib, maxcitenames=5, maxbibnames=99, minalphanames=3, style=ieee-alphabetic, hyperref, backref, useprefix=true, uniquename=false, doi=false,url=false,eprint=false]{biblatex}

\usepackage{csquotes}               

\bibliography{refs}

\DeclareCiteCommand{\cite}
  {\usebibmacro{prenote}}
  {\usebibmacro{citeindex}%
   \printtext[bibhyperref]{\usebibmacro{cite}}}
  {\multicitedelim}
  {\usebibmacro{postnote}}

\DeclareCiteCommand*{\cite}
  {\usebibmacro{prenote}}
  {\usebibmacro{citeindex}%
   \printtext[bibhyperref]{\usebibmacro{citeyear}}}
  {\multicitedelim}
  {\usebibmacro{postnote}}

\newbibmacro{string+doiurlisbn}[1]{%
  \iffieldundef{doi}{%
    \iffieldundef{url}{%
      \iffieldundef{isbn}{%
        \iffieldundef{issn}{%
          #1%
        }{%
          \href{http://books.google.com/books?vid=ISSN\thefield{issn}}{#1}%
        }%
      }{%
        \href{http://books.google.com/books?vid=ISBN\thefield{isbn}}{#1}%
      }%
    }{%
      \href{\thefield{url}}{#1}%
    }%
  }{%
    \href{https://doi.org/\thefield{doi}}{#1}%
  }%
}

\DeclareFieldFormat{title}{\usebibmacro{string+doiurlisbn}{\mkbibemph{#1}}}
\DeclareFieldFormat[article,incollection,inproceedings]{title}%
    {\usebibmacro{string+doiurlisbn}{\mkbibquote{#1}}}

\newcommand\newlink[2]{\hyperlink{#1}{\color{black}#2}}
\makeatletter
\newcommand\newtarget[2]{\Hy@raisedlink{\hypertarget{#1}{}}#2}
\makeatother

\usepackage[bb=boondox]{mathalpha}
\newcommand{\zeros}{\mathbb{0}} 

\newcommand\linktoproof[1]{{\normalfont[{\hyperlink{proof:#1}{$\downarrow$}}]}}
\newcommand\linkofproof[1]{\newtarget{proof:#1}{}}
\newcommandx*\ecanonical[1][1=i, usedefault]{\newlink{def:e_i_canonical_basis}{e_{#1}}}

\newcommand{\AGD}{\newlink{def:acronym_accelerated_gradient_descent}{AGD}}
\newcommand{\RGD}{\newlink{def:acronym_riemannian_gradient_descent}{RGD}}

\newcommandx*\gamman[1][1=, usedefault]{\newlink{def:gamma_n}{\gamma_{\operatorname{n}}^{#1}}}
\newcommandx*\gammap[1][1=, usedefault]{\newlink{def:gamma_p}{\gamma_{\operatorname{p}}^{#1}}}
\newcommandx*\gammanparam[1][1=, usedefault]{\newlink{def:gamma_n_parameter}{\gamma_{\operatorname{n}}^{#1}}}
\newcommandx*\gammapparam[1][1=, usedefault]{\newlink{def:gamma_p_parameter}{\gamma_{\operatorname{p}}^{#1}}}

\newcommand{\bigotilde}[1]{\newlink{def:big_o_tilde}{\widetilde{O}}( #1 )}
\newcommand{\bigotildel}[1]{\newlink{def:big_o_tilde}{\widetilde{O}}\left( #1 \right)}

\let\oldmu\mu
\renewcommand\mu{\newlink{def:strong_g_convexity_of_F}{\oldmu}}
\newcommand\mui[1]{\newlink{def:regularization_in_reduction}{\oldmu_{#1}}}

\renewcommand\M{\newlink{def:M}{\mathcal{M}}}
\newcommand\NN{\mathcal{N}} 
\newcommand\MEucl{\newlink{def:M_Euclidean}{M}}
\newcommandx*\Mk[1][1=K, usedefault]{\newlink{def:M_K}{\mathcal{M}_{#1}}}

\let\oldepsilon\epsilon
\renewcommand{\epsilon}{\newlink{def:accuracy_epsilon}{\oldepsilon}}
\newcommandx*\hatepsilon[1][1=i, usedefault]{\newlink{def:accuracy_of_binary_search}{\hat{\oldepsilon}_{#1}}}

\newcommand\expon[1]{\newlink{def:riemannian_exponential_map}{\operatorname{Exp}_{#1}}} 
\newcommand\exponinv[1]{\newlink{def:riemannian_exponential_map}{\operatorname{Exp}_{#1}^{-1}}} 

\newcommand\T{\newlink{def:total_number_of_iterations_T}{T}}
\newcommand{\n}{\newlink{def:dimension}{n}} 
\newcommand{\Lips}{\newlink{def:Lipschitz_constant}{L_{\mathtt{p}}}} 
\newcommand{\Ltildegeneric}{\newlink{def:generic_smoothness_in_euclidean_case}{\tilde{L}}} 
\newcommand{\Ltilde}{\newlink{def:smoothness_tilde_i_e_euclidean_coming_from_riemannian}{\tilde{L}}} 
\renewcommand{\L}{\newlink{def:riemannian_smoothness_of_F}{L}} 
\newcommand{\X}{\newlink{def:euclidean_ball_where_we_optimize}{\mathcal{X}}}
\newcommandx*\BR[1][1=R, usedefault]{\newlink{def:riemannian_ball_where_we_optimize}{\mathcal{B}_{#1}}}

\renewcommand{\time}{\newlink{def:time_alg_of_strongly_convex_for_reduction}{\operatorname{Time}}}
\newcommand{\timens}{\newlink{def:time_alg_of_just_convex_for_reduction}{\operatorname{Time}_{\operatorname{ns}}}}
\newcommand{\Alg}{\newlink{def:alg_of_strongly_convex_for_reduction}{\mathcal{A}}}
\newcommand{\Algns}{\newlink{def:alg_of_just_convex_for_reduction}{\mathcal{A}_{\operatorname{ns}}}}
\newcommand{\Q}{\newlink{def:convex_set_of_tilted_convex_thm}{Q}}

\newcommand{\riemMinus}{\newlink{def:formal_riemannian_subtraction}{-}}
\newcommand{\dist}{\newlink{def:distance}{d}}

\let\oldpsi\psi
\renewcommand{\psi}{\newlink{def:strongly_convex_regularizer}{\oldpsi}}
\newcommand{\astfenchel}{\newlink{def:asterisk_of_fenchel_dual}{\ast}}
\let\oldPi\Pi
\renewcommand{\Pi}{\newlink{def:euclidean_projection}{\oldPi}}

\renewcommand{\K}{\newlink{def:curvature}{K}} 
\newcommand{\deltaAux}{\newlink{def:auxiliary_angle_delta}{\delta}} 
\newcommand{\deltatildeAux}{\newlink{def:auxiliary_angle_delta}{\tilde{\delta}}} 
\newcommand{\deltapAux}{\newlink{def:auxiliary_delta_prime}{\delta'}} 
\newcommand{\deltaptildeAux}{\newlink{def:auxiliary_delta_prime}{\tilde{\delta}'}} 
\newcommand{\deltaastAux}{\newlink{def:auxiliary_delta_ast}{\delta^\ast}} 
\newcommand{\deltaasttildeAux}{\newlink{def:auxiliary_delta_ast}{\tilde{\delta}^\ast}}

\newcommandx*\sk[2][1= , 2= , usedefault]{\newlink{def:special_sine}{\operatorname{S}_{\K}^{#2}}( #1 )}
\newcommandx*\ck[2][1= , 2= , usedefault]{\newlink{def:special_cosine}{\operatorname{C}_{\K}^{#2}}( #1 )}
\newcommandx*\cotk[2][1= , 2= , usedefault]{\newlink{def:special_cotangent}{\operatorname{cot}_{\K}^{#2}}( #1 )}

\newcommand\sign{\newlink{def:sign_of_a_number}{\operatorname{sign}}}

\newcommandx*\distorp[1][1=R, usedefault]{\newlink{def:distortion_K_pos}{\mathcal{K}_{#1}^+}}
\newcommandx*\distorn[1][1=R, usedefault]{\newlink{def:distortion_K_neg}{\mathcal{K}_{#1}^-}}
\newcommand{\stcvxpsi}{\newlink{def:strong_convexity_of_regularizer}{\sigma}}

\newcommand{\h}{\newlink{def:geodesic_map}{h}}

\newcommand{\affLowerBound}{\newlink{def:affine_lower_bound}{\ell}}
\newcommand\lambdanot[1]{\newlink{def:lambda_superscript_notation}{#1}}

\newcommandx*\Ei[1][1=i, usedefault]{\newlink{def:discretization_error}{E_{#1}}}

\newcommandx*\Di[1][1=i, usedefault]{\newlink{def:gap_for_reduction_tost_g_convex}{D_{#1}}}

\newcommandx*\FReg[1][1=i, usedefault]{\newlink{def:regularized_F_for_reduction}{F^{(#1)}}}

\newcommand\F{\newlink{def:riemannian_function_F}{F}}
\newcommand\f{\newlink{def:tilted_convex_function_f_from_F}{f}}
\newcommand\ftilted{\newlink{def:tilted_convex_function_generic_euclidean}{f}}
\newcommand\Rtilde{\newlink{def:radius_of_Euclidean_ball}{\tilde{R}}}
\newcommand\RR{\newlink{def:radius_of_Riemannian_ball}{R}}

\newcommandx*\Rglobal[1][1=, usedefault]{\newlink{def:distance_to_global_minimizer}{R_{\operatorname{g}}^{#1}}}
\newcommandx*\Rglobaltilde[1][1=, usedefault]{\newlink{def:distance_to_global_minimizer}{\tilde{R}_{\operatorname{g}}^{#1}}}

\newcommand\ball{\newlink{def:closed_ball}{\bar{B}}}

\newcommandx*\hatgamma[1][1=i, usedefault]{\newlink{def:tilted_parameter_at_iteration_i}{\hat{\gamma}_{#1}}}
\newcommandx*\Gammahat[2][1=i, 2= , usedefault]{\newlink{def:function_for_tilted_parameter}{\hat{\Gamma}_{#1}^{#2}}}

\newcommandx*\Gt[1][1=t, usedefault]{\newlink{def:differentiable_gap_bound}{G_{#1}}}
\newcommandx*\Lt[1][1=t, usedefault]{\newlink{def:differentiable_lower_bound}{L_{#1}}}
\newcommandx*\Ut[1][1=t, usedefault]{\newlink{def:differentiable_upper_bound}{U_{#1}}}
\newcommandx*\Gbinsearch[1][1=i, usedefault]{\newlink{def:G_binary_search}{\mathcal{G}_{#1}}}

\newcommandx*\At[2][1=t, 2=, usedefault]{\newlink{def:integral_of_steps}{A_{#1}^{#2}}}
\newcommandx*\ai[2][1=t, 2=, usedefault]{\newlink{def:discrete_step}{a_{#1}^{#2}}}
\newcommandx*\alphat[1][1=, usedefault]{\newlink{def:step_function_alpha}{\alpha_{#1}}}
\newcommandx*\breg[1][1=\psi, usedefault]{\newlink{def:bregman_divergence}{D_{#1}}}
\newcommand\Tansp[1]{\newlink{def:tangent_space}{T_{#1}}}

\newcommand\notilde{\noop}

\NewDocumentCommand{\xI}{ooo}{%
    \newlink{def:iterate_x}{
        \IfNoValueTF{#2}{
        \IfNoValueTF{#3}
            {\tilde{x}_{#1}}
            {x_{#1}}
        }
        {
        \IfNoValueTF{#3}
            {\tilde{x}_{#1}^{#2}}
            {x_{#1}^{#2}}
        }
    }
}

\NewDocumentCommand{\zi}{ooo}{%
    \newlink{def:iterate_z}{
        \IfNoValueTF{#2}{
        \IfNoValueTF{#3}
            {\tilde{z}_{#1}}
            {z_{#1}}
        }
        {
        \IfNoValueTF{#3}
            {\tilde{z}_{#1}^{#2}}
            {z_{#1}^{#2}}
        }
    }
}
\NewDocumentCommand{\chii}{ooo}{%
    \newlink{def:iterate_chi}{
        \IfNoValueTF{#2}{
        \IfNoValueTF{#3}
            {\tilde{\chi}_{#1}}
            {\chi_{#1}}
        }
        {
        \IfNoValueTF{#3}
            {\tilde{\chi}_{#1}^{#2}}
            {\chi_{#1}^{#2}}
        }
    }
}
\NewDocumentCommand{\zetai}{ooo}{%
    \newlink{def:iterate_zeta}{
        \IfNoValueTF{#2}{
        \IfNoValueTF{#3}
            {\tilde{\zeta}_{#1}}
            {\zeta_{#1}}
        }
        {
        \IfNoValueTF{#3}
            {\tilde{\zeta}_{#1}^{#2}}
            {\zeta_{#1}^{#2}}
        }
    }
}
\NewDocumentCommand{\xt}{ooo}{%
    \newlink{def:iterate_x_continuous}{
        \IfNoValueTF{#2}{
        \IfNoValueTF{#3}
            {\tilde{x}_{#1}}
            {x_{#1}}
        }
        {
        \IfNoValueTF{#3}
            {\tilde{x}_{#1}^{#2}}
            {x_{#1}^{#2}}
        }
    }
}

\NewDocumentCommand{\zt}{ooo}{%
    \newlink{def:iterate_z_continuous}{
        \IfNoValueTF{#2}{
        \IfNoValueTF{#3}
            {\tilde{z}_{#1}}
            {z_{#1}}
        }
        {
        \IfNoValueTF{#3}
            {\tilde{z}_{#1}^{#2}}
            {z_{#1}^{#2}}
        }
    }
}

\NewDocumentCommand{\xInit}{o}{%
    \newlink{def:initial_point}{
        \IfNoValueTF{#1}
            {\tilde{x}_0}
            {x_0}
    }
}

\NewDocumentCommand{\xast}{o}{%
    \newlink{def:optimizer}{
        \IfNoValueTF{#1}
            {\tilde{x}^\ast}
            {x^\ast}
    }
}

\newcommand\xastg{\newlink{def:global_optimizer}{x_g^\ast}}
\newcommand\xastgtilde{\newlink{def:global_optimizer}{\tilde{x}_g^\ast}}

\newcommand\deltar{\newlink{def:delta_without_subindex_which_means_zeta_RR}{\delta}}
\newcommand\zetar{\newlink{def:zeta_without_subindex_which_means_zeta_RR}{\zeta}}

\newcommand\TT{\newlink{def:total_number_of_iterations_ball_boosting}{T}}
\newcommand{\epsilonp}{\newlink{def:accuracy_epsilon_prime_in_ball_subproblem}{\oldepsilon'}}
\newcommandx*\pk[1][1=k, usedefault]{\newlink{def:iterates_p_boosting_alg}{p_{#1}}}
\newcommandx*\pkast[1][1=k, usedefault]{\newlink{def:optimizers_in_balls_in_boosting_alg}{p_{#1}^\ast}}
\newcommand{\algsc}{\newlink{def:algorithm_1_strongly_convex_version}{\operatorname{AlgSC}}}

\begin{document}

\maketitle

\begin{abstract}
    We further research on the accelerated optimization phenomenon on Riemannian manifolds by introducing accelerated global first-order methods for the optimization of $\L$-smooth and geodesically convex (g-convex) or $\mu$-strongly g-convex functions defined on the hyperbolic space or a subset of the sphere. For a manifold other than the Euclidean space, these are the first methods to \emph{globally} achieve the same rates as accelerated gradient descent in the Euclidean space with respect to $\L$ and $\epsilon$ (and $\mu$ if it applies), up to log factors. Due to the geometric deformations, our rates have an extra factor, depending on the initial distance $\RR$ to a minimizer and the curvature $\K$, with respect to Euclidean accelerated algorithms.\footnote{V5 of this work significantly reduces the dependence of our convergence rates on $\RR$ and $\K$, with respect to V4, which is \href{https://proceedings.mlr.press/v167/martinez-rubio22a.html}{\normalcolor the ALT22 version} (e.g., for strongly g-convex problems in the hyperbolic case, from exponential to a small polynomial).
    In V3, we discussed lower bounds and added a proof about the minimum possible condition number of strongly g-convex functions. V2 and V4 improve writing. V1 was made available on Dec 7, 2020.}

    As a proxy for our solution, we solve a constrained non-convex Euclidean problem, under a condition between convexity and \textit{quasar-convexity}, of independent interest. Additionally, for any Riemannian manifold of bounded sectional curvature, we provide reductions from optimization methods for smooth and g-convex functions to methods for smooth and strongly g-convex functions and vice versa. We also reduce global optimization to optimization over bounded balls where the effect of the curvature is reduced.
\end{abstract}

\section{Introduction}

\freefootnote{Most of the notations in this work have a link to their definitions. For example, if you click or tap on any instance of $\L$, you will jump to the place where it is defined as the smoothness constant of the function we consider in this work.}

Acceleration in convex optimization is a phenomenon that has drawn plenty of attention and has yielded many important results, since the renowned Accelerated Gradient Descent (\newtarget{def:acronym_accelerated_gradient_descent}{\AGD{}}) method of \citet{nesterov1983method}. Having been proved successful for deep learning \citep{DBLP:conf/icml/SutskeverMDH13}, among other fields, there have been recent efforts to better understand this phenomenon \citep{allen2014linear,diakonikolas2017approximate,su2014differential,wibisono2016variational}. These have yielded numerous new results going beyond convexity or the standard oracle model, in a wide variety of settings \citep{allen2016katyusha, allen2017natasha, allen2018katyusha, DBLP:conf/stoc/ZhuO15,allen2016even, allen2017much, carmon2017convex, cohen2018acceleration, cutkosky2019matrix, diakonikolas2019generalized, diakonikolas2017accelerated, DBLP:conf/colt/GasnikovDGVSU0W19,wang2015unified}. This surge of research that applies tools of convex optimization to models going beyond convexity has been fruitful.  One of these models is the setting of geoesically convex Riemannian optimization. In this setting, the function to optimize is geodesically convex (g-convex), i.e., convex restricted to any geodesic (cf. \cref{def:g-convex_smooth}). 

Riemannian optimization, g-convex and non-g-convex alike, is an extensive area of research. In recent years there have been numerous efforts towards obtaining Riemannian optimization algorithms that share analogous properties to the more broadly studied Euclidean first-order methods: deterministic \citep{bento2017iteration,wei2016guarantees,zhang2016first}, stochastic \citep{hosseini2019alternative,khuzani2017stochastic,tripuraneni2018averaging}, saddle-point-escaping \citep{criscitiello2019efficiently,sun2019escaping,zhang2018r,zhou2019faster, criscitiello2020accelerated}, variance-reduced \citep{sato2017riemannian,kasai2018riemannian,zhang2016fast}, adaptive \citep{kasai2019riemannian}, and projection-free methods \citep{weber2017frank,weber2019nonconvex}, among others. Unsurprisingly, Riemannian optimization has found many applications in machine learning, including low-rank matrix completion \citep{DBLP:journals/siamsc/CambierA16,heidel2018riemannian,mishra2014r3mc,tan2014riemannian,vandereycken2013low}, dictionary learning \citep{cherian2016riemannian,sun2016complete}, optimization under orthogonality constraints \citep{edelman1998geometry}, with applications to Recurrent Neural Networks \citep{DBLP:conf/nips/Casado19, DBLP:conf/icml/CasadoM19}, robust covariance estimation in Gaussian distributions \citep{wiesel2012geodesic}, Gaussian mixture models \citep{hosseini2015matrix}, operator scaling \citep{allen2018operator}, and sparse principal component analysis \citep{genicot2015weakly,huang2019riemannian,jolliffe2003modified}.

However, the acceleration phenomenon, largely celebrated in the Euclidean space, is still not understood in Riemannian manifolds, although there has been some progress on this topic recently (cf. \hyperlink{sec:related_work}{Related work}). This poses the following question, which is the central subject of this paper:

\begin{center}
\textit{Can a Riemannian first-order method enjoy the same rates as \AGD{} does in the Euclidean space?}
\end{center}

In this work, we provide an answer in the affirmative for functions defined on hyperbolic and spherical spaces, up to constants depending on the sectional curvature $\K$ and the initial distance to a minimizer $\RR$, and up to log factors. We summarize our main results in the following:
\begin{itemize}
    \item \textit{Full acceleration}. We design algorithms that provably obtain an $\newtarget{def:accuracy_epsilon}{\epsilon}$ with the same rates of convergence as \AGD{} in the Euclidean space, up to constants and log factors, cf. \cref{table:comparisons:riemannian}. Previous accelerated approaches only showed local results \citep{zhang2018towards} or obtained results with rates in between the ones obtainable by Riemannian Gradient Descent (\newtarget{def:acronym_riemannian_gradient_descent}{\RGD{}}) and \AGD{} \citep{ahn2020nesterov}. Moreover, these previous works only apply to functions that are smooth and strongly g-convex and not to smooth functions that are only g-convex. As a proxy, we design an accelerated algorithm under a condition between convexity and \textit{quasar-convexity} in the constrained setting, of independent interest.

    \item \textit{Reductions}. We present reductions for any Riemannian manifold of bounded sectional curvature. Given an optimization method for smooth and g-convex functions they provide a method for optimizing smooth and strongly g-convex functions, and vice versa. We also reduce global optimization to sequential optimization in constant-diameter Riemannian balls.
\end{itemize}

It is often the case that methods and key geometric inequalities that apply to manifolds with bounded sectional curvatures are obtained from the ones existing for the spaces of constant extremal sectional curvature \citep{grove1997comparison, zhang2016first, zhang2018towards}. Consequently, our contribution is relevant not only because we establish an algorithm achieving global acceleration on functions defined on a manifold other than the Euclidean space, but also because understanding the constant sectional curvature case is an important step towards understanding the more general case of obtaining algorithms that optimize g-convex functions, strongly or not, defined on manifolds of bounded sectional curvature.

\paragraph{Structure of the paper.} We provide some definitions, notations, and related work in the rest of this section. We introduce our algorithms and their ideas and a proof sketch in \cref{sec:algorithm} and we present our reductions in \cref{sec:reductions}. \cref{app:acceleration} contains the convergence proofs of the accelerated algorithms. \cref{app:reductions} contains the proofs of the reductions and the corollaries showing how to apply them to our algorithms. In \cref{app:geometric_results}, we prove our geometric lemmas that show how to reduce our Riemannian problem to the Euclidean non-convex problem that we solve in an accelerated way. In \cref{app:constants} we comment on the constants of our algorithms, on rates of related work and on hardness results. 

\paragraph{Basic Geometric Definitions.} We recall basic definitions of Riemannian geometry that we use in this work. For a thorough introduction we refer to \citep{petersen2006riemannian}. A Riemannian manifold $(\M,\mathfrak{g})$ is a real smooth manifold $\M$ equipped with a metric $\mathfrak{g}$, which is a smoothly varying inner product. For $x \in \M$ and any two vectors $v,w  \in \newtarget{def:tangent_space}{\Tansp{x}}\M$ in the tangent space of $\M$, the inner product $\innp{v,w}_x$ is $\mathfrak{g}(v,w)$.  For $v\in \Tansp{x}\M$, the norm is defined as usual $\norm{v}_x \defi \sqrt{\innp{v,v}_x}$. Typically, $x$ is known given $v$ or $w$, so we will just write $\innp{v,w}$ or $\norm{v}$ if $x$ is clear from context. A geodesic of length $\ell$ is a curve $\gamma : [0,\ell] \to \M$ of unit speed that is locally distance minimizing. A uniquely geodesic space is a space such that for every two points there is one and only one geodesic that joins them. In such a case the exponential map $\newtarget{def:riemannian_exponential_map}{\expon{x}} : \Tansp{x}\M\to \M$ and inverse exponential map $\exponinv{x}:\M\to \Tansp{x}\M$ are well defined for every pair of points, and are as follows. Given $x, y\in\M$, $v\in \Tansp{x}\M$, and a geodesic $\gamma$ of length $\norm{v}$ such that $\gamma(0) =x$, $\gamma(\norm{v})=y$, $\gamma'(0)=v/\norm{v}$, we have that $\expon{x}(v) = y$ and $\exponinv{x}(y) = v$. Note, however, that $\expon{x}(\cdot)$ might not be defined for each $v\in \Tansp{x}\M$. We denote by $\newtarget{def:distance}{\dist}(x,y)$ the distance between $x$ and $y$. Its value is the same as $\norm{\exponinv{x}(y)}$. Given a $2$-dimensional subspace $V \subseteq \Tansp{x}\M$, the sectional curvature at $x$ with respect to $V$ is defined as the classical notion of Gauss curvature, for the surface $\expon{x}(V)$ at $x$. The Gauss curvature at a point $x$ can be defined as the product of the maximum and minimum curvatures of the curves resulting from intersecting the surface with planes normal to the surface at $x$. 

\paragraph{Notation.}\newtarget{sec:notation}{} Let $\newtarget{def:M}{\M}$ be a $\newtarget{def:dimension}{\n}$-dimensional Riemannian manifold. Given two points $x, y\in \M$ and a vector $v\in \Tansp{x}\M$ in the tangent space of $x$, we use the formal notation $\innp{v, y\newtarget{def:formal_riemannian_subtraction}{\riemMinus} x} \defi -\innp{v, x\riemMinus y} \defi \innp{v, \exponinv{x}(y)}$. We call $\newtarget{def:riemannian_function_F}{\F}:\M\to\R$ a function we want to optimize and that has at least one global minimizer $\newtarget{def:global_optimizer}{\xastg}$. We denote by $\newtarget{def:initial_point}{\xInit[\notilde]} \in \M$ an initial point of an optimization algorithm. We use $\newtarget{def:distance_to_global_minimizer}{\Rglobal} \geq \dist(\xInit[\notilde],\xastg)$ as a bound on the initial distance to a global minimizer. We use the notation $\expon{\xInit[\notilde]}(\ball(0,\Rglobal))\subset\M$ to mean that $\M$ is such that $\expon{\xInit[\notilde]}$ is defined on the closed ball $\newtarget{def:closed_ball}{\ball}(0,\Rglobal)\subset \Tansp{\xInit[\notilde]}\M$. We denote $\newtarget{def:radius_of_Riemannian_ball}{\RR}$ the radius of a Riemannian ball $\newtarget{def:riemannian_ball_where_we_optimize}{\BR} \defi \expon{\xInit[\notilde]}(\ball(0, \RR))$ in which \cref{alg:accelerated_gconvex} will optimize. We use $\newtarget{def:M_K}{\Mk}$ to denote any manifold that is a subset of an $\n$-dimensional complete and simply connected manifold of constant sectional curvature $\newtarget{def:curvature}{\K}$, namely a subset of the hyperbolic space or sphere \citep{petersen2006riemannian}, with the inherited metric, and such that $\expon{\xInit[\notilde]}(\ball(0,\Rglobal))\subset\Mk$. 
We want to work with the standard choice of uniquely geodesic manifolds \citep{ahn2020nesterov, liu2017accelerated, zhang2016first, zhang2018towards}. Therefore, if $\K > 0$ we restrict ourselves to $\Rglobal < \pi/(2\sqrt{\K})$ and $\RR < \pi/(2\sqrt{\K})$. In such a case, $\BR$ is uniquely geodesic (it is contained in an open hemisphere). We define $\xast[\notilde] \in \argmin_{x\in\BR} \{\F(x)\}$. For $\newtarget{def:M_Euclidean}{\MEucl}\subseteq\R^{\n}$, we denote by $\newtarget{def:geodesic_map}{\h}:\M\to \MEucl$ a geodesic map \citep{kreyszig1991differential}, which is a diffeomorphism such that the image and the inverse image of a geodesic is a geodesic. For convenience, we map $\h(\xInit[\notilde])=0$. Given a point $x\in\M$ we use the notation $\tilde{x} \defi \h(x)$ and vice versa; any point in $\MEucl$ will use a tilde. Given a vector $v\in \Tansp{x}\M$, we call $\newtarget{}{\tilde{v}} \in \R^{\n}$ the vector of the same norm such that $\{\tilde{x}+\tilde{\lambda} \tilde{v}|\tilde{\lambda}\in\R^+, \tilde{x}+\tilde{\lambda} \tilde{v} \in \MEucl\} = \{\h(\expon{x}(\lambda v))|\lambda\in I\subseteq\R^+\}$, for some interval $I$. Likewise, given $x$ and a vector $\tilde{v}\in\R^{\n}$, we define $v\in \Tansp{x}\M$. In the case of $\Mk$, we call $\newtarget{def:euclidean_ball_where_we_optimize}{\X}=\h(\BR)$. The big-$O$ notation $\newtarget{def:big_o_tilde}{\bigotilde{\cdot}}$ omits $\log$ factors  and considers $\RR\sqrt{\abs{\K}}=\bigo{1}$, since the subroutines of our final \cref{alg:instance_of_riemacon} only need to use values of $\RR$ satisfying this condition. We denote the geometric constants $\newtarget{def:zeta_without_subindex_which_means_zeta_RR}{\zetar} \defi \Rglobal\sqrt{\abs{\K}} \coth(\Rglobal\sqrt{\abs{\K}})\leq   \Rglobal\sqrt{\abs{\K}}+1$ if $\K < 0$ else $1$, and $\newtarget{def:delta_without_subindex_which_means_zeta_RR}{\deltar} \defi \Rglobal\sqrt{\K} \cot(\Rglobal\sqrt{\K})$ if $\K > 0$ else $1$.

We define now the main properties that will be assumed on the function $\F$ to be minimized.
\begin{definition}[Geodesic Convexity and Smoothness] \label{def:g-convex_smooth}
    Let $\F:\M \to \R$ be a differentiable function defined on a Riemannian manifold $(\M,\mathfrak{g})$. Given $\newtarget{def:riemannian_smoothness_of_F}{\L}\geq \newtarget{def:strong_g_convexity_of_F}{\mu} > 0$, we say that $\F$ is $\L$-smooth in $\NN \subseteq \M$, and respectively $\mu$-strongly g-convex, if for any two points $x, y \in \NN$, $\F$ satisfies
    \[
        \F(y) \leq \F(x) + \innp{\nabla \F(x), y\riemMinus x} + \frac{\L}{2}\dist(x,y)^2, \text{ resp. } \F(y) \geq \F(x) + \innp{\nabla \F(x), y\riemMinus x} + \frac{\mu}{2}\dist(x,y)^2.
    \]
    We say $\F$ is g-convex if the second inequality above, i.e., $\mu$-strong g-convexity, is satisfied with $\mu=0$.  We have used the formal notation above for the subtraction of points in the inner product.  
\end{definition}

Our main technique consists of mapping the ball $\BR$ to a subset $\X$ of the Euclidean space via a geodesic map $\h$. Given the gradient of a point $x\in\BR$, convexity defines a lower bound on the function that is affine over the tangent space of $x$, namely $\newtarget{def:affine_lower_bound}{\affLowerBound}(y) = \F(x)+\innp{\nabla \F(x), y\riemMinus x}\leq \F(y)$ and it implies a minimizer must be in the halfspace $H=\{y|\innp{\nabla \F(x), y\riemMinus x} \leq 0\}$, since $\affLowerBound(\cdot)$ is greater than $\F(x)$ outside of $H$. This lower bound induces, via the geodesic map, a function on $\X$. And $H$ is mapped to a halfspace $H'$ in the Euclidean space, because $\{\h(y)|\innp{\nabla \F(x), y\riemMinus x} = 0\}$ is mapped to a hyperplane by the definition of geodesic map. We find a lower bound of $\affLowerBound\circ \h^{-1}$ that is affine over $H'$ and such that it is equal to $\F(x)$ at $\h(x)$, despite the geodesic map being non-conformal, deforming distances, and breaking convexity, cf. \cref{prop:bounding_hyperplane}. This allows to aggregate the lower bounds easily in the Euclidean space by taking an average, in the same spirit as mirror descent algorithms do. We believe that effective lower bound aggregation is key to achieving Riemannian acceleration and optimality and it has been the main hurdle of previous algorithms. Using this strategy, we are able to define a continuous method that we discretize using an approximate implementation of the implicit Euler method, achieving the same rates as the Euclidean \AGD{}, up to constants and log factors, for the optimization of g-convex smooth functions. Our reductions take into account the deformations produced by the geometry to generalize existing optimal Euclidean reductions \citep{allen2016optimal, allen2014linear}. Applying them, we obtain an analogous algorithm for strongly g-convex and smooth functions. Applying them again to the latter they yield an algorithm for g-convex smooth functions with the rates of the same order as the first one, up to geometric constants. We can use this algorithm to implement an approximate ball optimization oracle of radius $\RR$ satisfying $\RR\sqrt{\abs{\K}}=\bigo{1}$. Iterating the application of this oracle, cf. \cref{alg:instance_of_riemacon}, we obtain global acceleration with constants depending on $\Rglobal$ that are much better than if we just used \cref{alg:accelerated_gconvex} with $\RR\gets\Rglobal$.

\newtarget{sec:related_work}{}\paragraph{Comparison with Related Work.} There are a number of works that study the problem of first-order acceleration in Riemannian manifolds of bounded sectional curvature. The first study is \citep{liu2017accelerated}. In this work, the authors develop an accelerated method with the same rates as \AGD{} for both g-convex and strongly g-convex functions, provided that at each step a given non-linear equation can be solved. No algorithm for solving this equation has been found and, in principle, it could be intractable or infeasible. In \citep{alimisis2019continuous} a continuous method analogous to the continuous approach to accelerated methods is presented, but it is not known if there exists an accelerated discretization of it. In \citep{alimisis2020practical}, an algorithm presented is claimed to enjoy an accelerated rate of convergence, but fails to provide convergence when the function value gets below a potentially large constant that depends on the manifold and smoothness constant. The work \citep{lin2020accelerated} is inspired by accelerated algorithms and focuses on adapting to the strong g-convex parameter but does not obtain accelerated algorithms. In \citep{huang2019extending} an accelerated algorithm is presented but relying on strong geometric inequalities that are not proved to be satisfied. \citet{zhang2018towards} obtain a \textit{local} algorithm that optimizes $\L$-smooth and $\mu$-strongly g-convex functions achieving the same rates as \AGD{} in the Euclidean space, up to constants. That is, the initial point needs to start close to the optimum, $O((\mu/\L)^{3/4})$ close, to be precise. Their approach consists of adapting Nesterov's estimate sequence technique by keeping a quadratic on $\Tansp{x_i}\M$ that induces on $\M$ a regularized lower bound on $\F(\xast[\notilde])$ via $\expon{x_i}(\cdot)$. They build another lower bound by aggregating the information yielded by the gradient $\nabla \F(x_i)$ to it, and use a geometric lemma to find a quadratic in $\Tansp{x_{i+1}}\M$ whose induced function lower bounds the previous one. \citet{ahn2020nesterov} generalize the previous algorithm and, by using similar ideas for the lower bound, they adapt it to work globally, obtaining strictly better rates than \RGD{}, recovering the local acceleration of the previous paper, but not achieving global rates comparable to the ones of \AGD{}. In fact, they prove that their algorithm eventually decreases the function value at a rate close to \AGD{} but this can take as many iterations as the ones needed by \RGD{} to reach the neighborhood of the previous local algorithm, cf. \cref{remark:comparison_rates_riemannian}.

In our work, we take a step back and focus on the constant sectional curvature case to provide a global algorithm that achieves the same rates as \AGD{}, up to constants on $\RR$, $\K$, and log factors. 
It is common to characterize the properties of spaces of bounded sectional curvature by using the ones of the spaces of constant extremal sectional curvature \citep{grove1997comparison, zhang2016first, zhang2018towards}, which makes the study of the constant sectional curvature case critical to the development of fully accelerated algorithms in the general bounded sectional curvature case. 
Our work also studies g-convexity besides strong g-convexity. No previous accelerated algorithms applied to this case. Because of the hardness of the geometry, our convergence rates have geometric constants depending on $\RR\sqrt{\abs{\K}}$. The constants of \cref{alg:accelerated_gconvex} are polynomial on $1/\cos(\RR\sqrt{\abs{\K}})$ in spherical spaces and $\cosh(\RR\sqrt{\abs{\K}})$ in hyperbolic spaces. \cref{alg:instance_of_riemacon} is a global and fully accelerated method that uses \cref{alg:accelerated_gconvex} as a subroutine with parameters satisfying $\RR\sqrt{\abs{\K}} = O(1)$. See \cref{table:comparisons:riemannian} for the final convergence rates of our algorithms and previous works. We note that previous works had to assume that their iterates would stay inside of a set defined a priori in order to bound the geometric deformations, while we do not need to make such assumption.

Due to the geometry, there are lower bounds \citep{hamilton2021no, criscitiello2021negative} that say for instance that in the strongly g-convex case, one must query the gradient oracle $\widetilde{\Omega}(\RR)$ times on several negatively curved Riemannian manifolds. This does not preclude achieving a globally accelerated rate, unless this lower bound preponderates over the accelerated dependence on the condition number. See \cref{remark:related_work_hardness_hyperbolic} for a comment on these lower bounds. We showed in \cref{prop:lower_bound_on_condition_number} a lower bound on the condition number of any strongly g-convex function defined on $\BR$. 

\begin{table}
    \centering
    \caption{Worse-case rates of related works for smooth problems with ${\protect\Rglobal} \geq {\protect\dist}({\protect\xInit[{\protect\notilde}]}, {\protect\xastg})$. {\protect\AGD{}} is a Euclidean algorithm. We used ${\protect\kappa} \defi {\protect\L}/{\protect\mu}$. The values $c_1$, $c_2$ are polynomial on ${\protect\ck[{\protect\Rglobal}\sqrt{\abs{{\protect\K}}}]}^{-{\protect\sign}({\protect\K})}$. \cref{alg:instance_of_riemacon} analyzed in \cref{thm:reduction_to_ball_opti} reduces these constants significantly.}
    \label{table:comparisons:riemannian} 
\begin{tabular}{llll} 
    \toprule
    \textbf{Method}   &  \textbf{g-convex} & \textbf{$\mu$-st. g-convex} & \textbf{curv. K?}    \\
    \midrule
    \midrule
    \AGD{} \citep{nesterov1983method}               & $O(\sqrt{\L\Rglobal[2]/\epsilon})$  & $\bigotilde{\sqrt{\kappa}}$ & 0 \\
    \citep{zhang2018towards} (it only works locally)   & -  & $\bigotilde{\sqrt{\kappa}}$ & bounded   \\
    \citep{ahn2020nesterov}                         & - & $\bigo{\kappa\log\kappa} +\bigotilde{\sqrt{\kappa}}$ & bounded \\
    \textbf{\cref{remark:comparison_rates_riemannian} (\RGD{}$ + $\citep{zhang2018towards})}       & - & $\bigo{\kappa\log\kappa} +\bigotilde{\sqrt{\kappa}}$ & bounded\\
    \textbf{\cref{thm:riemannian_acceleration} and \cref{coroll:acceleration_st_g_convex} resp.} & ${\bigotilde{c_1\sqrt{\L\Rglobal[2]/\epsilon}}}$ & $\bigotilde{c_2\sqrt{\kappa}}$ & ctant.$\neq 0$ \\
    \textbf{\cref{thm:reduction_to_ball_opti}} & $\bigotilde{\zetar^{3/2}\deltar^{-1/2}\sqrt{\zetar+\L\Rglobal[2]/\epsilon}}$ & $\bigotilde{\zetar^{3/2}\sqrt{\kappa}}$ & ctant.$\neq 0$ \\
    \bottomrule
\end{tabular}
\end{table}

On Euclidean optimization, a related work is the \textit{approximate duality gap technique} \citep{diakonikolas2017approximate}, which presents a unified view of the analysis of first-order methods. It defines a continuous duality gap and by enforcing a natural invariant, it obtains accelerated continuous dynamics and their discretizations for most classical first-order methods. A derived work \citep{diakonikolas2017accelerated} obtains Euclidean acceleration in a fundamentally different way from previous acceleration approaches, namely using an approximate implicit Euler method for the discretization of the acceleration dynamics. Our convergence analysis of \cref{thm:acceleration_quasiquasarconvexity} draws ideas from these two works. \citet{carmon2020acceleration} initiated the study of optimization with ball oracles, which is an active line of research \citep{carmon2021thinking, asi2021stochastic, carmon2022distributionally} and in turn, this shares some similarities with trust region methods \citep{conn2000trust}. We will see in the sequel that, for our manifolds of interest, g-convexity is related to a model known as quasar-convexity or weak-quasi-convexity \citep{guminov2017accelerated,nesterov2018primal, hinder2019near}.

\section{Algorithms}\label{sec:algorithm}
We study the minimization problem $\min_{x\in \Mk}{\F(x)}$ with a gradient oracle, for a twice differentiable smooth function $\F:\Mk\to\R$ that is g-convex or strongly g-convex, for an initial point $\xInit[\notilde]$.  We recall $\Mk$ refers to any manifold of constant non-zero sectional curvature such that $\BR=\expon{\xInit[\notilde]}(\ball(0, \RR))\subset\Mk$ for some $\RR>0$. We work in this setting in this entire section. A minimizer $\xastg$ of $\F$ that is assumed to exist, possibly outside of $\BR$, and we denote $\Rglobal>\dist(\xInit[\notilde],\xastg)$ a bound on the initial distance to $\xastg$. We perform constrained optimization over $\BR$ and control deformations caused by the geometry. We defer the proofs of the lemmas and theorems in this and following sections to the appendix. We assume without loss of generality that the sectional curvature of $\Mk$ is $\K \in \{1,-1\}$, since for any other value of $\K$ and any function $\F:\Mk\to\R$ defined on such a manifold, we can reparametrize $\F$ by rescaling so that it is defined over a manifold of constant sectional curvature $\K \in \{1,-1\}$. The parameters $\L$, $\mu$ and $\RR$ are rescaled accordingly as a function of $\K$, cf. \cref{remark:rescaling_of_K}. We denote the special cosine by $\newtarget{def:special_cosine}{\ck[\cdot]}$, which is $\cos(\cdot)$ if $\K=1$ and $\cosh(\cdot)$ if $\K=-1$. For a geodesic map $\h:\M\to \MEucl$, we define $\X \defi \h(\BR)\subseteq \MEucl \subseteq \R^{\n}$. We use classical geodesic maps for the manifolds that we consider: the Gnomonic projection for $\K=1$ and the Beltrami-Klein projection for $\K=-1$ \citep{greenberg1993euclidean}. They map an open hemisphere and the hyperbolic space of curvature $\K\in\{1,-1\}$ to $\R^{\n}$ and $B(0,1)\subseteq\R^{\n}$, respectively. We will derive our results from the following characterization of $\h$ \citep{greenberg1993euclidean}. Let $\tilde{x}, \tilde{y} \in \X$ be two points. Recall that we denote $x = \h^{-1}(\tilde{x}), y=\h^{-1}(\tilde{y}) \in \BR$. Then we have that $\dist(x,y)$, the distance between $x$ and $y$ with the metric of $\Mk$, satisfies
\begin{equation}\label{eq:characterization_of_geodesic_map_and_metric}
    \ck[\dist(x,y)]  = \frac{1+\K\innp{\tilde{x}, \tilde{y}}}{\sqrt{1+\K\norm{\tilde{x}}^2}\cdot\sqrt{1+\K\norm{ \tilde{y}}^2}}.
\end{equation}
Observe that the expression is symmetric with respect to rotations. In particular, $\X$ is a closed ball of some radius $\newtarget{def:radius_of_Euclidean_ball}{\Rtilde}$. Using $\tilde{x}=0$ and $\tilde{y}$ such that $\dist(\xInit[\notilde],y)=\RR$, we have $\ck[\RR] = (1+\K\Rtilde^2)^{-1/2}$. 

Consider a point $x\in\BR$ and the lower bound provided by the g-convexity assumption when computing $\nabla \F(x)$. Dropping the $\mu$ term in case of strong g-convexity, this bound is affine over $\Tansp{x} \BR$. In order to define a duality gap, as we show in \cref{subsec:sketch_of_my_axgd_proof}, we would like our algorithm to aggregate effectively the lower bounds it computes during the course of the optimization. The deformations of the geometry make the aggregation a difficult task, despite the fact that we have a simple description of each individual lower bound: each of them is affine over $\Tansp{\xI[i][][]}\BR$ but these simple functions are defined on different tangent spaces. 
We deal with this problem by obtaining a lower bound that is looser by a constant depending on $\RR$, and that is affine over $\X\subset\R^{\n}$. In this way the aggregation becomes easier: all of them are simple and are in the same space. Then, we are able to combine this lower bound with decreasing upper bounds in the fashion some other accelerated methods work in the Euclidean space \citep{nesterov1983method, allen2014linear, diakonikolas2017accelerated,  diakonikolas2017approximate}. Alternatively, we can see the approach in this work as the constrained optimization problem of minimizing the non-convex function $\newtarget{def:tilted_convex_function_f_from_F}{\f}:\X\to\R$, $\tilde{x} \mapsto \F(\h^{-1}(\tilde{x}))$
\[
    \text{minimize}\ \  \f(\tilde{x}),\quad \text{ for } \tilde{x}\in\X.
\]
In the rest of the section, we will focus on the g-convex case. For simplicity, instead of solving the strongly g-convex case directly in an analogous way by finding a lower bound that is quadratic over $\X$, we rely on the reductions of \cref{sec:reductions} to obtain the accelerated algorithm in this case.

The following two lemmas show that finding the affine lower bound is possible, and is defined as a function of $\nabla \f(\tilde{x})$. We first gauge the deformations caused by the geodesic map $\h$. Distances are deformed, the map $\h$ is not conformal, and the image of the geodesic $\expon{x}(\lambda \nabla \F(x))$ is not mapped into the image of the geodesic $\tilde{x}+\tilde{\lambda}\nabla \f(\tilde{x})$, i.e., the direction of the gradient changes. We are able to find the affine lower bound after bounding these deformations.
\begin{lemma}\linktoproof{lemma:deformations}\label{lemma:deformations}
    Let $\K\in\{1,-1\}$. Let $x,y\in\BR$ be two different points, and in part $b)$ different from $\xInit[\notilde]$. Let $\tilde{\alpha}$ be the angle $\angle \xInit\tilde{x}\tilde{y}$, formed by the vectors $\xInit-\tilde{x}$ and $\tilde{y}-\tilde{x}$. Let $\alpha$ be the corresponding angle, the one between the vectors $\exponinv{x}(\xInit[\notilde])$ and $\exponinv{x}(y)$. Assume without loss of generality that $\tilde{x} \in \operatorname{span}\{\hyperlink{def:e_i_canonical_basis}{\color{black}\tilde{e}_1}\}$ and $\nabla \f(\tilde{x}) \in \operatorname{span}\{\hyperlink{def:e_i_canonical_basis}{\color{black}\tilde{e}_1}, \hyperlink{def:e_i_canonical_basis}{\color{black}\tilde{e}_2}\}$ for the canonical orthonormal basis $\{\newtarget{def:e_i_canonical_basis}{\color{black}\hyperlink{def:e_i_canonical_basis}{\color{black}\tilde{e}_i}}\}_{i=1}^{\n}$. Let $e_i\in \Tansp{x}\Mk$ be the unit vector such that $\h$ maps the image of the geodesic $\expon{x}(\lambda e_i)$ to the image of the geodesic $\tilde{x}+\tilde{\lambda}\tilde{e}_i$, for $i=1,\dots, \n$, and $\lambda, \tilde{\lambda} \geq 0$. Then, the following holds.
    \begin{enumerate}[label=\alph*), leftmargin=*,]
        \item Distance deformation:
    \[
        \K \ck[\RR][2] \leq \K \frac{\dist(x,y)}{\norm{\tilde{x}-\tilde{y}}} \leq \K.
    \]
        \item Angle deformation:
    \[
\sin(\alpha) = \sin(\tilde{\alpha}) \sqrt{\frac{1+\K\norm{\tilde{x}}^2}{1+\K\norm{\tilde{x}}^2\sin^2(\tilde{\alpha})}}, \quad \quad \cos(\alpha) = \cos(\tilde{\alpha}) \sqrt{\frac{1}{1+\K\norm{\tilde{x}}^2\sin^2(\tilde{\alpha})}}.
    \]
        \item Gradient deformation:
    \[
    \nabla \F(x) = (1+\K\norm{\tilde{x}}^2)\nabla \f(\tilde{x})_1 e_1 + \sqrt{1+\K\norm{\tilde{x}}^2}\nabla \f(\tilde{x})_2 e_2 \quad \text{ and } \quad e_i \perp e_j \ \text{ for } i\neq j.
    \]
    And if $v \in \Tansp{x}\Mk$ is a vector that is normal to $\nabla \F(x)$, then $\tilde{v}$ is normal to $\nabla \f(x)$. 
    \end{enumerate}
\end{lemma}

The previous lemma shows that $\nabla f(\tilde{x})$ can be easily computed from $\nabla F(x)$. The following lemma uses the deformations described in \cref{lemma:deformations} to obtain the affine lower bound on the function, given a gradient at a point $\tilde{x}$. Note that \cref{lemma:deformations}.c implies that we have $\innp{\nabla \f(\tilde{x}), \tilde{y}-\tilde{x}}= 0$ if and only if $\innp{\nabla \F(x), y\riemMinus x}= 0$. In the proof we lower bound, generally, affine functions defined on $\Tansp{x}\Mk$ by affine functions in the Euclidean space $\X$. This generality allows to obtain a result with constants that only depend on $\RR$. See \cref{remark:constants} for a discussion on these constants.

\begin{lemma}\linktoproof{prop:bounding_hyperplane}\label{prop:bounding_hyperplane}
    Let $\F:\Mk\to\R$ be differentiable and let $\f=\F\circ \h^{-1}$.
    There are constants $\gamman, \gammap \in (0, 1]$ depending on $\RR$ only such that for all $x, y\in\BR$ satisfying $\innp{\nabla \f(\tilde{x}), \tilde{y}-\tilde{x}}\neq 0$ we have: 
    \begin{align}\label{eq:quotient_of_inner_products_is_bounded} 
 \begin{aligned}
     \gammap \leq\frac{\innp{\nabla \F(x), y\riemMinus x}}{\innp{\nabla \f(\tilde{x}), \tilde{y}-\tilde{x}}} \leq \frac{1}{\gamman}.
   \end{aligned}
\end{align}
    In particular, if $\F$ is g-convex we have the following condition, that we call tilted-convexity:
\begin{align} \label{eq:quasiquasarconvexity}
 \begin{aligned}
     \f(\tilde{x}) +  \frac{1}{\gamman}\innp{\nabla \f(\tilde{x}), \tilde{y}-\tilde{x}}\leq \f(\tilde{y})  &\ & & {\text{ if } \innp{\nabla \f(\tilde{x}), \tilde{y}-\tilde{x}}}\leq 0, \\
     \f(\tilde{x})+\gammap\innp{\nabla \f(\tilde{x}), \tilde{y}-\tilde{x}} \leq \f(\tilde{y})  &\ & & {\text{ if } \innp{\nabla \f(\tilde{x}), \tilde{y}-\tilde{x}}} \geq 0.
   \end{aligned}
\end{align}
\end{lemma}

\begin{figure}[h!]
    \centering
    \includegraphics[width=\linewidth]{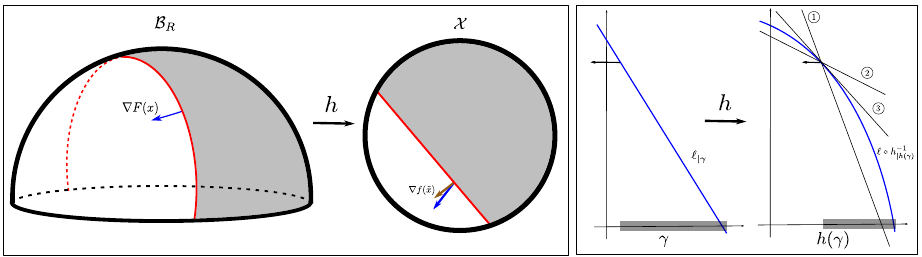}
    \caption{Deformations of the map $\protect\h$.}
    \label{fig:intuition}
\end{figure}

    We provide intuition for the previous lemma through \cref{fig:intuition}. The geodesic that $\nabla \F(x)$ induces on $\BR$ corresponds to the geodesic that the blue vector would induce on $\X$, which is in a different direction than the one induced by $\nabla \f(\tilde{x})$. The angle and gradient deformations of \cref{lemma:deformations} allow to show that, for any direction, inducing a geodesic $\gamma$, the slope of the affine function induced by $\nabla \f(\tilde{x})$ on $\X$ is within a constant factor $c_1$ of the one of the lower bound $\affLowerBound$ defined by $\nabla \F(x)$ in $\BR$. Our main aim is to bound $\F(\xast[\notilde])$ and the shaded area of each image represents where $\xast[\notilde]$ can be. On the right, we exemplify the deformation on a geodesic $\gamma$ passing through $x$. Initially, we have the affine lower bound $\affLowerBound$, but the map $\h$ deforms the domain. To lower bound the function on the shaded region, we can use an affine lower bound $\circled{1}$. Its slope is within a constant factor of the one of the tangent line $\circled{3}$ by the distance deformation of \cref{lemma:deformations} and the factor $c_1$---the latter bounds the change of the directional derivatives, in black. This gives the first line of $\eqref{eq:quasiquasarconvexity}$. The other one is analogous, using another affine function $\circled{2}$.

The first inequality in tilted-convexity shows the affine lower bound, which can be used to bound $\f(\xast)=\F(\xast[\notilde])$. This first inequality, only applied to $\tilde{y}=\xast$ for a function $f :\R^{\n}\to\R$, defines a model known in the literature as quasar-convexity or weak-quasi-convexity \citep{guminov2017accelerated,nesterov2018primal, hinder2019near}, for which accelerated algorithms exist in the \textit{unconstrained case}, provided smoothness is also satisfied. However, to the best of our knowledge, there is no known algorithm for solving the constrained case in an accelerated way. The condition in \eqref{eq:quasiquasarconvexity} is a relaxation of convexity that is stronger than quasar-convexity. We will make use of \eqref{eq:quasiquasarconvexity} in order to obtain acceleration in the constrained setting. This is of independent interest. Recall that we need the constraint to guarantee bounded deformation due to the geometry. We also require gradient Lipschitzness of $\f$, which we obtain in the following lemma.

\begin{lemma}\linktoproof{lemma:smoothness_of_transformed_function}\label{lemma:smoothness_of_transformed_function}
    The function $\f$ has $O(\L(\Rglobal+1))$-Lipschitz gradients in $\X =\h(\BR)$ if our g-convex function $\F:\Mk\to\R$ is $\L$-smooth in $\BR$.
\end{lemma}

Inspired by the \textit{approximate duality gap technique} \citep{diakonikolas2017approximate} we obtain accelerated continuous dynamics, for the optimization of the function $\f$. Then we achieve acceleration by obtaining an implicit Euler discretization of the dynamics. \AGD{} and techniques as Linear Coupling \citep{allen2014linear} or Nesterov's estimate sequence \citep{nesterov1983method} are equivalent to using explicit discretization. These techniques use a balancing gradient step at each iteration to compensate the regret of an implicit or explicit dual algorithm, like mirror descent. Our use of a looser lower bound makes this regret greater by a constant factor and it complicates guaranteeing finding a gradient step within the constraints to compensate this greater regret. Our implicit discretization does not present this problem. We state here the accelerated theorem and provide a sketch of the proof in \cref{subsec:sketch_of_my_axgd_proof}. Note that, for simplicity, we set the Lipschitz condition on $\f$ to be the one yielded by \cref{lemma:smoothness_of_transformed_function} and the bound $\Rglobal > \dist(\xInit, \xastg)$. The algorithm depends logarithmically on this Lipschitz constant.

\begin{theorem}\linktoproof{thm:acceleration_quasiquasarconvexity}\label{thm:acceleration_quasiquasarconvexity}
    Let $\newtarget{def:convex_set_of_tilted_convex_thm}{\Q}\subseteq \R^{\n}$ be a closed convex set of diameter $2\tilde{R}$. Let $\newtarget{def:tilted_convex_function_generic_euclidean}{\ftilted}:\Q\to\R$ be a tilted-convex function with constants $\newtarget{def:gamma_n_parameter}{\gammanparam},\newtarget{def:gamma_p_parameter}{\gammapparam} \in (0,1]$, and $\newtarget{def:generic_smoothness_in_euclidean_case}{\Ltildegeneric}$-Lipschitz gradients. Let $\f$ be $\Ltildegeneric(\Rglobaltilde+2\tilde{R})$ Lipschitz. We can obtain an $\epsilon$-minimizer of $\ftilted$ in $\Q$ by using $\bigotilde{ [\Ltildegeneric \tilde{R}^2/(\gammanparam[2]\gammapparam\epsilon)]^{1/2} }$ queries to the gradient oracle of $\ftilted$.
\end{theorem}

Finally, we show Riemannian acceleration as a consequence of the previous results. Recall that $\bigotilde{\cdot}$ omits constants with respect to $\RR$ since in the sequel we use \cref{alg:accelerated_gconvex}, with $\RR = \bigo{1}$, as a subroutine for our final \cref{alg:instance_of_riemacon}.

\begin{theorem}[\textbf{g-Convex Acceleration}]\linktoproof{thm:riemannian_acceleration}\label{thm:riemannian_acceleration}
    Let $\F:\Mk\to\R$ be an $\L$-smooth and g-convex function, $\Rglobal \geq \dist(\xInit[\notilde],\xastg)$ and let $\RR>0$. \cref{alg:accelerated_gconvex} computes a point $\xI[\T][][]\in\BR \defi \expon{\xInit[\notilde]}(\ball(0,\RR))$ satisfying $\F(\xI[\T][][])-\min_{x\in\BR}\F(x) \leq \epsilon$ using $\bigotilde{\sqrt{\L(\Rglobal+1)/\epsilon}}$ queries to the gradient oracle.
\end{theorem}

\begin{algorithm}[ht!]
    \caption{Global Fully Accelerated g-Convex Minimization}
    \label{alg:accelerated_gconvex}

\begin{algorithmic}[1] 
    \REQUIRE Initial point $\xInit[\notilde] \in\Mk$. Constants $\Ltildegeneric$, $\gammap$, $\gamman$. Geodesic map $\h$ satisfying \eqref{eq:characterization_of_geodesic_map_and_metric} and $\h(\xInit[\notilde])=0$.
    \Statex Smooth and g-convex function $\F:\Mk\to\R$ with a minimizer $\xast[\notilde] \in\BR$.
    \Statex Bound on the distance to a minimum $\RR \geq \dist(\xInit[\notilde],\xast[\notilde])$. Accuracy $\epsilon$ and number of iterations $\T$.
    \vspace{0.1cm}
    \hrule
    \vspace{0.1cm}
    \State $\X \defi \h(\BR)\subseteq \MEucl$; $\quad \f \defi \F \circ \h^{-1}\quad$ and $\quad\psi(\tilde{x}) \defi \frac{1}{2}\norm{\tilde{x}}^2$ 
    \State $\zi[0] \gets \nabla \psi(\xInit)$; $\quad \At[0] \gets 0$
    \FOR {$i = 0 \textbf{ to } \T-1$}
        \State $\ai[i+1] \gets (i+1)\gamman[2]\gammap/2\Ltildegeneric$
        \State $\At[i+1] \gets \At[i] + \ai[i+1]$
        \State $\lambda \gets \text{BinaryLineSearch}(\xI[i], \zi[i], \f, \X, \ai[i+1], \At[i], \epsilon, \Ltildegeneric,  \gamman, \gammap)$ \Comment{(cf. \cref{alg:bin_search} in \cref{app:acceleration})}
        \State $\chii[i] \gets (1-\lambda)\xI[i] + \lambda \nabla \psi^{\astfenchel}(\zi[i])$ 
        \State $\zetai[i] \gets \zi[i]-(\ai[i+1]/\gamman)\nabla \f(\chii[i])$
        \State $\xI[i+1] \gets (1-\lambda)\xI[i] + \lambda \nabla \psi^{\astfenchel}(\zetai[i])$  \Comment{$\big[\nabla \psi^{\astfenchel}(\tilde{p})= \argmin_{\tilde{z} \in \X} \{\norm{\tilde{z}-\tilde{p}}\} = \Pi_{\X}(\tilde{p})\big]$}
        \State $\zi[i+1] \gets \zi[i]-(\ai[i+1]/\gamman)\nabla \f(\xI[i+1])$
    \ENDFOR
    \State \textbf{return} $\xI[\T][][]$.
\end{algorithmic}
\end{algorithm}

We provide a sketch of the main optimization theorem in the section below. The full proof can be found in \cref{app:acceleration}. Our use of geodesic maps was a choice we used to be able to aggregate lower bounds. Our method showcases that an effective lower bound aggregation makes possible to achieve global full acceleration. It suggests that acceleration could also be achieved for functions defined on other manifolds by using our accelerated techniques if we can effectively aggregate the lower bounds yielded by the gradient at each iteration to build a lower bound on $\F(\xast[\notilde])$, similarly as in \eqref{eq:raw_lower_bound} below. We observe that if there is a geodesic map mapping a manifold into a convex subset of the Euclidean space then the manifold must necessarily have constant sectional curvature, cf. Beltrami's Theorem \citep{busemann1984general, kreyszig1991differential}. 

This means that lower bound aggregation in other manifolds would need to use a different kind of transformations.
The field of comparison geometry allows to obtain properties of spaces of bounded sectional curvature by using the properties of the spaces that have constant curvature equal to the bounds of the former  \citep{grove1997comparison}. Other Riemannian optimization algorithms have used comparison theorems that allow to obtain convergence bounds after computing the maximum possible deformations in spaces of extremal constant sectional curvature and relating them to the spaces of bounded sectional curvature \citep{zhang2016first, zhang2018towards}. The generalization to functions defined on manifolds of bounded sectional curvature is a future direction of research.  

\cref{alg:accelerated_gconvex}, that was yielded by this technique, presents constants that can be undesirable when $\RR$ is not $O(1)$. However, we show in the following theorem that by sequentially applying an approximate ball optimization oracle for balls of radius $\RR = O(1)$, we obtain global acceleration with greatly reduced constants, and we can implement the ball optimization oracle by using \cref{alg:accelerated_gconvex}, effectively boosting its convergence.
\begin{theorem}\label{thm:reduction_to_ball_opti}\linktoproof{thm:reduction_to_ball_opti}
    Let $\F:\Mk\to\R$ have a minimizer $\xastg$ and let $\Rglobal \geq \dist(x_0, \xastg)$. If $\K>0$ assume\footnote{By solving the ball subproblems more accurately and in smaller balls, one can relax this condition to $c\Rglobal + \RR < \frac{\pi}{2\sqrt{\K}}$ for any constants $c > 1, \RR > 0$.} $2\Rglobal+\RR < \frac{\pi}{2\sqrt{\K}}$. Let $\F$ be $\L$-smooth and $\mu$-strongly g-convex in $\ball(0,2\Rglobal+\RR)$. 
    \cref{alg:instance_of_riemacon} obtains an $\epsilon$-minimizer after $\bigotilde{\zetar^{3/2} \sqrt{L/\mu}}$ calls to the gradient oracle of $f$. By means of regularization, we obtain an algorithm for the $L$-smooth g-convex case with rates $\bigotilde{\zetar^{3/2} \deltar^{-1/2}\sqrt{\zetar + L\Rglobal[2]/\epsilon}}$.
\end{theorem}

\begin{algorithm}[ht!]
    \caption{Ball Optimization Boosting}
    \label{alg:instance_of_riemacon}

\begin{algorithmic}[1] 
    \REQUIRE Differentiable function $\F\subset\Mk\to\R$ that is $\L $-smooth and $\mu$-strongly g-convex in $\ball(\xastg, 3\Rglobal)\subset \Mk$; initial point $\xInit[\notilde] \in \Mk$; bound $\Rglobal \geq \dist(\xInit[\notilde], \xast[\notilde])$; accuracy $\epsilon$.
    \vspace{0.1cm}
    \hrule
    \vspace{0.1cm}
    \State $\RR \gets 1$; $\pk[0] \gets \xInit[\notilde]$
    \State \textbf{if} $\Rglobal \leq \RR$ \textbf{then return} $\operatorname{AlgSC}(\ball(\pk[0], \Rglobal), \pk[0], \F, \epsilon)$ \label{line:if_D_equal_R_one_single_Riemacon} \label{line:first_if_in_instanced_alg}
    \State $\newtarget{def:total_number_of_iterations_ball_boosting}{\TT} \gets \lceil \frac{2\Rglobal}{\RR} \ln(\frac{\L\Rglobal[2]}{\epsilon}) \rceil$; \ \  $\newtarget{def:accuracy_epsilon_prime_in_ball_subproblem}{\epsilonp} \gets \min\{\frac{\RR\epsilon}{4\Rglobal},\frac{\mu\Rglobal[2]}{2\TT^2}\}$
    \FOR {$k = 1 \textbf{ to } \TT$}
       \State $\X_k \gets \ball(\pk[k-1], \RR)$
       \State $\newtarget{def:iterates_p_boosting_alg}{\pk} \gets \algsc(\X_k, \pk[k-1], \F, \epsilonp)$ \label{line:alg_sc_as_subroutine}
       \State $\diamond$ $\algsc$ is the strongly convex version of \cref{alg:accelerated_gconvex} in \cref{coroll:acceleration_st_g_convex}.
    \ENDFOR
    \State \textbf{return} $x_{\TT}$.
\end{algorithmic}
\end{algorithm}

\subsection{Proof sketch of \texorpdfstring{\cref{thm:acceleration_quasiquasarconvexity}}{Theorem \ref{thm:acceleration_quasiquasarconvexity}} } \label{subsec:sketch_of_my_axgd_proof}

We let $\alphat[t]$ be an increasing function of time $t$, and denote $\At[t]=\int_{t_0}^t d\alphat[\tau]= \int_{t_0}^t \dott{\alphat}_\tau d\tau$. We define a continuous method that keeps a solution $\xt[t]$, along with a differentiable upper bound $\Ut[t]$ on $\ftilted(\xt[t])$ and a lower bound $\Lt[t]$ on $\ftilted(\xast)$. In our case $\ftilted$ is differentiable so we can just take $\Ut[t]=\ftilted(\xt[t])$. The lower bound comes from 
\begin{equation}\label{eq:raw_lower_bound}
    \ftilted(\xast) \geq \frac{\int_{t_{0}}^{t} \ftilted(\xt[\tau]) d \alphat[\tau]}{\At[t]}+\frac{\int_{t_{0}}^{t}\frac{1}{\gammanparam}\innp{\nabla \ftilted(\xt[\tau]), \xast-\xt[\tau]} d \alphat[\tau]}{\At[t]},
\end{equation}
after adding and subtracting a regularizer $\psi$, which is a $1$-strongly convex function, and after removing the unknown $\xast$ by taking a minimum over $\X$.

Note \eqref{eq:raw_lower_bound} comes from averaging \eqref{eq:quasiquasarconvexity} for $\tilde{y} = \xast$. Then, if we define the gap $\Gt[t]=\Ut[t]-\Lt[t]$ and design a method that forces $\alphat[t]\Gt[t]$ to be non-increasing, we can deduce $\ftilted(\xt[t][][])-\ftilted(\xast) \leq \Gt[t] \leq \alphat[t_0]\Gt[t_0]/\alphat[t]$. By forcing $\frac{d}{dt}(\alphat[t]\Gt[t]) =0$, we naturally obtain the following continuous dynamics, where $z_t$ is a mirror point and $\psi^{\astfenchel}$ is the Fenchel dual of $\psi$.
\begin{align}\label{modified_accelerated_continuous_dynamics}
\begin{aligned}
    \dott{\zt[]}_{t}&=-\frac{1}{\gammanparam}\dott{\alphat}_t \nabla \ftilted(\xt[t]); \ \ \
     \dott{\xt[]}_{t}&=\frac{1}{\gammanparam}\dott{\alphat}_t \frac{\nabla \psi^{\astfenchel}(\zt[t])-\xt[t]}{\alphat[t]}; \ \ \
    \zt[t_0]&=\nabla \psi^{\astfenchel}(\xt[t_0]), \xt[t_0] \in \X.
\end{aligned}
\end{align}
We note that except for the constant $\gammanparam$, these dynamics match the accelerated dynamics used in the optimization of convex functions \citep{DBLP:conf/nips/KricheneBB15, diakonikolas2017accelerated, diakonikolas2017approximate}. The AXGD algorithm \citep{diakonikolas2017accelerated}, designed for the accelerated optimization of convex functions, discretizes the dynamics coming from the optimization of convex functions by using an approximate implementation of implicit Euler discretization. This has the advantage of not needing a gradient step per iteration to compensate for some positive discretization error. In our case, the extra error we incur by using a looser lower bound does not seem to be able to be compensated by a gradient step in the constrained case, as other acceleration techniques like Linear Coupling \citep{allen2014linear} or Nesterov's estimate sequence \citep{nesterov1983method} do, so obtaining an approximate implicit Euler discretization proves to be a better approach. However, our dynamics are different and in our case we must use tilted-convexity \eqref{eq:quasiquasarconvexity} instead of convexity. We are able to obtain the following discretization coming from an approximate implicit Euler discretization: 
\begin{equation}\label{general_rule_modified_quasar_axgd}
\left\{\begin{array}{l}
    \chii[i]=\frac{\hatgamma[i] \At[i]}{\At[i] \hatgamma[i] +\ai[i+1]/\gammanparam} \xI[i]+\frac{\ai[i+1]/\gammanparam}{\At[i]\hatgamma[i]  + \ai[i+1]/\gammanparam} \nabla \psi^{\astfenchel}(\zi[i]) \ \\
    \zetai[i] = \zi[i]-\frac{\ai[i+1]}{\gammanparam} \nabla \ftilted(\chii[i]) \\
    \xI[i+1]=\frac{\hatgamma[i] \At[i]}{\At[i] \hatgamma[i] +\ai[i+1]/\gammanparam} \xI[i]+\frac{\ai[i+1]/\gammanparam}{\At[i]\hatgamma[i]  + \ai[i+1]/\gammanparam} \nabla \psi^{\astfenchel}(\zetai[i]) \\ 
    \zi[i+1] = \zi[i]-\frac{\ai[i+1]}{\gammanparam} \nabla \ftilted(\xI[i+1]) 
\end{array}\right.
\end{equation}
    where $\hatgamma[i] \in[\gammapparam, 1/\gammanparam]$ is a parameter, $\xInit\in\X$ is an arbitrary point, $\zi[0]=\nabla \psi(\xInit)$ and now $\alphat[t]$ is a discrete measure and $\dott{\alphat}_t$ is a weighted sum of Dirac delta $\bar{\delta}$ functions $\dott{\alphat}_t=\sum_{i=1}^\infty \ai[i] \bar{\delta}(t-(t_0+i-1))$. However, not having convexity, in order to have per-iteration discretization error less than $\hat{\epsilon}/\At[T]$, we require $\hatgamma[i] $ to be such that $\xI[i+1]$ satisfies
\begin{equation}\label{eq:approximate_binary_search_main_paper}
    \ftilted(\xI[i+1])-\ftilted(\xI[i]) \leq \hatgamma[i]  \innp{\nabla \ftilted(\xI[i+1]), \xI[i+1]-\xI[i]} + \hat{\oldepsilon}, 
\end{equation}
where $\hat{\oldepsilon}$ is chosen so that the accumulated discretization error is $<\epsilon/2$, after having performed the steps necessary to obtain an $\epsilon/2$ minimizer. We would like to use \eqref{eq:quasiquasarconvexity} to find such a $\hatgamma[i] $ but we need to take into account that we only know $\xI[i+1]$ a posteriori. Indeed, using \eqref{eq:quasiquasarconvexity} we conclude that setting $\hatgamma[i] $ to $1/\gammanparam$ or $\gammapparam$ then we either satisfy \eqref{eq:approximate_binary_search_main_paper} or there is a point $\hatgamma[i]  \in (\gammapparam, 1/\gammanparam)$ for which $\innp{\nabla \ftilted(\xI[i+1]), \xI[i+1]-\xI[i]}=0$, which satisfies the inequality for $\hat{\oldepsilon}=0$. Then, using gradient Lipschitzness of $\ftilted$, $\Rglobal \geq \dist(\xInit[\notilde], \xastg)$, and boundedness of $\X$ we can guarantee that a binary search finds a point satisfying \eqref{eq:approximate_binary_search_main_paper} in $O(\log (\Ltildegeneric \Rglobal i/\gamma_n\hat{\oldepsilon}))$ iterations. Each iteration of the binary search requires to run \eqref{general_rule_modified_quasar_axgd}, that is, one step of the discretization. Computing the final discretization error, we obtain acceleration after choosing appropriate learning rates $\ai[i]$. \cref{alg:accelerated_gconvex} contains the pseudocode of this algorithm along with the reduction of the problem from minimizing $\F$ to minimizing $\f$. We chose $\psi(\tilde{x}) \defi \frac{1}{2}\norm{\tilde{x}}^2$ as our strongly convex regularizer.

\section{Reductions} \label{sec:reductions}
The construction of reductions proves to be very useful in order to facilitate the design of algorithms in different settings. Moreover, reductions are a helpful tool to infer new lower bounds without extra ad hoc analyses. We present two reductions. We will see in \cref{coroll:acceleration_st_g_convex} and \cref{example:application_of_reduction_to_st_convex} that one can obtain full accelerated methods to minimize smooth and strongly g-convex functions from our accelerated methods for smooth and g-convex functions and vice versa. These are generalizations of some reductions designed to work in the Euclidean space \citep{allen2016optimal, allen2014linear}. The reduction to strongly g-convex functions takes into account the effect of the deformation of the space on the strong convexity of the function $F_y(x) = \dist(x,y)^2/2$, for $x, y \in \M$. It does not entail an extra $\log(1/\epsilon)$ factor. The reduction to g-convexity requires the rate of the algorithm that applies to g-convex functions to be proportional to the squared distance between the initial point and the optimum $\dist(\xInit[\notilde],\xast[\notilde])^2$ or to a bound $\RR^2$. The proofs of the statements in this section can be found in the appendix. We will use $\newtarget{def:time_alg_of_just_convex_for_reduction}{\timens(\cdot)}$ and $\newtarget{def:time_alg_of_strongly_convex_for_reduction}{\time}(\cdot)$ to denote the time algorithms $\Algns$ and $\Alg$ below require, respectively, to perform the tasks we define below.

\begin{theorem}\linktoproof{thm:reduction_to_g_convex}\label{thm:reduction_to_g_convex}
    Let $\M$ be a Riemannian manifold, let $\F:\M\to\R$ be an $\L$-smooth and $\mu$-strongly g-convex function, with a minimizer $\xast[\notilde]$. Suppose we have an algorithm $\newtarget{def:alg_of_just_convex_for_reduction}{\Algns}$ to minimize $\F$, such that for any starting point $\xInit[\notilde]$ such that  $\dist(\xInit[\notilde],\xast[\notilde]) \leq \RR$, it produces a point $\hat{x}_T$ in time $T =\timens(\L, \mu, \RR)$ satisfying $\F(\hat{x}_T)-\F(\xast[\notilde]) \leq \mu \RR^2/4$. Then we can compute an $\epsilon$-minimizer of $\F$ in time $O(\timens(\L, \mu, \RR)\log(\mu/\epsilon))$.
\end{theorem}

\cref{thm:reduction_to_g_convex} implies that if we forget about the strong g-convexity of a function and we treat it as if it is just g-convex then we can run in stages an algorithm designed for optimizing g-convex functions and still achieve acceleration.  The fact that the function is strongly g-convex is only used between stages. We exemplify the power of the reduction by applying it to \cref{alg:accelerated_gconvex}.

\begin{corollary}\linktoproof{coroll:acceleration_st_g_convex}\label{coroll:acceleration_st_g_convex}
    Using \cref{alg:accelerated_gconvex}, we can compute an $\epsilon$-minimizer of a function $\F:\Mk\to\R$ that is $\L$-smooth and $\mu$-strongly g-convex in $\BR$ by using $\bigotilde{\sqrt{\L(\Rglobal+1)/\mu}}$ queries to the gradient oracle.
\end{corollary}

We note that in the strongly convex case, by decreasing the function value by a factor we can guarantee we decrease the distance to $\xast[\notilde]$ by another factor, so we can periodically recenter the geodesic map to reduce the constants produced by the deformations of the geometry, see the proof of \cref{coroll:acceleration_st_g_convex}. Finally, we show the reverse reduction.

\begin{theorem}[\textbf{simplified, cf. \cref{thm:reduction_to_g_st_convex_full_theorem}}]\label{thm:reduction_to_g_st_convex}
    Let $\F:\M\to\R$ be $\L$-smooth and g-convex, and let $\M$ be of bounded sectional curvature. Let $\Delta$ satisfy $\F(\xInit[\notilde])-\F(\xast[\notilde]) \leq \Delta $. Let $\newtarget{def:alg_of_strongly_convex_for_reduction}{\Alg}$ be an algorithm that in time $\hat{T}=\time(\L,\mu,\M, \RR)$ produces $\hat{x}\in\expon{\xInit[\notilde]}(\ball(0, \RR))$ that reduces the gap of an $\L$-smooth and $\mu$-strongly g-convex function $\hat{F}:\M\to\R$, with minimizer in $\expon{\xInit[\notilde]}(\ball(0,\RR))$, by a factor of $1/4$, i.e., $\hat{F}(\hat{x})-\min_{x\in\M}\hat{F}(x) \leq (\hat{F}(\hat{x}_0) - \min_{x\in\M}\hat{F}(x))/4$. Let $T=\lceil\log_2(\Delta/\epsilon)\rceil+1$. Then, we can compute an $\epsilon$-minimizer in time $\sum_{t=0}^{T-1}\time(\L+O(2^{-t}\Delta), O(2^{-t}\Delta), \M, \RR)$.
\end{theorem}

\begin{example}\linktoproof{example:application_of_reduction_to_st_convex}\label{example:application_of_reduction_to_st_convex}
    \normalfont{
    Applying \cref{thm:reduction_to_g_st_convex} to the algorithm in \cref{coroll:acceleration_st_g_convex} we can optimize $\L$-smooth and g-convex functions defined on $\Mk$ with a gradient oracle complexity of $\bigotilde{\sqrt{\L(\Rglobal+1)/\epsilon}}$. 
    }
\end{example}

Note that this reduction cannot be applied to the locally accelerated algorithm in \citep{zhang2018towards}, that we discussed in the related work section. The reduction runs in stages by regularizing each time with a strongly g-convex regularizer whose parameter decreases exponentially until we use a regularizer with $O(\epsilon)$ maximum function value. The local assumption required by the algorithm in \citep{zhang2018towards} on the closeness to the minimum cannot be guaranteed. In \citep{ahn2020nesterov}, the authors give an unconstrained global algorithm whose rates are strictly better than \RGD{}. The reduction could be applied to a constrained version of this algorithm to obtain a method for smooth and g-convex functions defined on manifolds of bounded sectional curvature and whose rates are strictly better than \RGD{}.

\section{Conclusion}

In this work, we proposed an algorithm with the same rates as \AGD{}, for the optimization of smooth and strongly g-convex functions, up to constants and log factors, while previous approaches essentially only reached this for a ball around the minimizer of radius $O((\mu/\L)^{3/4})$. Our algorithm also applies to g-convex functions while previous accelerated algorithms did not apply. We focused on hyperbolic and spherical spaces, that have constant sectional curvature. The study of geometric properties for this is often employed to conclude that a space of bounded sectional curvature satisfies a property that is in between the ones for the cases of constant extremal sectional curvature. Several previous algorithms have been developed for the general case by utilizing this philosophy, for instance \citep{ahn2020nesterov,ferreira2019gradient,DBLP:journals/jota/WangLY15,zhang2016first,zhang2018towards}. In future work, we will attempt to use the techniques and insights developed in this work to give an algorithm with the same rates as \AGD{} for manifolds of bounded sectional curvature.

The key technique of our algorithm is the effective lower bound aggregation. Indeed, lower bound aggregation is the main hurdle to obtain accelerated first-order methods defined on Riemannian manifolds. Whereas the process of obtaining decreasing upper bounds on the function works similarly as in the Euclidean space---the same approach of locally minimizing the upper bound given by the smoothness assumption is used---obtaining adequate lower bounds proves to be a difficult task. We usually want a simple lower bound such that it, or a regularized version of it, can be easily optimized globally. We also want that the lower bound combines the knowledge that the g-convexity or strong g-convexity provides for all the queried points, commonly an average. These Riemannian convexity assumptions provide simple lower bounds, namely linear or quadratic, but each with respect to each of the tangent spaces of the queried points only. The deformations of the space complicate the aggregation of the lower bounds. Our work deals with this problem by finding appropriate lower bounds via the use of a geodesic map and takes into account the deformations incurred to derive a fully accelerated algorithm. We also used other tools for designing the accelerated algorithm. We worked with a relaxation of convexity that allowed to perform a binary search to reduce the discretization error. We had to use an implicit discretization of some accelerated continuous dynamics, since at least the vanilla application of usual approaches like Linear Coupling \citep{allen2014linear} or Nesterov's estimate sequence \citep{nesterov1983method}, that can be seen as a forward Euler discretization of the accelerated dynamics combined with a balancing gradient step \citep{diakonikolas2017approximate}, did not work in our constrained case. We interpret that the difficulty arises from trying to keep the gradient step inside the constraints while being able to compensate for a lower bound that is looser by a constant factor. Moreover, the use a ball optimization oracle proves to be a very useful tool in order to reduce geometric penalties in the convergence rates.

\acks{We thank Mario Lezcano-Casado for helpful discussions on this work. We thank Varun Kanade and Patrick Rebeschini for proofreading of this work. This work was supported by EP/N509711/1 from the EPSRC MPLS division, grant No 2053152.}

\printbibliography[heading=bibintoc] 

\clearpage

\appendix

We divide the appendix in four sections. \cref{app:acceleration} contains the proofs related to the accelerated algorithms, i.e., the proofs of Theorems \ref{thm:acceleration_quasiquasarconvexity}, \ref{thm:riemannian_acceleration}, and \ref{thm:reduction_to_ball_opti}. In \cref{app:reductions} we prove the results related to the reductions in \cref{sec:reductions}. In \cref{app:geometric_results}, we prove the geometric lemmas that take into account the geodesic map $\h$ to obtain relationships between $\F$ and $\f$, namely Lemmas \ref{lemma:deformations}, \ref{prop:bounding_hyperplane} and \ref{lemma:smoothness_of_transformed_function}. Finally, \cref{app:constants} contains a discussion on the constants of our algorithms, on rates of related work and on hardness results. We also show a lower bound on the condition number for any strongly g-convex function defined on $\BR$.

\section{Acceleration. Proofs of Theorems \ref{thm:acceleration_quasiquasarconvexity}, \ref{thm:riemannian_acceleration} and \ref{thm:reduction_to_ball_opti}} \label{app:acceleration}
\citet{diakonikolas2017approximate} developed the \textit{approximate duality gap technique} which is a technique that provides a structure to design and prove first order methods and their guarantees for the optimization of convex problems. We take inspiration from these ideas to apply them to the non-convex problem we have at hand \cref{thm:acceleration_quasiquasarconvexity}, as it was sketched in \cref{subsec:sketch_of_my_axgd_proof}. We start with two basic definitions.
\begin{definition}\label{def:bregman_divergence}
    Given two points $\tilde{x}, \tilde{y}$, we define the Bregman divergence with respect to $\psi(\cdot)$ as
    \[
        \newtarget{def:bregman_divergence}{}{\breg[\psi]} (\tilde{x}, \tilde{y}) \defi \psi(\tilde{x})-\psi(\tilde{y}) - \innp{\nabla \psi(\tilde{y}), \tilde{x}-\tilde{y}}.
    \]
\end{definition}

\begin{definition} \label{def:fenchel_dual}
    Given a closed convex set $\Q$ and a function $\psi:\Q \to \R$, we define the convex conjugate of $\psi$, also known as its Fenchel dual, as the function
    \[
        \psi^{\newtarget{def:asterisk_of_fenchel_dual}{\astfenchel}}(\tilde{z}) = \max_{\tilde{x}\in \Q}\{ \innp{\tilde{z}, \tilde{x}}-\psi(\tilde{x})\}.
    \]
\end{definition}
For simplicity, we will use $\newtarget{def:strongly_convex_regularizer}{\psi}(\tilde{x}) = \frac{1}{2}\norm{\tilde{x}}^2 + i_Q(\tilde{x})$ in \cref{alg:accelerated_gconvex}, but any strongly convex map works. Here $i_Q(x)=0$ if $x \in Q$ and $i_Q(x)=+\infty$ otherwise. The gradient of the Fenchel dual of $\psi(\cdot)$ is $\nabla\psi^{\astfenchel}(\tilde{z}) = \argmin_{\tilde{z}'\in \Q}\{\norm{\tilde{z}'-\tilde{z}}\}$, that is, the Euclidean projection $\newtarget{def:euclidean_projection}{\Pi}_{\Q}(\tilde{z})$ of the point $\tilde{z}$ onto $\Q$. Note that when we apply \cref{thm:acceleration_quasiquasarconvexity} to \cref{thm:riemannian_acceleration} our constraint $\Q$ will be $\X$, that is, a ball centered at $0$ of radius $\Rtilde$, so the projection of a point $\tilde{z}$ outside of $\X$ will be the vector normalization $\Rtilde\tilde{z}/\norm{\tilde{z}}$. Any continuously differentiable strongly convex $\psi$ would work, provided that $\nabla \psi^{\astfenchel}(z)$ is easily computable, preferably in closed form. Note that by the Fenchel-Moreau theorem we have for any such map that $\psi^{\astfenchel\astfenchel} =\psi$.

We recall we assume that $\ftilted$ satisfies tilted-convexity \eqref{eq:quasiquasarconvexity}: 
\begin{align*}
 \begin{aligned}
     \ftilted(\tilde{x}) +  \frac{1}{\gammanparam}\innp{\nabla \ftilted(\tilde{x}), \tilde{y}-\tilde{x}}\leq \ftilted(\tilde{y})  &\ & & {\text{ if } \innp{\nabla \ftilted(\tilde{x}), \tilde{y}-\tilde{x}}}\leq 0, \\
     \ftilted(\tilde{x})+\gammapparam\innp{\nabla \ftilted(\tilde{x}), \tilde{y}-\tilde{x}} \leq \ftilted(\tilde{y})  &\ & & {\text{ if } \innp{\nabla \ftilted(\tilde{x}), \tilde{y}-\tilde{x}}} \geq 0.
   \end{aligned}
\end{align*}

Let $\newtarget{def:step_function_alpha}{\alphat[t]}$ be an increasing function of time $t$. We use Lebesgue-Stieltjes integration and its notation, so that $\int_{0}^t \ftilted(x_\tau)\dott{\alphat}_\tau d\tau =\int_{0}^t \ftilted(x_\tau)d\alphat[\tau]$. We want to work with continuous and discrete approaches in a unified way. Thus, when $\alphat[t]$ is a discrete measure, we have that $\dott{\alphat}_t = \sum_{i=1}^\infty \newtarget{def:discrete_step}{\ai[i]} \bar{\delta}(t-(t_0+i-1))$ is a weighted sum of Dirac delta $\bar{\delta}$ functions. We define $\newtarget{def:integral_of_steps}{\At[t]} \defi \int_{t_0}^t d\alphat[\tau] = \int_{t_0}^t \dott{\alphat}_\tau d\tau$. In discrete time, it is $\At[t] = \sum_{i=1}^{\lfloor t-t_0+1\rfloor} \ai[i] = \alphat[t]$.  In the continuous case note that we have $\alphat[t] -\At[t] = \alphat[t_0]$.

We start defining a continuous method that we discretize with an approximate implementation of the implicit Euler method. Let $\xt[t]$ be the solution obtained by the algorithm at time $t$. We define the duality gap $\newtarget{def:differentiable_gap_bound}{\Gt[t]}\defi \Ut[t]-\Lt[t]$ as the difference between a differentiable upper bound $\Ut[t]$ on the function at the current point and a lower bound on $\ftilted(\xast)$. Since in our case $\ftilted$ is differentiable we use $\newtarget{def:differentiable_upper_bound}{\Ut[t]} \defi \ftilted(\xt[t])$. The idea is to enforce the invariant $\frac{d}{dt}(\alphat[t] \Gt[t]) = 0$, so we have at any time $\ftilted(\xt[t])-\ftilted(\xast)\leq \Gt[t] = \Gt[t_0]\alphat[t_0]/\alphat[t]$. 

Note that for the constrained minimizer $\xast$ of $\ftilted$ and any other point $\tilde{x}\in \Q$, we have $\innp{\nabla \ftilted(\tilde{x}), \xast-\tilde{x} }\leq 0$. Otherwise, we would obtain a contradiction since by tilted-convexity \eqref{eq:quasiquasarconvexity} we would have 
\[
f(\tilde{x}) <\ftilted(\tilde{x})+ \gammapparam\innp{\nabla \ftilted(\tilde{x}), \xast-\tilde{x}} \leq \ftilted(\xast).
\] 
Therefore, in order to define an appropriate lower bound, we will make use of the inequality $\ftilted(\xast) \geq \ftilted(\tilde{x}) + \frac{1}{\gammanparam}\innp{\nabla \ftilted(\tilde{x}), \xast -\tilde{x}}$, for any $\tilde{x}\in  \Q$, which holds true by tilted-convexity \eqref{eq:quasiquasarconvexity}, for $\tilde{y}=\xast$. Combining this inequality for all the points visited by the continuous method we have
\[
    \ftilted(\xast) \geq \frac{\int_{t_{0}}^{t} \ftilted(\xt[\tau]) d \alphat[\tau]}{\At[t]}+\frac{\int_{t_{0}}^{t}\frac{1}{\gammanparam}\innp{\nabla \ftilted(\xt[\tau]), \xast-\xt[\tau]} d \alphat[\tau]}{\At[t]}.
\]
We cannot compute this lower bound, since the right hand side depends on the unknown point $\xast$. We could compute a looser lower bound by taking the minimum over $\tilde{u}\in \Q$ of this expression, substituting $\xast$ by $\tilde{u}$. However, this would make the lower bound be non-differentiable and we could have problems at $t_0$. In order to solve the first problem, we first add a regularizer and then take the minimum over $\tilde{u}\in  \Q$.
\begin{align*} 
\begin{aligned}
f(\xast)+& \frac{\breg[\psi](\xast, \xt[t_0])}{\At[t]} \\
    & \geq \frac{\int_{t_{0}}^{t} \ftilted(\xt[\tau]) d \alphat[\tau]}{\At[t]}+\frac{\min_{\tilde{u} \in  \Q}\left\{\int_{t_{0}}^{t}\frac{1}{\gammanparam}\innp{\nabla \ftilted(\xt[\tau]), \tilde{u}-\xt[\tau] } d \alphat[\tau]+\breg[\psi](\tilde{u}, \xt[t_0])\right\}}{\At[t]}
\end{aligned}
\end{align*}
In order to solve the second problem, we mix this lower bound with the optimal lower bound $\ftilted(\xast)$ with weight $\alphat[t] -\At[t]$ (this is only necessary in continuous time, in discrete time this term is $0$). Not knowing $\ftilted(\xast)$ or $\breg[\psi](\xast, \xt[t_0])$ will not be problematic. Indeed, we only need to guarantee $\frac{d}{dt}(\alphat[t] \Gt[t])=0$. After taking the derivative, these terms will vanish. After rescaling the normalization factor, we finally obtain the lower bound
\begin{align} \label{lower_bound_modified_axgd}
\begin{aligned}
    \ftilted(\xast) \geq \newtarget{def:differentiable_lower_bound}{\Lt[t]} \defi &\frac{\int_{t_0}^t \ftilted(\xt[\tau]) d\alphat[\tau]}{\alphat[t]} + \frac{\min_{\tilde{u}\in \Q}\left\{\int_{t_0}^t \innp{\frac{1}{\gammanparam}\nabla \ftilted(\xt[\tau]), \tilde{u}-\xt[\tau]} d\alphat[\tau] + \breg[\psi](\tilde{u},\xt[t_0])\right\}}{\alphat[t]} \\
    & \quad + \frac{(\alphat[t]-\At[t])\ftilted(\xast)-\breg[\psi](\xast, \xt[t_0])}{\alphat[t]}.
\end{aligned}
\end{align}

Let $\zt[t] = \nabla \psi(\xt[t_0]) - \int_{t_0}^t \frac{1}{\gammanparam} \nabla \ftilted(\xt[\tau])d\alphat[\tau]$. Then, by \cref{grad_of_fenchel_dual}, we can compute the optimum $\tilde{u}\in \Q$ above as
\begin{equation}\label{eq:fenchel_dual_applied_to_z}
        \nabla \psi^{\astfenchel}(\zt[t]) = \argmin_{\tilde{u}\in \Q} \left\{\int_{t_0}^t\innp{\frac{1}{\gammanparam}\nabla \ftilted(\xt[\tau]), \tilde{u}-\xt[\tau]} d\alphat[\tau] + \breg[\psi](\tilde{u},\xt[t_0])\right\}.
\end{equation}
Recalling $\Ut[t] = \ftilted(\xt[t])$ and using \eqref{lower_bound_modified_axgd} and \eqref{eq:fenchel_dual_applied_to_z} we obtain:
\begin{align*} 
\begin{aligned}
    \frac{d}{dt}(\alphat[t] \Gt[t]) &= \frac{d}{dt} (\alphat[t] \ftilted(\xt[t])) - \dott{\alphat[t]}\ftilted(\xt[t]) - \dott{\alphat}_t \frac{1}{\gammanparam}\innp{\nabla \ftilted(\xt[t]), \nabla \psi^{\astfenchel}(\zt[t])-\xt[t]} \\
    &=\frac{1}{\gammanparam}\innp{\nabla \ftilted(\xt[t]), \gammanparam \alphat[t] \dott{\xt[]}_t - \dott{\alphat}_t(\nabla \psi^{\astfenchel}(\zt[t])-\xt[t])}.
\end{aligned}
\end{align*}
Thus, to satisfy the invariant $\frac{d}{dt}(\alphat[t] \Gt[t])=0$, it is enough to set $\gammanparam \alphat[t] \dott{\xt[]}_t = \dott{\alphat[t]}(\nabla \psi^{\astfenchel}(\zt[t])-\xt[t])$, yielding the following continuous accelerated dynamics
\begin{align}\label{appendix_modified_accelerated_continuous_dynamics}
\begin{aligned}
    \newtarget{def:iterate_z_continuous}{\dott{\zt[]}_{t}}&=-\frac{1}{\gammanparam}\dott{\alphat}_t \nabla \ftilted(\xt[t]), \\
    \newtarget{def:iterate_x_continuous}{\dott{\xt[]}_{t}}&=\frac{1}{\gammanparam}\dott{\alphat}_t \frac{\nabla \psi^{\astfenchel}(\zt[t])-\xt[t]}{\alphat[t]},\\
    \zt[t_0]&=\nabla \psi(\xt[t_0]),\\
    \xt[t_0] &\in  \Q \text { is an arbitrary initial point. }
\end{aligned}
\end{align}

Now we proceed to discretize the dynamics, so from now on we will use a discrete measure $\alphat[t]$, as we described above. We set $t_0$ to $1$. Let $ \newtarget{def:discretization_error}{\Ei[i+1]} \defi \At[i+1]\Gt[i+1]-\At[i]\Gt[i]$ be the discretization error. Then we have
\[
    \Gt[t] = \frac{\At[1]}{\At[t]}\Gt[1] + \frac{\sum_{i=1}^{t-1} \Ei[i+1]}{\At[t]}.
\]
\begin{lemma}\label{lemma:discretization_error_of_modified_axgd}
    If we have 
    \begin{equation}\label{eq:approximate_multiplied_convexity}
    \ftilted(\xI[i+1])-\ftilted(\xI[i]) \leq \hatgamma[i]  \innp{\nabla \ftilted(\xI[i+1]), \xI[i+1]-\xI[i]} + \hatepsilon[i],
    \end{equation}
    for some $\hatgamma[i], \hatepsilon[i] \geq 0$, then the discretization error satisfies
    \[
        \Ei[i+1] \leq \innp{\nabla \ftilted(\xI[i+1]), (\At[i]\hatgamma[i]  +\frac{\ai[i+1]}{\gammanparam})\xI[i+1]-\hatgamma[i]  \At[i] \xI[i]-\frac{\ai[i+1]}{\gammanparam}\nabla \psi^{\astfenchel}(\zi[i+1]))} - \breg[\psi^{\astfenchel}](\zi[i], \zi[i+1]) + \At[i]\hatepsilon[i].
    \]
\end{lemma}

\begin{proof}
    In a similar way to \citep{diakonikolas2017accelerated}, we could compute the discretization error as the difference between the gap and the gap computed allowing continuous integration rules in the integrals that it contains. However, we will directly bound $\Ei[i+1]$ as $\At[i+1]\Gt[i+1]-\At[i]\Gt[i]$ instead. Recall that in discrete time we have $\alpha_i = A_i$ so the definition of the lower bound in discrete time becomes the following, by combining \eqref{lower_bound_modified_axgd} and \eqref{eq:fenchel_dual_applied_to_z}:
\begin{align*}
\begin{aligned}
    \Lt[i] &= \sum_{j=1}^{i} \ai[j] \ftilted(\xI[j])  + \sum_{j=1}^{i} \innp{\frac{\ai[j]}{\gammanparam} \nabla \ftilted(\xI[j]), \nabla\psi^{\astfenchel}(\zi[i])-\xI[j]}  + \breg[\psi](\nabla \psi^{\astfenchel}(\zi[i]),\xt[t_0]) - \breg[\psi](\xast, \xt[t_0]).
\end{aligned}
\end{align*}

    Hence, using the definition of $\Gt[i], \Ut[i], \Lt[i]$ we have
\begingroup
\allowdisplaybreaks
\begin{align*}
\begin{aligned}
    \At[i+1]&\Gt[i+1]-\At[i]\Gt[i] \\
    &= (\At[i+1]\ftilted(\xI[i+1])-\At[i] \ftilted(\xI[i])) -\At[i+1]\Lt[i+1]+\At[i]\Lt[i] \\
    & \circled{1}[=] (\At[i]\ftilted(\xI[i+1])- \At[i]\ftilted(\xI[i]) + \ai[i+1]\ftilted(\xI[i+1]))  \\
    &\quad -\sum_{j=1}^{i+1} \ai[j] \ftilted(\xI[j]) - \sum_{j=1}^{i+1} \frac{\ai[j]}{\gammanparam} \innp{\nabla \ftilted(\xI[j]), \nabla \psi^{\astfenchel}(\zi[i+1])-\xI[j]} - \breg[\psi](\nabla \psi^{\astfenchel}(\zi[i+1]), \xt[t_0]) \\
    &\quad +\sum_{j=1}^{i} \ai[j] \ftilted(\xI[j]) + \sum_{j=1}^{i} \frac{\ai[j]}{\gammanparam} \innp{\nabla \ftilted(\xI[j]), \nabla \psi^{\astfenchel}(\zi[i])-\xI[j]} + \breg[\psi](\nabla \psi^{\astfenchel}(\zi[i]), \xt[t_0])  \\
    & \circled{2}[=] \At[i](\ftilted(\xI[i+1])-\ftilted(\xI[i])) -\innp{\frac{\ai[i+1]}{\gammanparam}\nabla \ftilted(\xI[i+1]), \nabla \psi^{\astfenchel}(\zi[i+1])-\xI[i+1]} \\
    &\quad+ \sum_{j=1}^i\innp{\frac{\ai[j]}{\gammanparam}\nabla \ftilted(\xI[j]), \nabla \psi^{\astfenchel}(\zi[i])-\nabla \psi^{\astfenchel}(\zi[i+1])}  \\
     &\quad [- \innp{\nabla \psi(\xt[t_0]), \nabla \psi^{\astfenchel}(\zi[i])-\nabla \psi^{\astfenchel}(\zi[i+1])} + \psi(\nabla\psi^{\astfenchel}(\zi[i])) -\psi(\nabla \psi^{\astfenchel}(\zi[i+1]))] \\
     & \circled{3}[=]  \At[i](\ftilted(\xI[i+1])-\ftilted(\xI[i])) -\innp{\frac{\ai[i+1]}{\gammanparam}\nabla \ftilted(\xI[i+1]), \nabla \psi^{\astfenchel}(\zi[i+1])-\xI[i+1]} -\breg[\psi^{\astfenchel}](\zi[i], \zi[i+1]) \\
     & \circled{4}[\leq] \innp{\nabla \ftilted(\xI[i+1]), (\At[i]\hatgamma[i]  +\frac{\ai[i+1]}{\gammanparam})\xI[i+1]-\hatgamma[i]  \At[i] \xI[i]-\frac{\ai[i+1]}{\gammanparam}\nabla \psi^{\astfenchel}(\zi[i+1])} - \breg[\psi^{\astfenchel}](\zi[i], \zi[i+1]) + \At[i]\hatepsilon[i].
\end{aligned}
\end{align*}
\endgroup

    In $\circled{1}$ we write down the definitions of $\Lt[i+1]$ and $\Lt[i]$ and split the first summand so it is clear that in $\circled{2}$ we cancel all the $\ai[j]\ftilted(\xI[j])$. In $\circled{2}$ we also cancel some terms involved in the inner products, we write the definitions of the Bregman divergences and cancel some of their terms. For equality $\circled{3}$, we recall $\zi[i] = \nabla \psi(\xt[t_0]) - \sum_{j=1}^i \frac{\ai[j]}{\gammanparam}\nabla \ftilted(\xI[j])$ so we use this fact and $\psi^{\astfenchel}(\tilde{z})= \innp{\nabla \psi^{\astfenchel}(\tilde{z}), \tilde{z}}-\psi(\nabla\psi^{\astfenchel}(\tilde{z}))$ (which holds by \cref{grad_of_fenchel_dual}) for $\tilde{z}=\zi[i]$ and $\tilde{z}=\zi[i+1]$ to conclude that the last two lines equal $-\breg[\psi^{\astfenchel}](\zi[i], \zi[i+1])$. Inequality $\circled{4}$ uses \eqref{eq:approximate_multiplied_convexity}.
\end{proof}

We show now how to cancel out the discretization error by an approximate implementation of implicit Euler discretization of \eqref{appendix_modified_accelerated_continuous_dynamics}. Note that we need to take into account the tilted-convexity assumption \eqref{eq:quasiquasarconvexity} instead of the usual convexity assumption. According to the previous lemma, we can set $\xI[i+1]$ so that the right hand side of the inner product in the bound of $\Ei[i+1]$ is $0$. Assume for the moment, that the point $\xI[i+1]$ we are going to compute satisfies the assumption of the previous lemma for some $\hatgamma[i]  \in [\gammapparam, 1/\gammanparam]$.  Thus, the implicit equation that defines the ideal method we would like to have is
\begin{align*}
\begin{aligned}
    \xI[i+1] = \frac{\hatgamma[i] \At[i]}{\At[i] \hatgamma[i] +\ai[i+1]/\gammanparam}\xI[i] + \frac{\ai[i+1]/\gammanparam}{\At[i]\hatgamma[i]  + \ai[i+1]/\gammanparam} \nabla \psi^{\astfenchel}(\zi[i]-\frac{\ai[i+1]}{\gammanparam} \nabla \ftilted(\xI[i+1])).
\end{aligned}
\end{align*}
Note that $\xI[i+1]$ is a convex combination of the other two points so it stays in $\Q$. Indeed, the initial point is in $\Q$ and by \eqref{eq:fenchel_dual_applied_to_z} we have that $\nabla \psi^{\astfenchel}(\zi[j]) \in \Q$ for all $j \geq 0$. However this method is implicit and possibly computationally expensive to implement. Nonetheless, two steps of a fixed point iteration procedure of this equation will be enough to have discretization error that is bounded by the term $\At[i]\hatepsilon[i]$: the last term of our bound. The error in the bound of $\Ei[i+1]$ that the inner product incurs is compensated by the Bregman divergence term. In such a case, the equations of this method become, for $i\geq 0$:
\begin{equation}\label{appendix_general_rule_modified_quasar_axgd}
\left\{\begin{array}{l}
    \newtarget{def:iterate_chi}{\chii[i]}=\frac{\hatgamma[i] \At[i]}{\At[i] \hatgamma[i] +\ai[i+1]/\gammanparam} \xI[i]+\frac{\ai[i+1]/\gammanparam}{\At[i]\hatgamma[i]  + \ai[i+1]/\gammanparam} \nabla \psi^{\astfenchel}(\zi[i]) \\
    \newtarget{def:iterate_zeta}{\zetai[i]} = \zi[i]-\frac{\ai[i+1]}{\gammanparam} \nabla \ftilted(\chii[i]) \\
    \newtarget{def:iterate_x}{\xI[i+1]}=\frac{\hatgamma[i] \At[i]}{\At[i] \hatgamma[i] +\ai[i+1]/\gammanparam} \xI[i]+\frac{\ai[i+1]/\gammanparam}{\At[i]\hatgamma[i]  + \ai[i+1]/\gammanparam} \nabla \psi^{\astfenchel}(\zetai[i]) \\
    \newtarget{def:iterate_z}{\zi[i+1]} = \zi[i]-\frac{\ai[i+1]}{\gammanparam} \nabla \ftilted(\xI[i+1]) 
\end{array}\right.
\end{equation}
We prove now that this indeed leads to an accelerated algorithm. After this, we will show that we can perform a binary search at each iteration, to ensure that even if we do not know $\xI[i+1]$ a priori, we can compute a $\newtarget{def:tilted_parameter_at_iteration_i}{\hatgamma[i]}  \in [\gammapparam, 1/\gammanparam]$ satisfying assumption \eqref{eq:approximate_multiplied_convexity}. This will only add a log factor to the overall complexity.

\begin{lemma}\label{lemma:approximate_implicit_method}
    Consider the method given in \eqref{appendix_general_rule_modified_quasar_axgd}, starting from an arbitrary point $\xInit \in  \Q$ with $\zi[0] = \nabla \psi(\xInit)$ and $\At[0]=0$. Assume we can compute $\hatgamma[i] $ such that $\xI[i+1]$ satisfies \eqref{eq:approximate_multiplied_convexity}. Then, the error from \cref{lemma:discretization_error_of_modified_axgd} is bounded by
\begin{align*}
\begin{aligned}
    \Ei[i+1] &\leq  \frac{\ai[i+1]}{\gammanparam}\innp{\nabla \ftilted(\xI[i+1]) - \nabla \ftilted(\chii[i]), \nabla \psi^{\astfenchel}(\zetai[i])-\nabla \psi^{\astfenchel}(\zi[i+1])} - \breg[\psi^{\astfenchel}](\zetai[i],\zi[i+1]) - \breg[\psi^{\astfenchel}](\zi[i],\zetai[i]) + \At[i]\hatepsilon[i].\\
\end{aligned}
\end{align*}
\end{lemma}

\begin{proof}
    Using \cref{lemma:discretization_error_of_modified_axgd} and the third line of \eqref{appendix_general_rule_modified_quasar_axgd} we have
    \begin{align*}
    \begin{aligned}
        \Ei[i+1] -\At[i]\hatepsilon[i] &\leq  \frac{\ai[i+1]}{\gammanparam}\innp{\nabla \ftilted(\xI[i+1]), \nabla \psi^{\astfenchel}(\zetai[i])-\nabla \psi^{\astfenchel}(\zi[i+1])} - \breg[\psi^{\astfenchel}](\zi[i],\zi[i+1])\\
        &\leq  \frac{\ai[i+1]}{\gammanparam}\innp{\nabla \ftilted(\xI[i+1]) - \nabla \ftilted(\chii[i])+ \nabla \ftilted(\chii[i]), \nabla \psi^{\astfenchel}(\zetai[i])-\nabla \psi^{\astfenchel}(\zi[i+1])}  - \breg[\psi^{\astfenchel}](\zi[i],\zi[i+1])\\
    \end{aligned}
    \end{align*}
    By the definition of $\zetai[i]$ we have $(\ai[i+1]/\gammanparam) \nabla \ftilted(\chii[i]) = \zi[i] - \zetai[i]$. Using this fact and the triangle equality of Bregman divergences \cref{prop:triangle_inequality_of_bregman_div}, we obtain
    \begin{align*}
    \begin{aligned}
        \frac{\ai[i+1]}{\gammanparam}\innp{\nabla \ftilted(\chii[i]), \nabla \psi^{\astfenchel}(\zetai[i])&-\nabla \psi^{\astfenchel}(\zi[i+1])} = \innp{\zi[i]-\zetai[i], \nabla \psi^{\astfenchel}(\zetai[i])-\nabla \psi^{\astfenchel}(\zi[i+1])}\\
        &= \breg[\psi^{\astfenchel}](\zi[i], \zi[i+1]) - \breg[\psi^{\astfenchel}](\zetai[i], \zi[i+1]) - \breg[\psi^{\astfenchel}](\zi[i], \zetai[i]).
    \end{aligned}
    \end{align*}
    The lemma follows after combining these two equations.
\end{proof}

\begin{theorem}\label{thm:accelerated_axgd_modified}
    Let $\Q$ be a closed convex set of diameter $D$. 
    Let $f: \Q \to \R$ be a tilted-convex function with constants $\gammanparam, \gammapparam$ with $\Ltildegeneric$-Lipschitz gradient. Assume there is a point $\xast \in \Q$ such that $\nabla \ftilted(\xast) = 0$. Let $\psi: \Q \to \R$ be a $\newtarget{def:strong_convexity_of_regularizer}{\stcvxpsi}$-strongly convex map. Let $\xI[i], \zi[i], \chii[i], \zetai[i]$ be updated according to \eqref{appendix_general_rule_modified_quasar_axgd}, for $i\geq 0$ starting from an arbitrary initial point $\xInit \in  \Q$ with $\zi[0] = \nabla \psi(\xInit)$ and $\At[0]=0$, assuming we can find $\hatgamma[i] $ at each iteration satisfying \eqref{eq:approximate_multiplied_convexity}. If $\Ltildegeneric \ai[i+1][2]/\gammanparam\stcvxpsi \leq \ai[i+1] + \At[i]\gammanparam\gammapparam$, then for all $\T\geq 1$ we have 
    \[
        \ftilted(\xI[\T])-\ftilted(\xast) \leq \frac{\breg[\psi](\xast, \chii[0])}{\At[\T]} + \sum_{i=1}^{\T-1} \frac{\At[i]\hatepsilon[i]}{\At[\T]}.
    \]
    In particular, if $\ai[i] = \frac{i}{2} \frac{\stcvxpsi}{\Ltildegeneric} \gammanparam[2] \gammapparam$, $\psi(\tilde{x}) =\frac{\stcvxpsi}{2}\norm{\tilde{x}}^2$, $\newtarget{def:accuracy_of_binary_search}{\hatepsilon[i]}=\frac{\At[\T]\epsilon}{2(\T-1) \At[i]}$ and $\newtarget{def:total_number_of_iterations_T}{\T}=\left\lceil\sqrt{4\Ltildegeneric\norm{\xInit-\xast}^2/(\gammanparam[2]\gammapparam\epsilon)}\right\rceil= O(\sqrt{\Ltildegeneric/(\gammanparam[2]\gammapparam\epsilon)})$  then 
    \[
        \ftilted(\xI[\T])-\ftilted(\xast)  < \epsilon.
    \]
\end{theorem}

\begin{proof}
    We bound the right hand side of the discretization error given by \cref{lemma:approximate_implicit_method}. Define $a = \norm{\nabla \psi^{\astfenchel}(\zetai[i])-\nabla \psi^{\astfenchel}(\zi[i+1])}$ and $b=\norm{\nabla \psi^{\astfenchel}(\zetai[i])-\nabla \psi^{\astfenchel}(\zi[i])}$. We have
    \begin{align*}
    \begin{aligned}
        \Ei[i+1] -\At[i]\hatepsilon[i]&\circled{1}[\leq]  \frac{\ai[i+1]}{\gammanparam}\innp{\nabla \ftilted(\xI[i+1]) - \nabla \ftilted(\chii[i]), \nabla \psi^{\astfenchel}(\zetai[i])-\nabla \psi^{\astfenchel}(\zi[i+1])} - \breg[\psi^{\astfenchel}](\zetai[i],\zi[i+1]) - \breg[\psi^{\astfenchel}](\zi[i],\zetai[i])\\
        &\circled{2}[\leq]  \frac{\ai[i+1]}{\gammanparam}\Ltildegeneric\norm{\xI[i+1]-\chii[i]}\cdot a - \breg[\psi^{\astfenchel}](\zetai[i],\zi[i+1]) - \breg[\psi^{\astfenchel}](\zi[i],\zetai[i])\\
        &\circled{3}[\leq]  \frac{\ai[i+1]}{\gammanparam}\Ltildegeneric\norm{\xI[i+1]-\chii[i]}\cdot a - \frac{\stcvxpsi}{2}(a^2+b^2)\\
        &\circled{4}[\leq]  \frac{\ai[i+1][2]/\gammanparam[2]}{\At[i] \hatgamma[i]  + \ai[i+1]/\gammanparam}\Ltildegeneric\cdot ab - \frac{\stcvxpsi}{2}(a^2+b^2)\\
        &\circled{5}[\leq]  ab \left(\frac{\ai[i+1][2]/\gammanparam[2]}{\At[i] \hatgamma[i]  + \ai[i+1]/\gammanparam}\Ltildegeneric - \stcvxpsi\right).\\
    \end{aligned}
    \end{align*}
    Here $\circled{1}$ follows from \cref{lemma:approximate_implicit_method}, $\circled{2}$ uses the Cauchy-Schwartz inequality and gradient Lipschitzness. In $\circled{3}$, we used \cref{prop:bounding_breg_div_by_norm_of_difference}, and $\circled{4}$ uses the fact that by the definition of the method \eqref{appendix_general_rule_modified_quasar_axgd} we have $\xI[i+1]-\chii[i] = \frac{\ai[i+1]/\gammanparam}{\At[i] \hatgamma[i]  + \ai[i+1]/\gammanparam} (\nabla \psi^{\astfenchel}(\zetai[i])-\nabla \psi^{\astfenchel}(\zi[i]))$. Finally $\circled{5}$ uses $-(a^2+b^2) \leq -2ab$, which comes from $(a-b)^2 \geq 0$. By the previous inequality, if we want $\Ei[i+1] \leq \At[i]\hatepsilon[i]$, it is enough to guarantee the right hand side of the last expression is $\leq 0$ which is implied by
    \begin{equation}\label{eq:assumption_on_the_learning_rate_quasi_quasar_convexity}
        \frac{\Ltildegeneric}{\stcvxpsi \gammanparam}\ai[i+1][2] \leq \ai[i+1] + \At[i] \gammanparam \gammapparam,
    \end{equation}
    since $\gammapparam \leq \hatgamma[i] $. And this is the assumption we made in the theorem. By inspection, if we use the value in the second part of the statement of the theorem $\ai[i] = \frac{i}{2}\cdot \frac{\stcvxpsi}{\Ltildegeneric}\cdot \gammanparam[2]\gammapparam$ into the previous inequality and noting that $\At[i] =\frac{i(i+1)}{4} \cdot \frac{\stcvxpsi}{\Ltildegeneric}\cdot \gammanparam[2]\gammapparam $ we prove that the previous inequality is satisfied:
    \begin{align*}
    \begin{aligned}
        \frac{\Ltildegeneric}{\stcvxpsi\gammanparam} \ai[i+1][2] &= \frac{(i+1)^2}{4}\cdot \frac{\stcvxpsi}{\Ltildegeneric}\cdot \gammanparam[3]\gammapparam[2] \\
        & \leq \left(\frac{i+1}{2} + \frac{i(i+1)}{4}\right) \frac{\stcvxpsi}{\Ltildegeneric}\cdot \gammanparam[3]\gammapparam[2]  \\
        & \leq \frac{i+1}{2} \frac{\stcvxpsi}{\Ltildegeneric}\cdot \gammanparam[2]\gammapparam + \frac{i(i+1)}{4} \frac{\stcvxpsi}{\Ltildegeneric}\cdot \gammanparam[3]\gammapparam[2] \\
        &= \ai[i+1] + \At[i] \gammanparam\gammapparam.
    \end{aligned}
    \end{align*}
    So this choice, and in particular any choice that satisfies \eqref{eq:assumption_on_the_learning_rate_quasi_quasar_convexity}, guarantees discretization error $\Ei[i+1] \leq \At[i]\hatepsilon[i]$. By the definition of $\Gt[i]$ and $\Ei[i]$ we have 
    \[
        \ftilted(\xI[\T])-\ftilted(\xast) \leq \frac{\At[1]\Gt[1]}{\At[\T]} + \sum_{i=1}^{\T-1} \frac{\At[i]\hatepsilon[i]}{\At[\T]} 
    \]
    So it only remains to bound the initial gap $\Gt[1]$. In order to do this, we note that the initial conditions and the method imply the following computation of the first points, from $\xInit \in  \Q$, which is an arbitrary initial point:
\begin{equation}
\left\{\begin{array}{l}
        \zi[0] = \nabla \psi(\xInit) \\
        \chii[0] = \frac{\hat{\gamma}_0\At[0]}{\At[0] \hat{\gamma}_0+\ai[1]/\gammanparam} \xInit + \frac{\ai[1]/\gammanparam}{\At[0] \hat{\gamma}_0+\ai[1]/\gammanparam} \nabla \psi^{\astfenchel}(\zi[0]) = \nabla\psi^{\astfenchel}(\nabla \psi(\xInit)) = \xInit \\
        \zetai[0] = \zi[0] -\frac{\ai[1]}{\gammanparam} \nabla \ftilted(\chii[0]) = \zi[0] -\frac{\ai[1]}{\gammanparam} \nabla \ftilted(\xInit)\\
        \xI[1] = \frac{\hat{\gamma}_0\At[0]}{\At[0] \hat{\gamma}_0+\ai[1]/\gammanparam} \xInit + \frac{\ai[1]/\gammanparam}{\At[0] \hat{\gamma}_0+\ai[1]/\gammanparam} \nabla \psi^{\astfenchel}(\zetai[0])  =  \nabla \psi^{\astfenchel}(\zetai[0])
\end{array}\right.
\end{equation}
    We have used $\At[0]=0$. Note this first iteration does not depend on $\hat{\gamma}_0$. Also, by using this discretization we start at $\xInit$ so we modify the definition of the lower bound \eqref{lower_bound_modified_axgd} so the regularizer added measures the distance from $\xInit$. This change of $\xt[t_0]$ to $\xInit =\chii[0]$ only changes the initial gap. Thus, the first lower bound computed is
    \begin{align*}
    \begin{aligned}
        \Lt[1] = \ftilted(\xI[1]) + \frac{1}{\gammanparam}\innp{\nabla \ftilted(\xI[1]), \nabla \psi^{\astfenchel}(\zi[1])-\xI[1]}+\frac{1}{\At[1]} \breg[\psi](\nabla \psi^{\astfenchel}(\zi[1]), \chii[0]) - \frac{1}{\At[1]} \breg[\psi](\xast, \chii[0]).
    \end{aligned}
    \end{align*}
    Using $\ai[1]=\At[1]$, $\xI[1] = \nabla \psi^{\astfenchel} (\zetai[0])$, $(\ai[1]/\gammanparam) \nabla \ftilted(\chii[0])=\zi[0]-\zetai[0]$, and the triangle equality for Bregman divergences \cref{prop:triangle_inequality_of_bregman_div} we obtain
    \begin{align}\label{bounding_g_1_equation_1}
    \begin{aligned}
        \frac{1}{\gammanparam}\innp{\nabla \ftilted(\chii[0]), \nabla \psi^{\astfenchel}(\zi[1])-\xI[1]} &= \frac{1}{\At[1]} \innp{\zi[0]-\zetai[0], \nabla \psi^{\astfenchel}(\zi[1]) - \nabla \psi^{\astfenchel} (\zetai[0])} \\
    &=\frac{1}{\At[1]}\left(\breg[\psi^{\astfenchel}](\zi[0],\zetai[0])- \breg[\psi^{\astfenchel}](\zi[0],\zi[1]) +\breg[\psi^{\astfenchel}](\zetai[0],\zi[1])\right).
    \end{aligned}
    \end{align}
    On the other hand, by the gradient Lipschitzness $\ftilted$ and the initial condition we have
    \begin{equation}\label{bounding_g_1_equation_2}
        \frac{1}{\gammanparam}\innp{\nabla \ftilted(\xI[1])-\nabla \ftilted(\chii[0]), \nabla \psi^{\astfenchel}(\zi[1])-\xI[1]} \geq -\frac{\Ltildegeneric}{\gammanparam} \norm{\nabla \psi^{\astfenchel}(\zetai[0])-\chii[0]} \norm{\nabla \psi^{\astfenchel}(\zi[1])-\xI[1]}.
    \end{equation}
    We can now finally bound $\Gt[1]$:
    \begin{align*}
    \begin{aligned}
        \Gt[1] &\circled{1}[\leq] \frac{\Ltildegeneric}{\gammanparam} \norm{\nabla \psi^{\astfenchel}(\zetai[0])-\chii[0]} \cdot \norm{\nabla \psi^{\astfenchel} (\zi[1])-\xI[1]} \\
            & \quad -\frac{1}{\At[1]} \left(\breg[\psi^{\astfenchel}](\zi[0],\zetai[0])+\breg[\psi^{\astfenchel}](\zetai[0], \zi[1])\right) + \frac{1}{\At[1]} \breg[\psi](\xast, \chii[0]) \\
            & \circled{2}[\leq] \frac{\Ltildegeneric}{\gammanparam} \norm{\nabla \psi^{\astfenchel}(\zetai[0])-\chii[0]} \cdot \norm{\nabla \psi^{\astfenchel}(\zi[1])-\xI[1]} \\
            &\quad -\frac{\stcvxpsi}{2\At[1]}\left(\norm{\nabla \psi^{\astfenchel}(\zetai[0])-\chii[0]}^2 +\norm{\nabla \psi^{\astfenchel}(\zi[1])-\xI[1]}^2\right) + \frac{1}{\At[1]}\breg[\psi](\xast, \chii[0]) \\
            & \circled{3}[\leq] \norm{\nabla \psi^{\astfenchel}(\zetai[0])-\chii[0]} \cdot \norm{\nabla \psi^{\astfenchel}(\zi[1])-\xI[1]}\left(\frac{\Ltildegeneric}{\gammanparam}-\frac{\stcvxpsi}{\At[1]} \right) + \frac{1}{\At[1]}\breg[\psi](\xast, \chii[0]) \\
            & \circled{4}[\leq] \frac{1}{\At[1]} \breg[\psi](\xast, \chii[0]).
    \end{aligned}
    \end{align*}
    We used in $\circled{1}$ the definition of $\Gt[1]=\Ut[1]-\Lt[1] = \ftilted(\xI[1])-\Lt[1]$ and we bound the inner product in $\Lt[1]$ using $- (\eqref{bounding_g_1_equation_1} + \eqref{bounding_g_1_equation_2})$. Also, since $\zi[0] = \nabla \psi(\chii[0])$ we have $\breg[\psi^{\astfenchel}](\zi[0], \zi[1]) = \breg[\psi^{\astfenchel}](\nabla \psi(\chii[0]), \zi[1]) = \breg[\psi](\nabla \psi^{\astfenchel}(\zi[1]),\chii[0])$, so we can cancel two of the Bregman divergences. 
    In $\circled{2}$, we used \cref{prop:bounding_breg_div_by_norm_of_difference}, $\nabla \psi^{\astfenchel}(\zi[0]) = \xInit= \chii[0]$, and $\nabla \psi^{\astfenchel}(\zetai[0])=\xI[1]$. In $\circled{3}$ we used again the inequality $-(a^2+b^2) \leq -2ab$. 
    Finally $\circled{4}$ is deduced from $\At[1]=\ai[1] \leq \stcvxpsi \gammanparam/\Ltildegeneric$ which comes from the assumption $\Ltildegeneric \ai[i+1][2]/\gammanparam\stcvxpsi \leq \ai[i+1]+\At[i]\gammanparam\gammapparam$ for $i=0$.

    The first part of the theorem follows. The second one is a straightforward application of the first one as we see below. Indeed, taking into account $\At[\T] = \frac{\T(\T+1)\stcvxpsi \gammanparam[2] \gammapparam}{4\Ltildegeneric}$, and the choice of $\T=\left\lceil\sqrt{4\Ltildegeneric\norm{\xInit-\xast}^2/(\gammanparam[2]\gammapparam\epsilon)}\right\rceil$, $\psi(\tilde{x}) =\frac{\stcvxpsi}{2}\norm{\tilde{x}}^2$, and $\hatepsilon[i]=\frac{\At[\T]\epsilon}{2(\T-1) \At[i]}$  we derive the second statement.
    \[
        \ftilted(\xI[\T]) - \ftilted(\xast) \leq \frac{\At[1] \Gt[1]}{\At[\T]} + \sum_{i=1}^{\T-1} \frac{\At[i]\hatepsilon[i]}{\At[\T]} \leq \frac{\frac{\stcvxpsi}{2}\norm{\xInit-\xast}^2}{\At[\T]} + \frac{\epsilon}{2} < \frac{2\Ltildegeneric\norm{\xInit -\xast}^2}{\gammanparam[2]\gammapparam \T^2} + \frac{\epsilon}{2} \leq \epsilon.
    \]
\end{proof}

We present now the final lemma, that proves that $\hatgamma[i] $ can be found efficiently. As we advanced in the sketch of the main paper, we use a binary search. The idea behind it is that due to tilted-convexity \eqref{eq:quasiquasarconvexity} we satisfy the equation for $\hatgamma[i] =\frac{1}{\gammanparam}$ or $\hatgamma[i] =\gammapparam$, or there is $\hatgamma[i] \in(\gammapparam, 1/\gammanparam)$ such that $\innp{\nabla \ftilted(\xI[i+1]), \xI[i+1]-\xI[i]}=0$. The existence of $\xast$ that satisfies $\nabla \ftilted(\xast)=0$ along with the boundedness of $\Q$ and smoothness, imply the Lipschitzness of $\ftilted$. Both Lipschitzness and smoothness allow to prove that a binary search finds efficiently a suitable point. 

\begin{lemma} \label{lemma:binary_search}
    Let $\Q \subseteq\R^{\n}$ be a convex set of diameter $2\Rtilde$.  Let $f:\Q\to\R$ be a function that satisfies tilted-convexity \eqref{eq:quasiquasarconvexity}, has $\Ltildegeneric$ Lipschitz gradients and such that there is a global optimizer $\xastgtilde\in \R^n$ with $\norm{\xInit- \xastgtilde} \leq \Rglobaltilde$. Let the strongly convex parameter of $\psi(\cdot)$ be $\stcvxpsi = O(1)$. Let $i\geq 1$ be an index. Given two points $\xI[i], \zi[i] \in  \Q$ and the method in \eqref{general_rule_modified_quasar_axgd} using the learning rates $\ai[i]=\frac{i}{2}\cdot \frac{\stcvxpsi}{\Ltildegeneric}\cdot \gammanparam[2]\gammapparam$ prescribed in \cref{thm:accelerated_axgd_modified}, we can compute $\hatgamma[i] $ satisfying \eqref{eq:approximate_multiplied_convexity}, i.e.,
    \begin{equation}\label{property_result_of_line_search_approximate}
        \ftilted(\xI[i+1])-\ftilted(\xI[i]) \leq \hatgamma[i]  \innp{\nabla \ftilted(\xI[i+1]), \xI[i+1]-\xI[i]} + \hatepsilon[i]. 
    \end{equation}
    And the computation of $\hatgamma[i] $ requires no more than
    \[
        O\left(\log\left(\frac{\Ltildegeneric(\Rglobaltilde+\tilde{R})}{\gammanparam\hatepsilon[i]}\cdot i\right)\right)
    \]
    queries to the gradient oracle. 
\end{lemma}

\begin{proof}
    Let $\hypertarget{def:function_for_tilted_parameter}{\Gammahat[i]}(\lambda):[\frac{\ai[i+1]}{\At[i+1]}, \frac{\ai[i+1]/\gammanparam}{\At[i]\gammapparam+\ai[i+1]/\gammanparam}] \to \R$ be defined as 
    \begin{equation}\label{def:Gamma}
        \Gammahat[i]\left(\frac{\ai[i+1]/\gammanparam}{\At[i] \mathtt{\tilde{x}}+\ai[i+1]/\gammanparam}\right) = \mathtt{\tilde{x}},\text{ for }\mathtt{\tilde{x}}\in [\gammapparam, \frac{1}{\gammanparam}].  
    \end{equation}
    By monotonicity, it is well defined. Let $\xI[i+1][{\newtarget{def:lambda_superscript_notation}{\lambdanot{\lambda}}}]$ be the point computed by one iteration of \eqref{general_rule_modified_quasar_axgd} using the parameter $\hatgamma[i]  = \Gammahat[i](\lambda)$. Likewise, we define the rest of the points in iteration \eqref{general_rule_modified_quasar_axgd} depending on $\lambda$. We first try $\hatgamma[i]  = 1/\gammanparam$ and $\hatgamma[i]  = \gammapparam$ and use any of them if they satisfy the conditions. If neither of them do, it means that for the first choice we had $\innp{\nabla \ftilted(\xI[i+1][\lambdanot{\lambda_1}]), \xI[i+1][\lambdanot{\lambda_1}] - \xI[i]} < 0$ and for the second one, it is $\innp{\nabla \ftilted(\xI[i+1][\lambdanot{\lambda_2}]), \xI[i+1][\lambdanot{\lambda_2}] - \xI[i]} > 0 $, for $\lambda_1 = \Gammahat[i][-1](1/\gammanparam)$ and $\lambda_2 = \Gammahat[i][-1](\gammapparam)$. Therefore, by continuity, there is $\lambda^\ast \in [\lambda_1, \lambda_2]$ such that $\innp{\nabla \ftilted(\xI[i+1][\lambdanot{\lambda^\ast}]), \xI[i+1][\lambdanot{\lambda^\ast}] - \xI[i]}=0$. The continuity condition is easy to prove. We omit it because it is derived from the Lipschitzness condition that we will prove below. Such a point satisfies \eqref{eq:approximate_multiplied_convexity} for $\hatepsilon[i]=0$. We will prove that the function $\newtarget{def:G_binary_search}{\Gbinsearch[i]}:[\frac{\ai[i+1]}{\At[i+1]}, \frac{\ai[i+1]/\gammanparam}{\At[i]\gammapparam+\ai[i+1]/\gammanparam}]\to\R$, defined as 
    \begin{equation}\label{def:G}
    \Gbinsearch[i](\lambda) \defi -\Gammahat[i](\lambda) \innp{\nabla \ftilted(\xI[i+1][\lambdanot{\lambda}]), \xI[i+1][\lambdanot{\lambda}]-\xI[i]}+ (\ftilted(\xI[i+1][\lambdanot{\lambda}])-\ftilted(\xI[i])),
    \end{equation}
    is Lipschitz so we can guarantee that \eqref{eq:approximate_multiplied_convexity} holds for a large enough interval around $\lambda^\ast$. Finally, we will be able to perform a binary search to efficiently find a point in such interval or another interval around another point that satisfies that the inner product is $0$.

    So 
    \begin{align}\label{lipschitzness_of_G}
    \begin{aligned}
        \abs{\Gbinsearch[i](\lambda) - \Gbinsearch[i](\lambda')} &\leq \abs{\ftilted(\xI[i+1][\lambdanot{\lambda}])-\ftilted(\xI[i+1][\lambdanot{\lambda'}])} \\
        &\quad+ \abs{\Gammahat[i](\lambda')}\cdot\abs{\innp{\nabla \ftilted(\xI[i+1][\lambdanot{\lambda'}]), \xI[i+1][\lambdanot{\lambda'}]-\xI[i]}-\innp{\nabla \ftilted(\xI[i+1][\lambdanot{\lambda}]), \xI[i+1][\lambdanot{\lambda}]-\xI[i]}} \\
        &\quad+\abs{\innp{\nabla \ftilted(\xI[i+1][\lambdanot{\lambda}]), \xI[i+1][\lambdanot{\lambda}]-\xI[i]}}\cdot\abs{\Gammahat[i](\lambda')-\Gammahat[i](\lambda)} 
    \end{aligned}
    \end{align}
    We have used the triangular inequality and the inequality 
    \begin{equation}\label{simple_inequality_from_triangular}
    \abs{\alpha_1 \beta_1 -\alpha_2\beta_2} \leq \abs{\alpha_1}\abs{\beta_1-\beta_2}+\abs{\beta_2}\abs{\alpha_1-\alpha_2},
   \end{equation} 
   which is a direct consequence of the triangular inequality, after adding and subtracting $\alpha_1\beta_2$ in the $\abs{\cdot}$ on the left hand side. We bound each of the three summands of the previous inequality separately, but first we bound the following which will be useful for our other bounds, 
\begingroup
\allowdisplaybreaks
\begin{align} 
\begin{aligned}
      \norm{\xI[i+1][\lambdanot{\lambda'}]&-\xI[i+1][\lambdanot{\lambda}]}    \circled{1}[=] \norm{(\lambda'\nabla\psi^{\astfenchel}(\zetai[i][\lambdanot{\lambda'}])+(1-\lambda')\xI[i])-(\lambda\nabla\psi^{\astfenchel}(\zetai[i][\lambdanot{\lambda}])+(1-\lambda)\xI[i])} \\
        & \circled{2}[\leq] \norm{\nabla \psi^{\astfenchel}(\zetai[i][\lambdanot{\lambda}])-\xI[i]}\abs{\lambda'-\lambda} + \norm{\lambda'\nabla\psi^{\astfenchel}(\zetai[i][\lambdanot{\lambda'}]) -\lambda'\nabla\psi^{\astfenchel}(\zetai[i][\lambdanot{\lambda}])} \\
        & \circled{3}[\leq] 2\Rtilde\abs{\lambda-\lambda'} +\norm{\nabla\psi^{\astfenchel}(\zetai[i][\lambdanot{\lambda'}])-\nabla\psi^{\astfenchel}(\zetai[i][\lambdanot{\lambda}])} \circled{4}[\leq] 2\Rtilde\abs{\lambda-\lambda'} + \frac{a_{i+1}}{\gammanparam\stcvxpsi} \norm{\nabla \ftilted(\chii[i][\lambdanot{\lambda}])-\nabla \ftilted(\chii[i][\lambdanot{\lambda'}])} \\ 
        & \circled{5}[\leq] 2\Rtilde\abs{\lambda-\lambda'} + \frac{a_{i+1}\Ltildegeneric}{\gammanparam\stcvxpsi}\norm{\chii[i][\lambdanot{\lambda}]-\chii[i][\lambdanot{\lambda'}]}  \circled{6}[\leq] \left(2\Rtilde+\frac{2a_{i+1}\Ltildegeneric\Rtilde}{\gammanparam\stcvxpsi}\right)\abs{\lambda-\lambda'}
\end{aligned}
\end{align}
\endgroup
    Here, $\circled{1}$ uses the definition of $\xI[i+1][\lambdanot{\lambda}]$ as a convex combination of $\xI[i]$ and $\nabla\psi^{\astfenchel}(\zetai[i][\lambdanot{\lambda}])$. $\circled{2}$ adds and substracts $\lambda'\nabla \psi^{\astfenchel}(\zetai[i][\lambdanot{\lambda}])$, groups terms and uses the triangular inequality. In $\circled{3}$ we use the fact that the diameter of $\Q$ is $2\Rtilde$ and bound $\lambda'\leq 1$, and $\abs{\lambda} \leq 1$. $\circled{4}$ uses the $\frac{1}{\stcvxpsi}$ Lipschitzness of $\nabla \psi^{\astfenchel}(\cdot)$, which is a consequence of the $\stcvxpsi$-strong convexity of $\psi(\cdot)$. $\circled{5}$ uses the smoothness of $\ftilted$. In $\circled{6}$, from the definition of $\chii[i][\lambdanot{\lambda}]$ we have that $ \norm{\chii[i][\lambdanot{\lambda}]-\chii[i][\lambdanot{\lambda'}]}\leq \norm{\xI[i]-\zi[i]}\abs{\lambda-\lambda'}$. We bounded this further using the diameter of $ \Q$.

    Note that $\ftilted$ is Lipschitz over $\Q$. By the existence of $\xastg$, $\Ltildegeneric$-gradient Lipschitzness and the diameter of $\Q$, we have $\norm{\nabla \ftilted(\tilde{x})} =\norm{\nabla \ftilted(\tilde{x})-\nabla \ftilted(\xastgtilde)} \leq \Ltildegeneric\norm{\tilde{x}-\xastg} \leq 2(\Rglobaltilde+\tilde{R})\Ltildegeneric $. So the Lipschitz constant $\hypertarget{def:Lipschitz_constant}{\Lips}$ of $\ftilted$ is $\Lips \leq 2(\Rglobaltilde+\tilde{R})\Ltildegeneric$. Now we can proceed and bound the three summands of \eqref{lipschitzness_of_G}. The first one reduces to the inequality above after using Lipschitzness of $\ftilted(\cdot)$: 
    \begin{equation}\label{eq:first_summand}
        \abs{\ftilted(\xI[i+1][\lambdanot{\lambda}])-\ftilted(\xI[i+1][\lambdanot{\lambda'}])}  \leq \Lips\norm{\xI[i+1][\lambdanot{\lambda'}]-\xI[i+1][\lambdanot{\lambda}]}.
    \end{equation}
    We prove Lipschitzness of $\Gammahat[i]$. Note that 
    \begin{equation}\label{eq:lipschitzness_of_Gamma}
        \abs{(\Gammahat[i][-1])'(\mathtt{\tilde{x}})} = \absadj{\frac{\At[i]\ai[i+1]/\gammanparam}{(\At[i]\mathtt{\tilde{x}}+\ai[i+1]/\gammanparam)^2}} \geq  \frac{\gammanparam \At[i]\ai[i+1]}{\At[i+1][2]},
    \end{equation}
    so $\Gammahat[i]'(\lambda)$ is bounded by $\At[i+1][2]/(\gammanparam \At[i] \ai[i+1])$ for any $\lambda$. In order to bound the second summand, we use $\mathtt{\tilde{x}}\in[\gammapparam, 1/\gammanparam]$ and obtain $\abs{\Gammahat[i](\lambda)} \leq \frac{1}{\gammanparam}$. For the second factor, we add and subtract $\innp{\nabla \ftilted(\xI[i+1][\lambdanot{\lambda}]),\xI[i+1][\lambdanot{\lambda'}]-\xI[i]}$ and use the triangular inequality and then Cauchy-Schwartz. Thus, we obtain
    \begin{align}\label{eq:second_summand}
    \begin{aligned}
        \abs{&\innp{\nabla \ftilted(\xI[i+1][\lambdanot{\lambda'}]), \xI[i+1][\lambdanot{\lambda'}]-\xI[i]}-\innp{\nabla \ftilted(\xI[i+1][\lambdanot{\lambda}]), \xI[i+1][\lambdanot{\lambda}]-\xI[i]}} \\
       &\leq \norm{\nabla \ftilted(\xI[i+1][\lambdanot{\lambda}])}\cdot \norm{\xI[i+1][\lambdanot{\lambda'}]-\xI[i+1][\lambdanot{\lambda}]} +\norm{\nabla \ftilted(\xI[i+1][\lambdanot{\lambda'}])-\nabla \ftilted(\xI[i+1][\lambdanot{\lambda}])}\cdot\norm{\xI[i+1][\lambdanot{\lambda'}]-\xI[i]} \\
        &\circled{1}[\leq] (\Lips +2\Ltildegeneric\Rtilde)\norm{\xI[i+1][\lambdanot{\lambda'}]-\xI[i+1][\lambdanot{\lambda}]}.
    \end{aligned}
\end{align}
In $\circled{1}$, we used Lipschitzness to bound the first factor. We also used the diameter of $ \Q$ to bound the last factor and the smoothness of $\ftilted(\cdot)$ to bound the first factor of the second summand.

    For the third summand, we will bound the first factor using Cauchy-Schwartz, Lipschitzness of $\ftilted(\cdot)$ and the diameter of $ \Q$. We just proved in \eqref{eq:lipschitzness_of_Gamma} that $\Gammahat[i]$ is Lipschitz, so use this property for the second factor. The result is the following
    \begin{align}\label{eq:third_summand}
    \begin{aligned}
        \abs{\innp{\nabla \ftilted(\xI[i+1][\lambdanot{\lambda}]), \xI[i+1][\lambdanot{\lambda}]-\xI[i]}}\cdot\abs{\Gammahat[i](\lambda')-\Gammahat[i](\lambda)}  \leq 2\Lips\Rtilde \frac{\At[i+1][2]}{\gammanparam \At[i] \ai[i+1]}\abs{\lambda'-\lambda}.
    \end{aligned}
\end{align}

Applying the bounds of the three summands \eqref{eq:first_summand}, \eqref{eq:lipschitzness_of_Gamma}, \eqref{eq:second_summand}, \eqref{eq:third_summand} into \eqref{lipschitzness_of_G} we obtain the inequality $\abs{\Gbinsearch[i](\lambda') - \Gbinsearch[i](\lambda)} \leq \hat{L}\abs{\lambda'-\lambda}$ for 
\begin{align*}
    \begin{aligned}
        \hat{L} = \left(2\Rtilde+\frac{2a_{i+1}\Ltildegeneric\Rtilde}{\gammanparam\stcvxpsi}\right)\left(\Lips + (\Lips+2\Ltildegeneric\Rtilde)\frac{1}{\gammanparam}\right)+ 2\Lips\Rtilde \frac{\At[i+1][2]}{\gammanparam \At[i] \ai[i+1]}.
    \end{aligned}
\end{align*}

We will use the following to bound $\hat{L}$. If we use the learning rates prescribed in \cref{thm:accelerated_axgd_modified}, namely $\ai[i] = \frac{i\stcvxpsi \gammanparam[2]\gammapparam}{2L}$ and thus $\At[i]=\frac{i(i+1)\stcvxpsi\gammanparam[2]\gammapparam}{4L}$ we can bound $A^2_{i+1}/(\At[i]\ai[i+1]) \leq 4(i+2)$, using that $i\geq 1$. We recall we computed $\Lips \leq 2(\Rglobaltilde+\tilde{R})\Ltildegeneric$ and that we used $\stcvxpsi=O(1)$. In \cref{alg:accelerated_gconvex} we use $\stcvxpsi=1$. 

On the other hand the initial length of the search interval, which is the domain of definition of $\Gbinsearch[i]$ is at most $1$ since the interval is in $(0, 1)$. Recall we are denoting by $\lambda^\ast$ a value such that $\innp{\nabla \ftilted(\xI[i+1][\lambdanot{\lambda^\ast}]), \xI[i+1][\lambdanot{\lambda^\ast}] - \xI[i]}=0$ so $\Gbinsearch[i](\lambda^\ast) \leq 0$. Lipschitzness of $G$ implies that if $\Gbinsearch[i](\lambda^\ast) \leq 0$ then $\Gbinsearch[i](\lambda) \leq \hatepsilon[i]$ for 
\[
\lambda \in [\lambda^\ast-\frac{\hatepsilon[i]}{\hat{L}}, \lambda^\ast+\frac{\hatepsilon[i]}{\hat{L}}]\cap[\Gammahat[i][-1](1/\gammanparam), \Gammahat[i][-1](\gammapparam)].
\] 
If the extremal points, $\Gammahat[i][-1](1/\gammanparam),\Gammahat[i][-1](\gammapparam)$ did not satisfy \eqref{property_result_of_line_search_approximate}, then this interval is of length $\frac{2\hatepsilon[i]}{\hat{L}}$ and a point in such interval or another interval that is around another point $\bar{\lambda}^\ast$ that satisfies $\innp{\nabla \ftilted(xI[i+1][\bar{\lambda}^\ast]), \xI[i+1][\bar{\lambda}^\ast]-\xI[i]}=0$ can be found with a binary search in at most
\[
    O\left(\log\left(\frac{\hat{L}}{\hatepsilon[i]}\right)\right) \circled{1}[=] O\left(\log\left(\frac{\Ltildegeneric(\Rglobaltilde+\tilde{R})}{\gammanparam\hatepsilon[i]}\cdot i\right)\right)
\]
iterations, provided that at each step we can ensure we halve the size of the search interval. The bounds of the previous paragraph are applied in $\circled{1}$.The binary search can be done easily: we start with $[\Gammahat[i][-1](1/\gammanparam), \Gammahat[i][-1](\gammapparam)]$  and assume the extremes do not satisfy \eqref{property_result_of_line_search_approximate}, so the sign of $\innp{\nabla \ftilted(\xI[i+1][\lambdanot{\lambda}]), \xI[i+1][\lambdanot{\lambda}]-\xI[i]}$ is different for each extreme. Each iteration of the binary search queries the midpoint of the current working interval and if \eqref{property_result_of_line_search_approximate} is not satisfied, we keep the half of the interval such that the extremes keep having the sign of $\innp{\nabla \ftilted(\xI[i+1][\lambdanot{\lambda}]), \xI[i+1][\lambdanot{\lambda}]-\xI[i]}$ different from each other, ensuring that there is a point in which this expression evaluates to $0$ and thus keeping the invariant. We include the pseudocode of this binary search in \cref{alg:bin_search}.
\end{proof}

\begin{algorithm}[h!]
    \caption{BinaryLineSearch${(\xI[i], \zi[i], {\protect\ftilted}, {\protect\X}, \ai[i+1], {\protect\At[i]}, {\protect\epsilon}, {\protect\Ltildegeneric}, {\protect\gammanparam}, {\protect\gammapparam})}$}
    \label{alg:bin_search}
    \begin{algorithmic}[1]
        \REQUIRE Points $\xI[i]$, $\zi[i]$, function $\ftilted$, domain $\X$, learning rate $\ai[i+1]$, accumulated learning rate $\At[i]$, final target accuracy $\epsilon$, final number of iterations $\T$, smoothness constant $\Ltildegeneric$, constants $\gammanparam, \gammapparam$.
        Define $\hatepsilon[i] \gets (\At[\T]\epsilon)/(2(\T-1)\At[i])$ as in \cref{thm:accelerated_axgd_modified}, i.e., with $\At[\T]=\T(\T+1)\gammanparam[2]\gammapparam/4\Ltildegeneric$. $\Gammahat[i]$ defined as in \eqref{def:Gamma} and $\Gbinsearch[i]$ defined as in \eqref{def:G} i.e.
        \[
    \Gbinsearch[i](\lambda) \defi -\Gammahat[i](\lambda) \innp{\nabla \ftilted(\xI[i+1][\lambdanot{\lambda}]), \xI[i+1][\lambdanot{\lambda}]-\xI[i]}+ (\ftilted(\xI[i+1][\lambdanot{\lambda}])-\ftilted(\xI[i])),
        \]
        for $\xI[i+1][\lambdanot{\lambda}][\notilde]$ being the result of method \eqref{appendix_general_rule_modified_quasar_axgd} when $\hatgamma[i]  = \Gammahat[i](\lambda)$.
        \ENSURE $\lambda = \frac{\ai[i+1]/\gammanparam}{\At[i]\hatgamma[i] +\ai[i+1]/\gammanparam}$ for $\hatgamma[i] $ such that $\Gbinsearch[i](\Gammahat[i][-1](\hatgamma[i] )) \leq \hatepsilon[i]$.
        \IF{$\Gbinsearch[i](\Gammahat[i][-1](1/\gammanparam)) \leq \hatepsilon[i]$} $\lambda = \Gammahat[i][-1](1/\gammanparam)$
        \ELSIF{$\Gbinsearch[i](\Gammahat[i][-1](\gammapparam)) \leq \hatepsilon[i]$} $\lambda = \Gammahat[i][-1](\gammapparam)$
        \ELSE
            \State $\varl \gets \Gammahat[i][-1](1/\gammanparam)$
            \State $\varr \gets \Gammahat[i][-1](\gammapparam)$
            \State $\lambda \gets (\varl + \varr)/2$
            \WHILE{$\Gbinsearch[i](\lambda) > \hatepsilon[i]$}
                \IF{$\innp{\nabla \ftilted(\xI[i+1][\lambdanot{\lambda}]), \xI[i+1][\lambdanot{\lambda}]-\xI[i]} <0$} $\varr \gets \lambda$
                \ELSE $\ \varl \gets \lambda$
                \ENDIF
                \State $\lambda \gets (\varl + \varr)/2$
            \ENDWHILE
        \ENDIF
        \State \textbf{return} $\lambda$
\end{algorithmic}
\end{algorithm}

We proceed to prove \cref{thm:acceleration_quasiquasarconvexity}, which is an immediate consequence of the previous results.

\begin{proof}\textbf{of \cref{thm:acceleration_quasiquasarconvexity}.}\linkofproof{thm:acceleration_quasiquasarconvexity}
    The proof follows from \cref{thm:accelerated_axgd_modified}, provided that we can find $\hatgamma[i] $ satisfying \eqref{eq:approximate_multiplied_convexity}. \cref{lemma:binary_search} shows that this is possible after performing a logarithmic number of queries to the gradient oracle. Note that given our choice of $\hatepsilon[i]$, $\T$ and $\ai[i]$, the number of queries to the gradient oracle that \cref{lemma:binary_search} requires is no more than $O(\log(i\Ltildegeneric (\Rglobaltilde + \tilde{R}) /(\gammanparam\hatepsilon[i])))$ for any $i \leq \T$. So we find an $\epsilon$-minimizer of $\ftilted$ after $\bigotilde{\sqrt{\Ltildegeneric/(\gammanparam[2]\gammapparam\epsilon)}}$ queries to the gradient oracle.
    
\end{proof}

\begin{proof}\textbf{of \cref{thm:riemannian_acceleration}.}\linkofproof{thm:riemannian_acceleration}
    Given the function to optimize $\F:\Mk\to\R$ and the geodesic map $\h$, we define $\f=\F\circ \h^{-1}$. Using \cref{lemma:smoothness_of_transformed_function} we know that $\f$ has $\newtarget{def:smoothness_tilde_i_e_euclidean_coming_from_riemannian}{\Ltilde}$ Lipschitz gradients, with $\Ltilde =O(\L(\Rglobal+1))$. \cref{prop:bounding_hyperplane} proves that $\f$ satisfies tilted-convexity \eqref{eq:quasiquasarconvexity} for constants $\gamman$ and $\gammap$ depending on $\RR$ only. So \cref{thm:acceleration_quasiquasarconvexity} applies and the total number of queries to the oracle needed to obtain an $\epsilon$-minimizer of $\f$ is $\bigotilde{\sqrt{\Ltilde/\gamman[2]\gammap\epsilon}} = \bigotilde{\sqrt{\L(\Rglobal+1)/\epsilon}} =\bigotilde{\sqrt{\L\zetar/\epsilon}} $. The result follows, since $\f(\xI[\T]) - \f(\xast) = \F(\xI[\T][][])-\F(\xast[\notilde])$.
\end{proof}

We recall a few concepts that were assumed during \cref{sec:algorithm} to better interpret \cref{thm:riemannian_acceleration}. We work in the hyperbolic space, or in an open hemisphere. The aim is to minimize a smooth and g-convex function defined on any of these manifolds, or a submanifold of them. Starting from an arbitrary point $\xInit[\notilde]$, we perform constrained optimization over $\BR=\expon{\xInit[\notilde]}(\ball(0,\RR))$, for some $\RR > 0$. We assume $\F:\Mk\to\R$ is a differentiable function, $\BR\subset \Mk$, and $\Mk$ has constant sectional curvature $\K$. If $\K$ is positive, we restrict $\RR<\pi/(2\sqrt{\K})$ so $\BR$ is contained in an open hemisphere and it is uniquely geodesic. We define a geodesic map $\h:\Mk\to \MEucl$, where $\MEucl\subset\R^{\n}$ and define the function $f:h(\Mk)\to\R$ as $\f = \F\circ \h^{-1}$. We perform constrained optimization over this function $\f$ in $\X=\h(\BR)$ in an accelerated way,  up to constants and log factors, where the constants appear as an effect of the deformation of the geometry and depend on $\RR$ and $\K$ only.

{
\subsection{Reducing constants with an approximate ball optimization oracle}

In this section, we present our algorithm that shows that global optimization can be reduced to sequential optimization in Riemannian balls, which boosts convergence by reducing geometric constants in the algorithm. We use the strongly convex version of \cref{alg:accelerated_gconvex} in \cref{coroll:acceleration_st_g_convex} for the optimization in these balls, where $\newtarget{def:algorithm_1_strongly_convex_version}{\algsc}(\X_k, \pk[k-1], \F, \epsilonp)$ means the algorithm is run with the initial point $\pk[k-1]$ to optimize function $\F$ over the set $\X_k$ with accuracy $\epsilonp$. Recall that without loss of generality we assume $\K \in \{-1, 1\}$, cf. \cref{remark:rescaling_of_K}.

We start by showing that the iterates of \cref{alg:instance_of_riemacon} stay reasonably bounded, which is crucial in order to bound geometric penalties.

\begin{proposition}\label{prop:we_go_no_farther_than_2R}
    The iterates $\pk$ of \cref{alg:instance_of_riemacon} satisfy $\dist(\pk, \xastg) \leq 2\Rglobal$.
\end{proposition}

\begin{proof}
    We first show that the optimizer $\newtarget{def:optimizers_in_balls_in_boosting_alg}{\pkast}$ of $\F$ in the ball $\X_k$ is no farther than the center of $\X_k$ to $\xastg$, that is, $\dist(\pkast, \xastg) \leq \dist(\pk[k-1], \xastg)$. We assume $\xastg$ is not in the ball because otherwise the property holds trivially. The geodesic segment joining $\pkast$ and $\xastg$ does not contain any other point of the ball, since otherwise by strong convexity we would have that the function value of one such point would be lower than $\F(\pkast)$. This fact implies that the angle between $\exponinv{\pkast}(\xastg)$ and $\exponinv{\pkast}(\pk[k-1])$ is obtuse, and so $\circled{1}$ holds below and by using \eqref{eq:cosine_ineq_delta} we conclude $\dist(\pkast, \xastg) \leq \dist(\pk[k-1], \xastg)$:
\begin{align*}
 \begin{aligned}
     0 &\circled{1}[\geq] 2\innp{\exponinv{\pkast}(\xastg), \exponinv{\pkast}(\pk[k-1])} \geq \dist(\pkast, \xastg)^2 + \distorp[\dist(\pkast, \xastg)] \cdot \dist(\pkast, \pk[k-1])^2 - \dist(\pk[k-1], \xastg)^2 \\
     &\geq \dist(\pkast, \xastg)^2  - \dist(\pk[k-1], \xastg)^2.
   \end{aligned}
\end{align*}

    If instead of optimizing exactly in the ball we obtain a close approximation, the iterates do not get very far from $\xastg$. Indeed, by $\mu$-strong convexity, if $\pk$ is an $\epsilonp$-minimizer of $\F$ in $\X_k$, it holds that $\dist(\pkast, \pk) \leq \sqrt{\frac{2\epsilonp}{\mu}} \leq \frac{\Rglobal}{\TT}$, where we used the definition of $\epsilonp = \min\{\frac{\RR\epsilon}{4\Rglobal},\frac{\mu\Rglobal[2]}{2\TT^2}\}$ in the last inequality. Consequently, applying the non-expansiveness and this last inequality recursively, we obtain
    \[
        \dist(\pk[\TT], \xastg) \leq \dist(\pkast[\TT], \xastg) + \dist(\pkast[\TT], \pk[\TT]) \leq \dist(\pk[\TT-1], \xastg) + \frac{\Rglobal}{\TT} \leq \dots \leq \dist(\pk[0], \xastg) + \Rglobal \leq 2\Rglobal.
    \] 
\end{proof}

\begin{proof}\textbf{of \cref{thm:reduction_to_ball_opti}}\linkofproof{thm:reduction_to_ball_opti}. 
    If $\Rglobal \leq \RR= 1$, which is the case in which the condition in Line \ref{line:if_D_equal_R_one_single_Riemacon} of \cref{alg:instance_of_riemacon} is satisfied, then we just need to call \cref{alg:accelerated_gconvex} once in the corresponding ball $\ball(\xInit[\notilde], \Rglobal)$ and we obtain rates $\bigotilde{\sqrt{\frac{\L}{\mu}}}$. So from now on we assume $\RR < \Rglobal$. Let $\TT = \lceil \frac{2\Rglobal}{\RR} \ln(\frac{\L\Rglobal[2]}{\epsilon}) \rceil$ and let $\epsilonp = \min\{\frac{\RR\epsilon}{4\Rglobal},\frac{\mu\Rglobal[2]}{2\TT^2}\}$. Since every time we call \cref{alg:accelerated_gconvex} we do it over a ball of radius $\RR$, then Line \ref{line:alg_sc_as_subroutine} of \cref{alg:instance_of_riemacon} takes $\bigotilde{\sqrt{\frac{\L}{\mu}}\log(\frac{1}{\epsilonp})}$ gradient oracle calls to optimize in the ball $\X_k$ of radius $\RR = 1$ up to precision $\epsilonp$, for any $k$.
Recall that we denote the global optimizer of $\F$ by $\xastg$. Define the g-convex combination 
\[
    \hat{p}_{k} = \expon{\pk[k-1]}\left((1-\frac{\RR}{2\Rglobal})\exponinv{\pk[k-1]}(\pk[k-1])+\frac{\RR}{2\Rglobal}\exponinv{\pk[k-1]}(\xastg)\right).
\] 
Since $\X_k$ is a ball of radius $\RR$ and by \cref{prop:we_go_no_farther_than_2R}, it is $\dist(\pk, \xastg) \leq 2\Rglobal$, we have $\hat{p}_{k} \in \X_k$. Consequently, we have
\[
    \F(\pk) \circled{1}[\leq] \F(\hat{p}_{k}) + \epsilonp \circled{2}[\leq] (1-\frac{\RR}{2\Rglobal})\F(\pk[k-1]) + \frac{\RR}{2\Rglobal} \F(\xastg) + \epsilonp,
\] 
where $\circled{1}$ is due to the guarantees of the optimization in the ball and the fact that $\hat{p}_k \in \X_k$, $\circled{2}$ holds due to g-convexity. Subtracting $\F(\xastg)$ in both sides and rearranging, we obtain 
\[
\F(\pk)-\F(\xastg) \leq (1-\frac{\RR}{2\Rglobal})(\F(\pk[k-1]) - \F(\xastg)) + \epsilonp.
\] 
Applying this inequality recursively, we obtain
\begin{align*}
 \begin{aligned}
     \F(\pk[\TT])-\F(\xastg) &\leq (1-\frac{\RR}{2\Rglobal})^{\TT}(\F(\xInit[\notilde]) - \F(\xastg)) + \epsilonp\sum_{i=0}^{\TT-1} (1-\frac{\RR}{2\Rglobal})^i \\
     &\circled{1}[\leq] \exp(-\frac{\TT\RR}{2\Rglobal})\frac{\L\Rglobal[2]}{2} + \frac{2\Rglobal}{\RR}\epsilonp\\
     &\leq \frac{\epsilon}{2} + \frac{\epsilon}{2} = \epsilon.
   \end{aligned}
\end{align*}
    Above, we used $1-x \leq \exp(-x)$, we used smoothness to bound $\F(\xInit[\notilde]) - \F(\xastg) \leq \frac{Ld(\xInit[\notilde], \xastg)^2}{2}$, we bounded $\sum_{i=0}^{\TT-1} (1-\frac{\RR}{2\Rglobal})^i \leq \sum_{i=0}^{\infty} (1-\frac{\RR}{2\Rglobal})^i = \frac{2\Rglobal}{\RR}$ and we used the values of $\epsilonp$ and $\TT$.  Finally, we compute the complexity of this algorithm. We have $\TT$ iterations taking $\bigotilde{\sqrt{\frac{\L(\Rglobal+1)}{\mu}}}$ gradient oracle queries each. Using the value of $\TT$ and $\RR$, we obtain that in total, we call the gradient oracle $\bigotilde{\frac{\Rglobal}{\RR} \sqrt{\frac{\L(\Rglobal+1)}{\mu}}} = \bigotilde{\zetar^{3/2}\sqrt{\frac{\L}{\mu}}}$ times, where $\zetar \defi \distorn[\Rglobal] \in [\max\{\Rglobal, 1\},\Rglobal+1]$, and $\distorn[\Rglobal]$ was defined in \eqref{eq:def_distortion_neg}. 

    Using the reduction in \cref{thm:reduction_to_g_st_convex} as in \cref{example:application_of_reduction_to_st_convex}, we obtain an algorith for optimizing the g-convex case that runs in $\bigotilde{\zetar^{3/2}\sqrt{\frac{\zetar}{\deltar} + \frac{L\Rglobal[2]}{\deltar\epsilon}}}$, where $\deltar \defi \distorp[\Rglobal]$, and $\distorp[\Rglobal]$ was defined in \eqref{eq:def_distortion_pos}.
\end{proof}

} 

\subsection{Auxiliary lemmas}
The following are classical lemmas of convex optimization that we used in this section and that we add for completeness.

\begin{fact}\label{grad_of_fenchel_dual}
    Let $\psi:\Q\to\R$ be a differentiable strongly-convex function. Then
    \[
        \nabla \psi^{\astfenchel}(\tilde{z}) = \argmin_{\tilde{x}\in \Q}\{\innp{-\tilde{z}, \tilde{x}}+\psi(\tilde{x})\}.
    \] 
\end{fact}
See \citep{bertsekas2003convex} for a proof. Note that by the definition of Bregman divergence $\argmin_{\tilde{x}\in \Q}\{\innp{\tilde{y}-\nabla \psi(\xInit), \tilde{x}} + \psi(\tilde{x})\} = \argmin_{\tilde{x}\in \Q}\{\innp{\tilde{y}, \tilde{x}} + \breg[\psi](\tilde{x}, \xInit)\}$.

\begin{lemma}[Triangle equality of Bregman Divergences] \label{prop:triangle_inequality_of_bregman_div}
    For all $\tilde{x}, \tilde{y}, \tilde{z} \in \Q$ we have 
    \[
        \breg[\psi^{\astfenchel}] (\tilde{x}, \tilde{y}) = \breg[\psi^{\astfenchel}] (\tilde{z}, \tilde{y})  + \breg[\psi^{\astfenchel}] (\tilde{x}, \tilde{z}) + \innp{\nabla \psi^{\astfenchel}(\tilde{z})-\nabla \psi^{\astfenchel}(\tilde{y}), \tilde{x}-\tilde{z}}.
    \] 
\end{lemma}

\begin{proof}
\begin{align*}
    \begin{aligned}
        \breg[\psi^{\astfenchel}]& (\tilde{z}, \tilde{y}) + \breg[\psi^{\astfenchel}] (\tilde{x}, \tilde{z}) + \innp{\nabla \psi^{\astfenchel}(\tilde{z})-\nabla \psi^{\astfenchel}(\tilde{y}), \tilde{x}-\tilde{z}} \\
        &= (\psi^{\astfenchel}(\tilde{z})-\psi^{\astfenchel}(\tilde{y}) -\innp{\nabla \psi^{\astfenchel}(\tilde{y}), \tilde{z}-\tilde{y}}) \\
        & \quad + (\psi^{\astfenchel}(\tilde{x})-\psi^{\astfenchel}(\tilde{z}) -\innp{\nabla \psi^{\astfenchel}(\tilde{z}), \tilde{x}-\tilde{z}}) \\
        &\quad + \innp{\nabla \psi^{\astfenchel}(\tilde{z})-\nabla \psi^{\astfenchel}(\tilde{y}), \tilde{x}-\tilde{z}}  \\
        & = \psi^{\astfenchel}(\tilde{x}) -\psi^{\astfenchel}(\tilde{y}) -\innp{\nabla \psi^{\astfenchel}(\tilde{y}), \tilde{z}-\tilde{y}}+ \innp{-\nabla \psi^{\astfenchel}(\tilde{y}), \tilde{x}-\tilde{z}} \\
        & =  \breg[\psi^{\astfenchel}] (\tilde{x}, \tilde{y}).
    \end{aligned}
    \end{align*}
\end{proof}   

\begin{lemma}\label{prop:bounding_breg_div_by_norm_of_difference}
    Given a $\stcvxpsi$-strongly convex function $\psi(\cdot)$ the following holds:
    \[
        \breg[\psi^{\astfenchel}](\tilde{z}_1, \tilde{z}_2) \geq \frac{\stcvxpsi}{2}\norm{\nabla \psi^{\astfenchel}(\tilde{z}_1)-\nabla \psi^{\astfenchel}(\tilde{z}_2)}^2.
    \] 
\end{lemma}

\begin{proof}
    Using the first order optimality condition of the Fenchel dual and \eqref{grad_of_fenchel_dual} we obtain
    \[
        \innp{\nabla\psi(\nabla\psi^{\astfenchel}(\tilde{z}_1))-\tilde{z}_1, \nabla \psi^{\astfenchel}(\tilde{z}_2)-\nabla \psi^{\astfenchel}(\tilde{z}_1)} \geq 0
    \] 
    Using $\stcvxpsi$-strong convexity of $\psi$ and the previous inequality we have
    \begin{align*}
    \begin{aligned}
        \breg[\psi^{\astfenchel}](\tilde{z}_1, \tilde{z}_2) &= \psi(\nabla \psi^{\astfenchel}(\tilde{z}_2)) -\psi(\nabla\psi^{\astfenchel}(\tilde{z}_1)) - \innp{\tilde{z}_1, \nabla \psi^{\astfenchel}(\tilde{z}_2)-\nabla \psi^{\astfenchel}(\tilde{z}_1)} \\
        &\geq \frac{\stcvxpsi}{2} \norm{\nabla \psi^{\astfenchel}(\tilde{z}_1)-\nabla \psi^{\astfenchel}(\tilde{z}_2)}^2 + \innp{\nabla \psi(\nabla\psi^{\astfenchel}(\tilde{z}_1))-\tilde{z}_1, \nabla\psi^{\astfenchel}(\tilde{z}_2)-\nabla\psi^{\astfenchel}(\tilde{z}_1)} \\
        &\geq \frac{\stcvxpsi}{2} \norm{\nabla \psi^{\astfenchel}(\tilde{z}_1)-\nabla \psi^{\astfenchel}(\tilde{z}_2)}^2.
    \end{aligned}
    \end{align*}
\end{proof}   

\section[Reductions. Proofs of results in Section \ref{sec:reductions}]{Reductions. Proofs of results in \cref{sec:reductions}} \label{app:reductions}

\begin{proof}\textbf{of \cref{thm:reduction_to_g_convex}.}\linkofproof{thm:reduction_to_g_convex}
    Let $\Algns$ be the algorithm in the statement of the theorem. By strong g-convexity of $\F$ and the assumptions on $\Algns$ we have that $\hat{x}_T$, the point computed by $\Algns$, satisfies
    \[
        \frac{\mu}{2}\dist(\hat{x}_T, \xast[\notilde])^2 \leq \F(\hat{x}_T)-\F(\xast[\notilde]) \leq \frac{\mu}{2}\frac{\RR^2}{2},
    \]
    after $T=\timens(\L, \mu, \RR)$ queries to the gradient oracle. This implies $\dist(\hat{x}_T, \xast[\notilde])^2 \leq \RR^2/2$. We perform this process $r\defi \lceil\log (\mu \RR^2/\epsilon)-1\rceil$ times. We use the previous output as input for the next round. The distance bound to $\xast[\notilde]$ can be updated to a lower value. We denote $R_i$ the distance bound from the input to $\xast[\notilde]$ at stage $i$ and we set its value to $R_i = R_{i-1}/\sqrt{2}$, for $R_1 = \RR$. Thus, after $r$ stages we obtain a point $\hat{x}_T^r$ that satisfies
    \[
        \F(\hat{x}_T^r)-\F(\xast[\notilde]) \leq \frac{\mu \cdot R_{r}^2 }{4} = \frac{\mu \cdot R_1^2}{4 \cdot 2^{r-1}} \leq \epsilon.
    \]
    And the total running time is $\timens(\L, \mu, \RR)\cdot r = O(\timens(\L, \mu, \RR)\log(\mu\RR^2/\epsilon))$.
\end{proof}

\begin{proof}\textbf{of \cref{coroll:acceleration_st_g_convex}.}\linkofproof{coroll:acceleration_st_g_convex}
    We can assume without loss of generality $\K\in\{-1, 1\}$ as we did in \cref{sec:algorithm}. Let $\RR$ be an upper bound on the distance between the initial point $\xInit[\notilde]$ and an optimizer $\xast[\notilde]$, i.e., $\dist(\xInit[\notilde], \xast[\notilde])\leq \RR$. Note that $\norm{\xInit-\xast}/\RR$ is bounded by a constant depending on $\RR$ by \cref{lemma:deformations}.a. Note that $\gamman$ and $\gammap$ are constants depending on $\RR$ by \cref{prop:bounding_hyperplane}. As any g-strongly convex function is g-convex, by using \cref{thm:accelerated_axgd_modified} and \cref{lemma:binary_search} with $\epsilon=\mu \frac{\RR^2}{4}$ we obtain that \cref{alg:accelerated_gconvex} obtains a $\mu \frac{\RR^2}{4}$-minimizer in at most 
\[
    T = \bigotildel{\frac{\norm{\xInit-\xast}}{\RR}\sqrt{\frac{4L(\Rglobal+1)}{\mu\gamman[2]\gammap }}} = \bigotildel{\sqrt{\L(\Rglobal+1)/\mu}}
\] 
    queries to the gradient oracle. Subsequent stages, i.e., calls to \cref{alg:accelerated_gconvex}, use the point computed at the previous stage as its input. The distance bound to $\xast[\notilde]$ is updated, following the proof of \cref{thm:reduction_to_g_convex}. Because the constant depending on $\RR$ in the running time of the subroutine decreases when $\RR$ has a lower value, subsequent stages need a time which is $\bigotilde{(\L(\Rglobal+1)/\mu)^{1/2}}$ as well. So we satisfy the assumption of \cref{thm:reduction_to_g_convex} for $\timens(\L, \mu, \RR)=\bigotilde{(\L(\Rglobal+1)/\mu)^{1/2}}$. We conclude that given $\epsilon>0$ and running \cref{alg:accelerated_gconvex} in stages, we obtain an $\epsilon$-minimizer of $\F$ in 
    \[
        \bigotilde{\sqrt{\L(\Rglobal+1)/\mu}\log(\mu\RR^2/\epsilon)} = \bigotilde{\sqrt{\L(\Rglobal+1)/\mu}}
    \] 
    queries to the gradient oracle.

    Note that each time we call \cref{alg:accelerated_gconvex} we recenter the geodesic map. In order to perform the method with these recentering steps, we need the function $\F$ to be defined over at least $\expon{\xInit[\notilde]}(\ball(0, \RR\cdot(1+2^{-1/2})))$, since subsequent centers are only guaranteed to be $\leq \RR/\sqrt{2}$ close to $\xast[\notilde]$, and they could get slightly farther than $\RR$ from $\xInit[\notilde]$. But they are no farther than $\RR+\RR/\sqrt{2}$ since $\dist(\xInit[\notilde], \hat{x}_T^i) \leq \dist(\xInit[\notilde], \xast[\notilde])+ \dist(\xast[\notilde], \hat{x}_T^i) \leq \RR+\RR/\sqrt{2}$, where $\hat{x}_T^i$ is the center at stage $i$, and where $i>1$.
\end{proof}

\subsection[Proof of Theorem \ref{thm:reduction_to_g_st_convex}]{Proof of \cref{thm:reduction_to_g_st_convex}}

We provide here the full statement of the theorem.

\begin{theorem}\label{thm:reduction_to_g_st_convex_full_theorem}
    Let $\M$ be a Riemannian manifold of bounded sectional curvature, let $\F:\M\to\R$ be a differentiable function with a point $\xast[\notilde]\in\M$ such that $\nabla \F(\xast[\notilde])=0$. Let $\xInit[\notilde]$ be a starting point such that $\dist(\xInit[\notilde],\xast[\notilde]) \leq \RR$ and let $\Delta$ satisfy $\F(\xInit[\notilde])-\F(\xast[\notilde]) \leq \Delta $. Assume $\F$ is an $\L$-smooth and g-convex function in $\X \defi\expon{\xInit[\notilde]}(\ball(0, \RR))$ and that $\expon{\xInit[\notilde]}$ is a diffeomorphism when restricted to $\ball(0,\RR)$. Further, we assume access to an algorithm $\Alg$ that given an $\L$-smooth and $\mu$-strongly g-convex function $\hat{F}:\M\to\R$ in $\X$, with minimizer in $\X$, and any initial point $\hat{x}_0\in\M$, produces a point $\hat{x}\in\X$ in time $\hat{T}=\time(\L,\mu,\M, \RR)$ satisfying $\hat{F}(\hat{x})-\min_{x\in\M}\hat{F}(x) \leq (\hat{F}(\hat{x}_0) - \min_{x\in\M}\hat{F}(x))/4$. Let $T=\lceil\log_2(\Delta/\epsilon)\rceil+1$. Then, we can compute an $\epsilon$-minimizer in time $\sum_{t=0}^{T-1}\time(\L+2^{-t}\Delta\distorn/\RR^2, 2^{-t}\Delta\distorp/\RR^2, \M, \RR)$, where $\distorp$ and  $\distorn$ are constants that depend on $\RR$ and the bounds on the sectional curvature of $\M$.
\end{theorem}

    The algorithm is the following. We successively regularize the function with strongly g-convex regularizers in this way $\newtarget{def:regularized_F_for_reduction}{\FReg[\mui{i}]}(x) \defi \F(x)+ \frac{\mui{i}}{2}\dist(x, \xInit[\notilde])^2$ 
    for $i \geq 0$. For each $i \geq 0$, we use the algorithm $\Alg$ on the function $\FReg[\mui{i}]$ 
    for the time in the statement of the theorem and obtain a point $\hat{x}_{i+1}$, starting from point $\hat{x}_i$, where $\hat{x}_0=\xInit[\notilde]$. The regularizers are decreased exponentially $\newtarget{def:regularization_in_reduction}{\mui{i+1}} = \mui{i}/2$ 
    from $\mui{0}=\Delta/\RR^2$, until we reach roughly $\mui{T}=\epsilon/\RR^2$, see below for the precise value. Let's see how this algorithm works. We first state the following fact, that says that indeed $\frac{\mui{i}}{2}\dist(x, \xInit[\notilde])^2$ is a strongly g-convex  regularizer. Recall that $\expon{\xInit[\notilde]}(\ball(0,\RR))\subset \M$. We define the following quantities
    \begin{equation}\label{eq:def_distortion_pos}
    \newtarget{def:distortion_K_pos}{\distorp} \defi 
        \begin{cases} 
            {1} & {\text{if } K_{\max } \leq 0} \\
            {\sqrt{K_{\max }} \RR \cot (\sqrt{K_{\max }} \RR)} & {\text{if }  K_{\max }>0}
        \end{cases}
    \end{equation}

    \begin{equation}\label{eq:def_distortion_neg}
    \newtarget{def:distortion_K_neg}{\distorn} \defi 
        \begin{cases} 
            {\sqrt{-K_{\min }} \RR \operatorname{coth}(\sqrt{-K_{\min }} \RR)} & {\text{if } K_{\min }<0} \\
            {1} & {\text{if } K_{\min } \geq 0}
        \end{cases}
    \end{equation}
Here $K_{\max}$ and $K_{\min}$ are the upper and lower bounds on the sectional curvature of the manifold $\M$.

\begin{fact}\label{strong_convexity_and_smoothness_of_ell_2}
    Let $\M$ be a manifold with sectional curvature bounded below and above by $K_{\min}$ and $K_{\max}$, respectively. For a point $\xInit[\notilde]\in\M$, assume $\expon{\xInit[\notilde]}$ is a diffeomorphism when restricted to $\ball(0,\RR)$. The function $f:\M\to\R$ defined as $\ftilted(x) = \frac{1}{2}\dist(x, \xInit[\notilde])^2$ is $\distorp$-g-strongly convex and $\distorn$-smooth in $\expon{\xInit[\notilde]}(\ball(0,\RR))$.
\end{fact}

The result regarding strong convexity can be found, for instance, in \citep{alimisis2019continuous} and it is a direct consequence of the following inequality, which can also be found in \citep{alimisis2019continuous}:
\begin{equation}\label{eq:cosine_ineq_delta}
    \dist(y,\xInit[\notilde])^2 \geq \dist(x,\xInit[\notilde])^2-2\innp{\exponinv{x}(\xInit[\notilde]), y\riemMinus x} + \distorp \dist(x,y)^2,
\end{equation}
along with the fact that $\operatorname{grad} \ftilted(x) = -\exponinv{x}(\xInit[\notilde])$. The result regarding smoothness is, similarly, obtained from the following inequality:
\begin{equation}\label{eq:cosine_ineq_zeta}
    \dist(y,\xInit[\notilde])^2 \leq \dist(x,\xInit[\notilde])^2-2\innp{\exponinv{x}(\xInit[\notilde]), y\riemMinus x} + \distorn \dist(x,y)^2,
\end{equation}
which can be found in \citep{zhang2016first} (Lemma $6$). These inequalities are tight in spaces of constant sectional curvature. Alternatively, one can derive these inequalities from upper and lower bounds on the Hessian of $\ftilted(x) = \frac{1}{2}\dist(x, \xInit[\notilde])^2$, as it was done in \citep{lezcano2020curvature} (Theorem 3.15).

We prove now that the regularization makes the minimum be closer to $\xInit[\notilde]$, so the assumption of the theorem on $\hat{F}$ holds for the functions we use. Define $x_{i+1}$ as the minimizer of $\FReg[\mui{i}]$.

\begin{lemma}\label{claim:reduction_sc}
We have $\dist(x_{i+1},\xInit[\notilde]) \leq \dist(\xast[\notilde], \xInit[\notilde])$.
\end{lemma}

\begin{proof}
    By the fact that $x_{i+1}$ is the minimizer of $\FReg[\mui{i}]$ we have $\FReg[\mui{i}](x_{i+1})-\FReg[\mui{i}](\xast[\notilde]) \leq 0$. Note that by g-strong convexity, equality only holds if $x_{i+1} = \xast[\notilde]$ which only happens if $\xInit[\notilde]=x_{i+1}=\xast[\notilde]$. By using the definition of $\FReg[\mui{i}](x) = \F(x) + \frac{\mui{i}}{2}\dist(x,\xInit[\notilde])^2$ we have:
\begin{align*} 
 \begin{aligned}
 \F(x_{i+1}) & + \frac{\mui{i}}{2} \dist(x_{i+1},\xInit[\notilde])^2 - \F(\xast[\notilde])-\frac{\mui{i}}{2}\dist(\xast[\notilde],\xInit[\notilde])^2 \leq 0 \\
    \Rightarrow  \ \ \ & \dist(x_{i+1},\xInit[\notilde]) \leq \dist(\xast[\notilde],\xInit[\notilde]),
   \end{aligned}
\end{align*}
    where in the last step we used the fact $\F(x_{i+1}) - \F(\xast[\notilde]) \geq 0$ that holds because $\xast[\notilde]$ is the minimizer of $\F$.
\end{proof}

We note that previous techniques proved and used the fact that $\dist(x_{i+1}, \xast[\notilde]) \leq \dist(\xInit[\notilde], \xast[\notilde])$ instead \citep{allen2016optimal}. But crucially, we need our former lemma in order to prove the bound for our non-Euclidean case. Our variant can be applied to \citep{allen2016optimal} to decrease the constants of their Euclidean reduction. Now we are ready to prove the theorem.

\begin{proof}\textbf{of \cref{thm:reduction_to_g_st_convex}.}\linkofproof{thm:reduction_to_g_st_convex}
    We recall the definitions above. $\FReg[\mui{i}](x) = \F(x)+\frac{\mui{i}}{2}\dist(x,\xInit[\notilde])^2$. We start with $\hat{x}_0=\xInit[\notilde]$ and compute $\hat{x}_{i+1}$ using algorithm $\Alg$ with starting point $\hat{x}_i$ and function $\FReg[\mui{i}]$ for time $\time(L^{(i)}, \mu^{(i)}, \M, \RR)$, where $L^{(i)}$ and $\mu^{(i)}$ are the smoothness and strong g-convexity parameters of $\FReg[\mui{i}]$. We denote by $x_{i+1}$ the minimizer of $\FReg[\mui{i}]$. We pick $\mui{i}=\mui{i-1}/2$ and we will choose later the value of $\mui{0}$ and the total number of stages. By the assumption of the theorem on $\Alg$, we have that 
    \begin{equation} \label{eq:aux_eq_reduction_to_g_st_convex}
        \FReg[\mui{i}](\hat{x}_{i+1})-\min_{x\in\M} \FReg[\mui{i}](x) = \FReg[\mui{i}](\hat{x}_{i+1})- \FReg[\mui{i}](x_{i+1}) \leq \frac{\FReg[\mui{i}](\hat{x}_{i})- \FReg[\mui{i}](x_{i+1})}{4}.
    \end{equation}

    Define $\newtarget{def:gap_for_reduction_tost_g_convex}{\Di[i]} \defi \FReg[\mui{i}]\left(\hat{x}_{i}\right)-\FReg[\mui{i}]\left(x_{i+1}\right)$ to be the initial objective distance to the minimum on function $\FReg[\mui{i}]$ before we call $\Alg$ for the $(i+1)$-th time. At the beginning, we have the upper bound $\Di[0]=\FReg[\mui{0}](\hat{x}_{0})- \min_{x} \FReg[\mui{0}](x) \leq \F(x_{0})-\F(\xast[\notilde])$. For each stage $i \geq 1$, we compute that
\begin{align*}
 \begin{aligned}
     \quad& \Di[i] = \FReg[\mui{i}]\left(\hat{x}_{i}\right)-\FReg[\mui{i}]\left(x_{i+1}\right) \\
    &\circled{1}[=] \FReg[\mui{i-1}]\left(\hat{x}_{i}\right)-\frac{\mui{i-1}-\mui{i}}{2}\dist(\xInit[\notilde],\hat{x}_{i})^{2}-\FReg[\mui{i-1}]\left(x_{i+1}\right)+\frac{\mui{i-1}-\mui{i}}{2}\dist(\xInit[\notilde],x_{i+1})^{2} \\
    &\circled{2}[\leq] \FReg[\mui{i-1}]\left(\hat{x}_{i}\right)-\FReg[\mui{i-1}]\left(x_{i}\right)+\frac{\mui{i-1}-\mui{i}}{2}\dist(\xInit[\notilde],x_{i+1})^{2} \\
     &\circled{3}[\leq] \frac{\Di[i-1]}{4} + \frac{\mui{i}}{2}\dist(\xInit[\notilde],x_{i+1})^{2} \circled{4}[\leq] \frac{\Di[i-1]}{4}+\frac{\mui{i}}{2}\dist(\xInit[\notilde],\xast[\notilde])^{2}.
   \end{aligned}
\end{align*}
    Above, $\circled{1}$ follows from the definition of $\FReg[\mui{i}](\cdot)$ and $\FReg[\mui{i-1}](\cdot)$; $\circled{2}$ follows from the fact that $x_{i}$ is the minimizer of $\FReg[\mui{i-1}](\cdot)$. And we dropped the negative term $-(\mui{i-1}-\mui{i})\dist(\xInit[\notilde], \hat{x}_i)/2$. $\circled{3}$ follows from the definition of $\Di[i-1]$, the assumption on $\Alg$, and the choice $\mui{i} = \mui{i-1}/2$ for $i\geq 1$; and $\circled{4}$ follows from \cref{claim:reduction_sc}. 
Now applying the above inequality recursively, we have
    \begin{equation}\label{eq:aux_eq_bounding_D_T}
        \Di[T] \leq \frac{\Di[0]}{4^T} + \dist(\xInit[\notilde],\xast[\notilde])^2 \cdot (\frac{\mui{T}}{2} + \frac{\mui{T-1}}{8} + \cdots) \leq \frac{\F(\xInit[\notilde])-\F(\xast[\notilde])}{4^T} +  \mui{T} \cdot \dist(\xInit[\notilde],\xast[\notilde])^2.
\end{equation}
We have used the choice $\mui{i} = \mui{i-1}/2$ for the second inequality. Lastly, we can prove that $\hat{x}_T$, the last point computed, satisfies 
\begin{align*}
 \begin{aligned}
 \F(\hat{x}_{T})-\F(\xast[\notilde]) &\circled{1}[\leq] \FReg[\mui{T}](\hat{x}_{T})-\FReg[\mui{T}](\xast[\notilde])+\frac{\mui{T}}{2}\dist(\xInit[\notilde],\xast[\notilde])^{2} \\
         & \circled{2}[\leq] \FReg[\mui{T}](\hat{x}_{T})-\FReg[\mui{T}](x_{T+1}) +\frac{\mui{T}}{2}\dist(\xInit[\notilde],\xast[\notilde])^{2} \\ 
         &\circled{3}[=] \Di[T]+\frac{\mui{T}}{2}\dist(\xInit[\notilde],\xast[\notilde])^{2} \\
         &\circled{4}[\leq] \frac{\F(\xInit[\notilde])-\F(\xast[\notilde])}{4^{T}}+ \frac{3\mui{T}}{2} \dist(\xInit[\notilde],\xast[\notilde])^{2}.
   \end{aligned}
\end{align*}

    We use the definition of $\FReg[\mui{T}]$ in $\circled{1}$ and drop $-\frac{\mui{T}}{2}\dist(\xInit[\notilde],\hat{x}_T)^2$. In $\circled{2}$ we use the fact that $x_{T+1}$ is the minimizer of $\FReg[\mui{T}]$. The definition of $\Di[T]$ is used in $\circled{3}$. We use inequality \eqref{eq:aux_eq_bounding_D_T} for step $\circled{4}$. Recall the assumption of the theorem $\F(\xInit[\notilde])-\F(\xast[\notilde])\leq\Delta$. Finally, by choosing $T = \lceil \log_2(\Delta/\epsilon)\rceil+1$ and $\mui{0} = \Delta/\RR^2$ we obtain that the point $\hat{x}_T$ satisfies 
    \[
    \F(\hat{x}_T)-\F(\xast[\notilde]) \leq \frac{\F(\xInit[\notilde])-\F(\xast[\notilde])}{4\Delta/\epsilon}+ \frac{3\mui{0}}{8\Delta/\epsilon} \dist(\xInit[\notilde],\xast[\notilde])^{2} \leq \frac{\epsilon}{4}+\frac{3\epsilon}{8}<\epsilon,
    \] 
    and can be computed in time $\sum_{t=0}^{T-1} \time(\L+2^{-t}\mui{0} \distorn, 2^{-t}\mui{0} \distorp, \M, \RR)$, since by \cref{strong_convexity_and_smoothness_of_ell_2} the function $\FReg[\mui{t}]$ is $\L+2^{-t}\mui{0} \distorn$ smooth and $2^{-t}\mui{0} \distorp$ g-strongly convex.
\end{proof}

\subsection[Example \ref{example:application_of_reduction_to_st_convex}]{\cref{example:application_of_reduction_to_st_convex}}\linkofproof{example:application_of_reduction_to_st_convex}
We use the algorithm in \cref{coroll:acceleration_st_g_convex} as the algorithm $\Alg$ of the reduction of \cref{thm:reduction_to_g_st_convex}. Given a manifold under consideration $\Mk$, the assumption on $\Alg$ is satisfied for $\time(\L, \mu, \Mk, \RR)=\bigotilde{\sqrt{\L(\Rglobal+1)/\mu}}$. Indeed, if $\Delta$ is a bound on the gap $\hat{F}(\xInit[\notilde])-\hat{F}(\xast[\notilde]) = \hat{F}(\xInit[\notilde])-\min_{x\in\expon{\xInit[\notilde]}(\ball(0, \RR))}\hat{F}(x)$ for some $\mu$-strongly g-convex $\hat{F}$, then we know that $\dist(\xInit[\notilde], \xast[\notilde])^2 \leq \frac{2\Delta}{\mu}$ by $\mu$-strong g-convexity. By calling the algorithm in \cref{coroll:acceleration_st_g_convex} with $\epsilon = \frac{\Delta}{4}$ we require a time that is 
\begin{align*}
 \begin{aligned}
     \bigo{ \sqrt{\L(\Rglobal+1)/\mu}\log(\mu \cdot \dist(\xInit[\notilde], \xast[\notilde])^2/(\Delta/4))}&= \bigotilde{ \sqrt{\L(\Rglobal+1)/\mu}\log(\mu \cdot (2\Delta/\mu)/(\Delta/4)) } \\
     &= \bigotilde{ \sqrt{\L(\Rglobal+1)/\mu} }.
   \end{aligned}
\end{align*}

Let $T = \lceil \log_2(\Delta/\epsilon)\rceil+1$. The reduction of \cref{thm:reduction_to_g_st_convex} gives an algorithm with rates
\begin{align*}
 \begin{aligned}
     \sum_{t=0}^{T-1}& \time(\L+2^{-t}\mui{0} \distorn, 2^{-t}\mui{0} \distorp, \Mk, \RR)   \\
     &\circled{1}[=] \bigotildel{\sum_{t=0}^{T-1}\sqrt{\frac{\distorn}{\distorp}+\frac{\L}{2^{-t}\distorp\Delta/\RR^2}}\sqrt{\Rglobal+1}}\\
     &\circled{2}[=] \bigotildel{\left(\sqrt{\frac{\distorn}{\distorp}}\log(\Delta/\epsilon)+\sum_{t=0}^{T-1}\sqrt{\frac{\L}{2^{-t}\distorp\Delta/\RR^2}}\right)\sqrt{\Rglobal+1}} \\
     &\circled{3}[=] \bigotildel{\left(\sqrt{\frac{\distorn}{\distorp}}\log(\Delta/\epsilon)+\sqrt{\frac{\L}{\distorp\epsilon}}\right)\sqrt{\Rglobal+1}} \\
     &\circled{4}[=] \bigotilde{\sqrt{\L(\Rglobal+1)/\epsilon}}
   \end{aligned}
\end{align*}
In $\circled{1}$ we write down the definition and use the value $\mui{0}=\Delta/\RR^2$. In $\circled{2}$ we have used Minkowski's inequality $\sqrt{a+b} \leq \sqrt{a} + \sqrt{b}$. We added up the first group of summands. For the $\log$ factor, we upper bounded $\L/(2^{-t}\distorp\Delta/\RR^2) = O(\L/\distorp\epsilon)$, for $t< T$. In $\circled{3}$ we used the fact that $\sqrt{1/\epsilon} + \sqrt{1/2\epsilon} + \cdots = O(\sqrt{1/\epsilon})$, along with the fact $\epsilon/2R^2 \leq 2^{-(T-1)}\mui{0}\leq \epsilon/\RR^2$. Note that by $\L$-smoothness and the diameter being $2R$, we have $\Delta \leq 2LR^2$ so $\sqrt{\distorn/\distorp}\log(\Delta/\epsilon) = \bigotilde{1}$. We applied this in $\circled{4}$.

\section{Geometric results. Proofs of Lemmas \ref{lemma:deformations}, \ref{prop:bounding_hyperplane} and \ref{lemma:smoothness_of_transformed_function}} \label{app:geometric_results}
In this section we prove the lemmas that take into account the deformations of the geometry and the geodesic map $\h$ to obtain relationships between $\F$ and $\f$. Namely \cref{lemma:deformations}, \cref{prop:bounding_hyperplane} and \cref{lemma:smoothness_of_transformed_function}. First, we recall the characterizations of the geodesic map and some consequences. Then in \cref{app:sec_distance_deformation}, \cref{app:sec_angle_deformation} and \cref{sec:app_gradient_deformation_and_smoothness}, we prove the results related to distances angles and gradient deformations, respectively. That is, each of the three parts of \cref{lemma:deformations}. In \cref{sec:app_gradient_deformation_and_smoothness} we also prove \cref{lemma:smoothness_of_transformed_function}, which comes naturally after the proof of \cref{lemma:deformations}.c. In \cref{app:sec_proof_of_bounding_hyperplane} we prove \cref{prop:bounding_hyperplane} and finish with a proof on lower bounds for the condition number of strongly g-convex functions and an intuitive comment on its implications. 

Before this, we note that we can assume without loss of generality that the curvature of our manifolds of interest can be taken to be $\K\in\{1, -1\}$. One can see that the final rates depend on $\K$ through $\RR$, $\L$ and $\mu$.

\begin{remark}\label{remark:rescaling_of_K}
    For a function $\F:\Mk\to\R$ where $\Mk$ is a manifold of constant sectional curvature $\K\not\in\{1, -1, 0\}$, we can apply a rescaling to the Gnomonic or Beltrami-Klein projection to define a function on a manifold of constant sectional curvature $\K\in\{1, -1\}$. Namely, we can map $\Mk$ to $\MEucl$ via the geodesic map $\h:\Mk\to \MEucl$, then we can rescale $\MEucl$ by multiplying each vector in $\MEucl$ by the factor $\sqrt{\abs{\K}}$ and then we can apply the inverse geodesic map for the manifold of curvature $\K\in\{1, -1\}$. If $\RR$ is the original bound of the initial distance to an optimum, and $\F$ is $\L$-smooth and $\mu$-strongly g-convex (possibly with $\mu=0$) then the initial distance bound becomes $\sqrt{\abs{\K}}\RR$ and the induced function becomes $\L/\abs{\K}$-smooth and $\mu/\abs{\K}$-strongly g-convex. This is a consequence of the transformation rescaling distances by a factor of $\sqrt{\abs{\K}}$, i.e., if $r:\Mk \to \Mk[\K/\abs{\K}]$ is the rescaling function, then $d_{\K}(x,y)\sqrt{\abs{\K}} = d_{\K/\abs{\K}}(r(x),r(y))$, where $d_c(\cdot, \cdot)$ denotes the distance on the manifold of constant sectional curvature $c$.
\end{remark}

\subsection{Preliminaries}
We recall our characterization of the geodesic map. Given two points $\tilde{x}, \tilde{y}\in \X$, we have that $\dist(x,y)$, the distance between $x$ and $y$ with the metric of $\Mk$, satisfies
\begin{equation}\label{eq:appendix_characterization_of_geodesic_map_and_metric}
    \ck[\dist(x,y)]  = \frac{1+\K\innp{\tilde{x}, \tilde{y}}}{\sqrt{1+\K\norm{\tilde{x}}^2}\cdot\sqrt{1+\K\norm{ \tilde{y}}^2}}.
\end{equation}
And since the expression is symmetric with respect to rotations, $\X=\h(\BR)$ is a closed ball of radius $\Rtilde$, with $\ck[\RR] = (1+\K\Rtilde^2)^{-1/2}$. Equivalently, 
\begin{align} \label{eq:R_tilde_vs_R}
 \begin{aligned}
     \Rtilde = \tan(\RR)  &\ & & \text{ if } \K=1, \\
     \Rtilde = \tanh(\RR)  &\ & & \text{ if } \K=-1.
   \end{aligned}
\end{align}
Similarly, we can write the distances as
\begin{align} \label{eq:distances}
 \begin{aligned}
     \dist(x, y) = \arccos\left(\frac{ 1 + \innp{\tilde{x}, \tilde{y}}}{\sqrt{1+\norm{\tilde{x}^2}}\sqrt{1+\norm{\tilde{y}^2}}}\right)  &\ & & \text{ if } \K=1, \\
     \dist(x, y) = \arccosh\left(\frac{ 1 - \innp{\tilde{x}, \tilde{y}}}{\sqrt{1-\norm{\tilde{x}^2}}\sqrt{1-\norm{\tilde{y}^2}}}\right)  &\ & & \text{ if } \K=-1, 
   \end{aligned}
\end{align}
Alternatively, we have the following expression for the distance $\dist(x, y)$ when $\K=-1$. Let $\tilde{a}, \tilde{b}$ be the two points of intersection of the ball $B(0,1)\supseteq \MEucl$ with the line joining $\tilde{x}, \tilde{y}$, so the order of the points in the line is $\tilde{a}, \tilde{x}, \tilde{y}, \tilde{b}$. Then
\begin{equation} \label{eq:distances_hyper_alt}
    \dist(x, y) = \frac{1}{2}\log\left(\frac{\norm{\tilde{a}-\tilde{y}}\norm{\tilde{x}-\tilde{b}}}{\norm{\tilde{a}-\tilde{x}}\norm{\tilde{b}-\tilde{y}}}\right) \text{ if } \K=-1.
\end{equation}
We will use this expression when working with the hyperbolic space. A simple elementary proof of the equivalence of the expressions in \eqref{eq:distances} and \eqref{eq:distances_hyper_alt} when $\K=-1$ is the following. We can assume without loss of generality that we work with the hyperbolic plane, i.e., $\n=2$. By rotational symmetry, we can also assume that $\tilde{x} = (x_1, x_2)$ and $\tilde{y} = (y_1, y_2)$, for $x_1=y_1$. In fact, it is enough to prove it in the case $x_2=0$ because we can split a general segment into two, each with one endpoint at $(x_1, 0)$, and then add their lengths up. So according to \eqref{eq:distances} and \eqref{eq:distances_hyper_alt}, respectively, we have
\[
    \frac{1}{\cosh^2(\dist(x, y))} =\frac{(1-x_1^2)(1-y_1^2-y_2^2)}{(1-x_1^2)^2} =  \frac{(1-x_1^2-y_2^2)}{1-x_1^2},
\] 
\begin{align*} 
 \begin{aligned}
     \dist(x, y)&=\frac{1}{2}\log \left(\frac{(\sqrt{1-y_1^2}+y_2)(\sqrt{1-x_1^2})}{(\sqrt{1-x_1^2})(\sqrt{1-y_1^2}-y_2)}\right) = \frac{1}{2}\log\left(\frac{1+y_2/\sqrt{1-x_1^2}}{1-y_2/\sqrt{1-x_1^2}}\right) \\
     &= \arctanh\left(\frac{y_2}{\sqrt{1-x_1^2}}\right).
   \end{aligned}
\end{align*}
where we have used the equality $\arctanh(t) = \frac{1}{2}\log(\frac{1+t}{1-t})$. Now, using the trigonometric identity $\frac{1}{\cosh^2(t)} = 1-\tanh^2(t)$, for $t=\dist(x, y)$, we obtain that the two expressions above are equal. See Theorem 7.4 in \citep{greenberg1993euclidean} (p. 268) for more details about the distance formula under this geodesic map.

The spherical case is of a remarkable simplicity. If we have an $(\n)$-sphere of radius $1$ centered at $0$, we can see the transformation of the geodesic map as the projection onto the plane $x_{\n+1}=1$. Given two points $\mathbf{x}=(\tilde{x}, 1)$, $\mathbf{y}=(\tilde{y}, 1)$ then the angle between these two vectors is the distance of the projected points on the sphere so we have $\cos(\dist(x,y))=\innp{\mathbf{x}, \mathbf{y}}/\norm{\mathbf{x}}\norm{\mathbf{y}}$ which is equivalent to the corresponding formula in \eqref{eq:distances}.

\subsection{Distance deformation} \label{app:sec_distance_deformation}

\begin{lemma} \label{lemma:distances_hyperbolic} 
    Let $x,y\in\BR=\expon{\xInit[\notilde]}(\ball(0, \RR))\subseteq \Mk$ be two different points, where $\Mk$ is the hyperbolic space with constant sectional curvature $\K=-1$. Then, we have
    \[
        1 \leq \frac{\dist(x, y)}{\norm{\tilde{x}-\tilde{y}}} \leq \cosh^2(\RR).
    \]
\end{lemma}
\begin{proof}
    We can assume without loss of generality that the dimension is $\n=2$. As in \eqref{eq:R_tilde_vs_R}, let $\Rtilde=\tanh(\RR)$, so any point $\tilde{x}\in\X$ satisfies $\norm{\tilde{x}} \leq \Rtilde$, or equivalently $\dist(x, \xInit[\notilde]) \leq \RR$.  Recall $\xInit=\h(\xInit[\notilde])=0$. Without loss of generality, we parametrize an arbitrary segment of length $\ell$ in $\X$ by two endpoints $\tilde{x}, \tilde{y}$ with coordinates $\tilde{x}=(x_1, x_2)$ and $\tilde{y}=(x_1-\ell, x_2)$, for $0\leq x_2\leq \Rtilde$, $0\leq x_1 \leq\sqrt{\Rtilde^2-x_2^2}$ and $0<\ell \leq x_1+\sqrt{\Rtilde^2-x_2^2}$. Let $\mathfrak{d}(x_1, x_2, \ell) \defi \frac{\dist(x,y)}{\ell}$, the quantity we aim to bound. We will prove the upper bound on $\mathfrak{d}(x_1, x_2, \ell)$ in three steps. 
    \begin{enumerate}
       \item  If $x_1 = \ell$ then $\mathfrak{d}(\cdot)$ is larger the larger $x_1$ is. This allows to prove that it is enough to consider points with the extra constraint $\ell \leq x_1$.

       \item The partial derivative of $\mathfrak{d}(\cdot)$ with respect to $x_1$, whenever $\ell \leq x_1$, is non-negative. So we can just look at the points for which $x_1=\sqrt{\Rtilde^2-x_2^2}$.

       \item With the constraints above, $\mathfrak{d}(\cdot)$ is larger the smaller $\ell$ is. So we have 
           \[
           \mathfrak{d}(x_1, x_2, \ell) \leq \lim_{\ell\to 0} \mathfrak{d}(\sqrt{\Rtilde^2-x_2^2}, x_2, \ell) = \sqrt{1-x_2^2}/(1-\Rtilde^2).
           \] 
           This expression is maximized at $x_2=0$ and evaluates to $1/(1-\tanh^2(\RR)) = \cosh^2(\RR)$.
   \end{enumerate} 
   We proceed now to prove the steps above. For the first step, we note
   \[
       \mathfrak{d}(x_1, x_2, x_1) = \frac{1}{2x_1}\log\left(\frac{\sqrt{1-x_2^2}(\sqrt{1-x_2^2}+x_1)}{\sqrt{1-x_2^2}(\sqrt{1-x_2^2}-x_1)}\right) = \frac{1}{2x_1}\log\left(1+\frac{2x_1}{\sqrt{1-x_2^2}-x_1}\right).
   \]
   We prove that the inverse of this expression is not increasing with respect to $x_1$. By taking a partial derivative:
\begin{align*} 
 \begin{aligned}
     \frac{\partial (1/\mathfrak{d}(x_1, x_2,x_1))}{\partial x_1} = 2\frac{\frac{-2x_1\sqrt{1-x_2^2}}{1-x_2^2-x_1^2}+\log(1+2x_1/(\sqrt{1-x_2^2}-x_1))}{\log(1+2x_1/(\sqrt{1-x_2^2}-x_1))^2} \stackrel{?}{\leq} 0  \\
    \iff \frac{2x_1\sqrt{1-x_2^2}}{1-x_2^2-x_1^2}-\log(1+(2x_1\sqrt{1-x_2^2}+2x_1^2)/(1-x_2^2-x_1^2)) \stackrel{?}{\geq} 0.
   \end{aligned}
\end{align*}
In order to see that the last inequality is true, note that the expression on the left hand side is $0$ when $x_1 = x_2 = 0$. And the partial derivatives of this with respect to $x_1$ and $x_2$, respectively, are:
\[
   \frac{4\sqrt{1-x_2^2} x_1^2}{(1-x_2^2-x_1^2)^2} \text{ and } \frac{4x_2x_1^3}{\sqrt{1-x_2^2}(1-x_2^2-x_1^2)^2} .
\]
Both are greater than $0$ in the interior of the domain $0\leq x_2\leq \Rtilde$, $0\leq x_1 \leq\sqrt{\Rtilde^2-x_2^2}$ and at least $0$ in the border.
    Now we use this monotonicity to prove that we can consider $\ell \leq x_1$ only. Suppose $\ell > x_1$. The segment $\ell$ is divided into two parts by the $e_2$ axis and we can assume without loss of generality that the negative part is no greater than the other, i.e., $x_1 \geq \ell - x_1$. Otherwise, we can perform the computations after a symmetry over the $e_2$ axis. Let $\tilde{r}$ be the point $(0, x_2)$. We want to see that the segment from $\tilde{x}$ to $\tilde{r}$ gives a greater value of $\mathfrak{d}(\cdot)$:
    \begin{align*} 
     \begin{aligned}
         \frac{\dist(x, r)}{x_1} \geq \frac{\dist(x, y)}{\ell} & \iff \dist(x, r) (x_1 + (\ell-x_1)) \geq x_1( \dist(x, r) + \dist(r, y) )   \\ &\iff \dist(x,r)/x_1 \geq \dist(r, y)/(\ell-x_1),
       \end{aligned}
    \end{align*}
    and the last inequality holds true by the monotonicity we just proved.

    In order to prove the second step, we take the partial derivative of $\mathfrak{d}(x_1, x_2, \ell)$ with respect to $x_1$. We have
    \[
        \mathfrak{d}(x_1, x_2, \ell) = \frac{1}{2\ell}\log\left(\frac{1+\ell/(\sqrt{1-x_2^2}-x_1)}{1-\ell/\sqrt{1-x_2^2}+x_1}\right),
    \]
    \[
        \frac{\partial \mathfrak{d}(x_1, x_2, \ell)}{\partial x_1} = \frac{\sqrt{1-d^2}(2x_1-\ell)}{2(1-x_2^2-x_1^2)(1-x_2^2-(x_1-\ell)^2)}.
    \]
    And the derivative is positive in the domain we are considering.

    We now prove step $3$. We want to show that $\mathfrak{d}(\sqrt{\Rtilde^2-x_2^2}, x_2, \ell\cdot)$ decreases with $\ell$, within our constraints $\ell \leq x_1 = \sqrt{\Rtilde^2-x_2^2}$, $0\leq x_2\leq \Rtilde$. If we split the segment joining $\tilde{x}$ and $\tilde{y}$ in half, with respect to the metric in $\X$, we see that due to the monotonicity proved in step $1$, the segment that is farther to the origin is longer in $\M$ than the other one and so $\mathfrak{d}(\cdot)$ is greater for this half of the segment than for the original one. In symbols, let $\tilde{r}$ be the middle point of the segment joining $\tilde{x}$ and $\tilde{y}$. We have by monotonicity that $\mathfrak{d}(x_1, x_2, \ell/2) \geq \mathfrak{d}(x_1, x_2-\ell/2, \ell/2)$. So $\mathfrak{d}(x_1, x_2, \ell/2) = \frac{\dist(\tilde{x},\tilde{r})}{\ell/2} \geq \frac{\dist(\tilde{x},\tilde{r}) + \dist(\tilde{r}, \tilde{y})}{\ell}  = \mathfrak{d}(x_1, x_2, \ell)$. Thus,
    \begin{align*} 
     \begin{aligned}
        \mathfrak{d}(x_1,x_2, \ell) &\leq \lim_{\ell\to 0} \mathfrak{d}(\sqrt{\Rtilde^2-x_2^2},x_2, \ell) \\
        &= \lim_{\ell\to 0}    \frac{1}{2\ell}\log\left(\frac{1+\ell/\left(\sqrt{1-x_2^2}-\sqrt{\Rtilde^2-x_2^2}\right)}{1-\ell/\left(\sqrt{1-x_2^2}+\sqrt{\Rtilde^2-x_2^2}\right)}\right) \\
        & \circled{1}[=] \lim_{\ell\to 0}  \frac{\sqrt{1-x_2^2}}{1-\Rtilde^2-2\ell\sqrt{\Rtilde^2-x_2^2}+\ell^2} \\
        & = \frac{\sqrt{1-x_2^2}}{1-\Rtilde^2}.
       \end{aligned}
    \end{align*}
        We used L'Hôpital's rule for $\circled{1}$. We can maximize the last the result of the limit by setting $x_2=0$ and obtain that for any two different $\tilde{x}, \tilde{y} \in \X$
    \[
        \frac{\dist(x, y)}{\norm{\tilde{x}-\tilde{y}}} \leq \frac{1}{1-\Rtilde^2}= \frac{1}{1-\tanh^2(\RR)} = \cosh^2(\RR).
    \]

    The lower bound is similar, assume that $\ell>x_1$ and define $\tilde{r}$ as above. We assume again without loss of generality that $x_1\geq \ell-x_1$. Then
    \[
        \frac{\dist(x, r)+\dist(r, y)}{\ell} \geq \frac{\dist(x, r)}{\ell-x_1} \iff \frac{\dist(r, y)}{x_1} \geq \frac{\dist(x,r)}{\ell-x_1}
    \] 
    and the latter is true by the monotonicity proved in step $1$. This means that we can also consider $\ell \leq x_1$. But this time, according to step $2$, we want $x_1$ to be the lowest possible, so it is enough to consider $x_1=\ell$. Using step $1$ again, we obtain that the lowest value of $\mathfrak{d}(\cdot)$ can be bounded by the limit $\lim_{\ell\to 0} \mathfrak{d}(\ell, x_2, \ell)$ which using L'Hôpital's rule in $\circled{1}$ is
    \begin{align*} 
         \begin{aligned}
            \mathfrak{d}(x_1,x_2, \ell) &\geq \lim_{\ell\to 0} \mathfrak{d}(\ell,x_2, \ell) \\
             &= \lim_{\ell\to 0}    \frac{1}{2\ell}\log\left(1+\frac{2\ell}{\sqrt{1-x_2^2}-\ell}\right) \\
             & \circled{1}[=] \lim_{\ell\to 0}  \frac{\frac{2(\sqrt{1-x_2^2}-\ell)+2\ell}{(\sqrt{1-x_2^2}-\ell)^2}}{2(1+2\ell/(\sqrt{1-x_2^2}-\ell))} \\
            & = \frac{1}{\sqrt{1-x_2^2}}.
           \end{aligned}
        \end{align*}
       The expression is minimized at $x_2=0$ and evaluates to $1$.
\end{proof}

The proof of the corresponding lemma for the sphere is analogous, we add it for completeness.

\begin{lemma} \label{lemma:distances_spherical} 
    Let $x,y\in B_R=\expon{\xInit[\notilde]}(\ball(0, \RR))\subseteq \Mk$ be two different points, where $\Mk$ is the spherical space with constant sectional curvature $\K=1$, and $\RR<\pi/2$. Then, we have
    \[
         \cos^2(\RR) \leq \frac{\dist(x, y)}{\norm{\tilde{x}-\tilde{y}}} \leq 1.
    \]
\end{lemma}

\begin{proof}
    We proceed in a similar way than with the hyperbolic case. We can also work with $d=2$ only, since $\tilde{x}, \tilde{y}$ and $\xInit$ lie on a plane. We parametrize a general pair of points as $\tilde{x} = (x_1, x_2) \in \X$ and $y=(x_1-\ell,x_2) \in \X$, so $x_1^2+x_2^2 \leq \Rtilde^2$, for $\Rtilde = \tan(\RR)$ and by definition $\ell=\norm{\tilde{x}-\tilde{y}}$. 
    
    Let $\mathfrak{d}(x_1,x_2,\ell) \defi \dist(x,y)/\norm{\tilde{x}-\tilde{y}}$. We proceed to prove the result in three steps, similarly to the hyperbolic case.
\begin{enumerate}
    \item If $x_1=\ell$ then $\mathfrak{d}(x_1,x_2,\ell)$ decreases whenever $x_1$ increases. This allows to prove that it is enough to consider points in which $\ell\leq x_1$.
    \item $\frac{\partial \mathfrak{d}(\cdot)}{\partial x_1} \leq 0$, whenever $\ell \leq x_1$. So we can consider $x_1=\sqrt{\Rtilde^2-x_2^2}$ only.
    \item With the constraints above, $\mathfrak{d}(\cdot)$ increases with $\ell$, so in order to lower bound $\mathfrak{d}(\cdot)$ we can consider $\lim_{\ell\to 0}\mathfrak{d}(\sqrt{\Rtilde-x_2},x_2, \ell) = \sqrt{1+x_2^2}/(1+\Rtilde^2)$. This is minimized at $x_2=0$ and evaluates to $1/(1+\Rtilde^2)$.

\end{enumerate}

For the first step, we compute the partial derivative:
\begin{equation}\label{eq:partial_derivative_spherical_deformation}
    \frac{\partial \mathfrak{d}(x_1, x_2, x_1)}{\partial x_1} = \frac{x_1 \sqrt{1 + x_2^2}/(1 + x_1^2 + x_2^2) - \arccos\left(\sqrt{(1 + x_2^2)/(1 + x_1^2 + x_2^2)}\right)}{x_1^2}.
\end{equation}
In order to see that it is non-positive, we compute the partial derivative of the denominator with respect to $x_2$ and obtain
\[
    \frac{2x_1^3 x_2}{\sqrt{1+x_2^2}(1+x_1^2+x_2^2)} \geq 0.
\]
so in order to maximize \eqref{eq:partial_derivative_spherical_deformation} we set $x_2=\sqrt{\Rtilde-x_1^2}$. In that case, the numerator is
\begin{equation}\label{eq:aux_eq_sph_def}
\frac{x_1 \sqrt{1 + \RR^2 - x_1^2}}{1 + \RR^2} - \arccos\left(\sqrt{\frac{1 + \RR^2 - x_1^2}{1 + \RR^2}}\right),
\end{equation}
and its derivative with respect to $x_1$ is
\[
    -\frac{2 x_1^2}{(1 + \RR^2) \sqrt{1 + \RR^2 - x_1^2}} \leq 0.
\]
and given that \eqref{eq:aux_eq_sph_def} with $x_1=0$ evaluates to $0$ we conclude that \eqref{eq:partial_derivative_spherical_deformation} is non-positive. Similarly to \cref{lemma:distances_hyperbolic}, suppose the horizontal segment that joins $\tilde{x}$ and $\tilde{y}$ passes through $\tilde{r} \defi (0,x_2)$. And suppose without loss of generality that $\dist(x,r) \geq \dist(r,y)$, i.e., $x_1 \geq \ell-x_1$. Then by the monotonicity we just proved, we have 
\begin{equation}\label{eq:aux_mathfrakd_functions}
    \frac{\dist(x,r)}{\norm{\tilde{x}-\tilde{r}}}  = \mathfrak{d}(x_1, x_2, x_1) \leq \mathfrak{d}(\ell - x_1, x_2, \ell-x_1) = \frac{\dist(r,y)}{\norm{\tilde{r}-\tilde{y}}}.
\end{equation}
And this implies $\mathfrak{d}(x_1, x_2, x_1) \leq \mathfrak{d}(x_1, x_2, \ell)$. Indeed, that is equivalent to show
\[
    \frac{\dist(x,r)}{\norm{\tilde{x}-\tilde{r}}} \leq \frac{\dist(x,y)}{\norm{\tilde{x}-\tilde{y}}} = \frac{\dist(x,r)+\dist(r+y)}{\norm{\tilde{x}-\tilde{r}}+\norm{\tilde{r}-\tilde{y}}}.
\]
Which is true, since after simplifying we arrive to \eqref{eq:aux_mathfrakd_functions}. So in order to lower bound $\mathfrak{d}(\cdot)$, it is enough to consider $\ell\leq x_1$.

We focus on step $2$ now. We have
\[
    \frac{\partial \mathfrak{d}(x_1,x_2,\ell)}{\partial x_1}  = \frac{\sqrt{1+x_2^2}(\ell-2 x_1)}{(1+x_2^2+(\ell-x_1)^2)(1+x_2^2+x_1^2)}.
\]
which is non-positive given the restrictions we imposed after step $1$. So in order to lower bound $\mathfrak{d}(\cdot)$ we can consider $x_1=\sqrt{\Rtilde-x_2^2}$ only.

Finally, in order to complete step $3$ we compute
\begin{align*}
 \begin{aligned}
     \frac{\partial \mathfrak{d}(\sqrt{\Rtilde-x_2^2},x_2,\ell)}{\partial \ell} &= \frac{\sqrt{1+x_2^2}}{\ell(1+\Rtilde^2) + \ell^3-2\ell^2\sqrt{\Rtilde^2-x_2^2}} \\
      &\quad - \frac{1}{\ell^2}\arccos\left(\frac{1+\Rtilde^2-\ell\sqrt{\Rtilde^2-x_2^2}}{\sqrt{(1+\Rtilde^2)(1+\Rtilde^2+\ell^2-2\ell\sqrt{\Rtilde^2-x_2^2})}}\right)
   \end{aligned}
\end{align*}

And in order to prove that this is non-negative, we will prove that the same expression is non-negative, when multiplied by $\ell^2$. We compute the partial derivative of the aforementioned expression with respect to $\ell$:
\begin{align*}
 \begin{aligned}
     \frac{\partial}{\partial \ell}\left(\frac{\partial \mathfrak{d}(\sqrt{\Rtilde-x_2^2},x_2,\ell)}{\partial \ell} \ell^2\right) = \frac{2\ell\sqrt{1+x_2^2}(\sqrt{\Rtilde^2-x_2^2}-\ell)}{(1+\Rtilde^2+\ell^2-2\ell\sqrt{\Rtilde^2-x_2^2})^2} \geq 0.
   \end{aligned}
\end{align*}

And $\ell^2(\partial \mathfrak{d}(\sqrt{\Rtilde-x_2^2},x_2,\ell)/\partial \ell )$ evaluated at $0$ is $0$ for all choices of parameters $\RR$ and $x_2$ in the domain. So we conclude that $\partial \mathfrak{d}(\sqrt{\Rtilde-x_2^2},x_2,\ell)/\partial \ell \geq 0$.

Thus, we can consider the limit when $\ell\to 0$ in order to lower bound $\mathfrak{d}(\cdot)$. In the defined domain, we have 
\begin{align*}
 \begin{aligned}
     \lim_{\ell\to 0}\mathfrak{d}(\sqrt{\Rtilde-x_2},x_2, \ell) &= \lim_{\ell\to 0}\frac{1}{\ell}\arccos\left(\frac{1+\Rtilde^2-x\sqrt{\Rtilde^2-x_2^2}}{\sqrt{1+\Rtilde^2}\sqrt{1+x_2^2+(\ell-\sqrt{\Rtilde^2-x_2^2})^2}}\right)  \\
     &\circled{1}[=] \lim_{\ell\to 0} \frac{\sqrt{1+x_2^2}}{1+\Rtilde^2+\ell^2-2\ell\sqrt{\Rtilde^2-x_2^2}}\\
     &=\frac{\sqrt{1+x_2^2}}{1+\Rtilde^2}.
   \end{aligned}
\end{align*}
We used L'Hôpital's rule for $\circled{1}$. Now, the right hand side of the previous expression is minimized at $x_2=0$ so we conclude that we have
\[
    \cos^2(\RR) = \frac{1}{1+\tan^2(\RR)} = \frac{1}{1+\Rtilde^2} \leq \mathfrak{d}(x_1, x_2, \ell) = \frac{\dist(p,q)}{\norm{\tilde{p}-\tilde{q}}}.
\]

    The upper bound uses again a similar argument. Assume that $\ell>x_1$ and define $\tilde{r}$ as above. We assume again without loss of generality that $x_1\geq \ell-x_1$. Then
    \[
        \frac{\dist(x, r)+\dist(r, y)}{\ell} \leq \frac{\dist(x, r)}{\ell-x_1} \iff \frac{\dist(r, y)}{x_1} \leq \frac{\dist(x,r)}{\ell-x_1}
    \] 
    and the latter is true by the monotonicity proved in step $1$. Consequently we can just consider the points that satisfy $\ell \leq x_1$. By step $2$, $\mathfrak{d}(\cdot)$ is maximal whenever $x_1$ is the lowest possible, so it is enough to consider $x_1=\ell$. Using step $1$ again, we obtain that the greatest value of $\mathfrak{d}(\cdot)$ can be bounded by the limit $\lim_{\ell\to 0} \mathfrak{d}(\ell, x_2, \ell)$ which using L'Hôpital's rule in $\circled{1}$ and simplifying is
    \begin{align*} 
         \begin{aligned}
            \mathfrak{d}(x_1,x_2, \ell) &\leq \lim_{\ell\to 0} \mathfrak{d}(\ell,x_2, \ell) \\
             &= \lim_{\ell\to 0}    \frac{1}{\ell}\arccos\left(\sqrt{\frac{1+x_2^2}{1+\ell^2+x_2^2}}\right) \\
             & \circled{1}[=] \frac{1}{\sqrt{1+x_2^2}}.
           \end{aligned}
        \end{align*}
       The expression is maximized at $x_2=0$ and evaluates to $1$.
\end{proof}

\subsection{Angle deformation}\label{app:sec_angle_deformation}
\begin{lemma}\label{lemma:angle_deformation} 
    Let $x,y\in B_R=\expon{\xInit[\notilde]}(\ball(0, \RR))\subseteq \Mk$ be two different points and different from $\xInit[\notilde]$, where $\Mk$ is a manifold constant sectional curvature $\K\in \{1, -1\}$, and if $\K=1$, then $\RR<\pi/2$. Let $\tilde{\alpha}$ be the angle $\angle \xInit[\notilde]xy$, formed by the vectors $\xInit[\notilde]-x$ and $y-x$. Let $\alpha$ be the corresponding angle between the vectors $\exponinv{x}(\xInit[\notilde])$ and $\exponinv{x}(y)$.  The following holds:
\begin{align*}
\begin{aligned}
\sin(\alpha) = \sin(\tilde{\alpha}) \sqrt{\frac{1+\K\norm{\tilde{x}}^2}{1+\K\norm{\tilde{x}}^2\sin^2(\tilde{\alpha})}}, \quad \quad \cos(\alpha) = \cos(\tilde{\alpha}) \sqrt{\frac{1}{1+\K\norm{\tilde{x}}^2\sin^2(\tilde{\alpha})}}.
\end{aligned}
\end{align*}
\end{lemma}
\begin{proof}
    Note that we can restrict ourselves to $\alpha \in [0, \pi]$ because we have $\widetilde{(-w)} = -\tilde{w}$ (recall our \hyperlink{sec:notation}{notation} about vectors with tilde). This means that the result for the range $\alpha\in[-\pi, 0]$ can be deduced from the result for $-\alpha$.

    We start with the case $\K=-1$. We can assume without loss of generality that the dimension is $\n=2$, and that the coordinates of $\tilde{x}$ are $(0,x_2)$, for $x_2\leq \tanh(\RR)$ that $\tilde{y}=(y_1, y_2)$, for some $y_1\leq 0$ and $\tilde{\delta} \defi \angle \tilde{y}\xInit\tilde{x} \in [0, \pi/2]$, since we can make the distance $\norm{\tilde{x}-\tilde{y}}$ as small as we want. Recall $\xInit=\zeros_n$. We recall that $\dist(x, \xInit[\notilde]) = \arctanh(\norm{\tilde{x}})$ and we note that $\sinh(\arctanh(t)) = \frac{t}{1-t^2}$, so that $\sinh(\dist(x, \xInit[\notilde])) = \norm{\tilde{x}}/\sqrt{1-\norm{\tilde{x}}^2}$, for any $\tilde{x}\in\X$. We will apply the hyperbolic and Euclidean law of sines \cref{thm:law_of_sines} in order to compute the value of $\sin(\alpha)$ with respect to $\tilde{\alpha}$. Let $\tilde{a}$ and $\tilde{b}$ be points in the border of $B(0,1)$ such that the segment joining $\tilde{a}$ and $\tilde{b}$ is a chord that contains $\tilde{x}$ and $\tilde{y}$ and $\norm{\tilde{a}-\tilde{x}}\leq\norm{\tilde{b}-\tilde{x}}$. So $\norm{\tilde{a}-\tilde{x}}$ and $\norm{\tilde{b}-\tilde{x}}$ are $\sqrt{1-\norm{\tilde{x}}^2\sin^2(\tilde{\alpha})}-\norm{\tilde{x}}\cos(\tilde{\alpha})$ and $\sqrt{1-\norm{\tilde{x}}^2\sin^2(\tilde{\alpha})}+\norm{\tilde{x}}\cos(\tilde{\alpha})$, respectively. We have
\begingroup
\allowdisplaybreaks
    \begin{align*}
     \sin(\alpha) &\circled{1}[=] \frac{\sinh(\dist(\xInit[\notilde], y))\sin(\tilde{\delta})}{\sinh(\dist(x, y))} \circled{2}[=] \frac{\norm{\xInit-\tilde{y}}}{\sqrt{1-\norm{\xInit-\tilde{y}}^2}} \cdot\frac{\norm{\tilde{x}-\tilde{y}} \sin(\tilde{\alpha})}{\norm{\xInit-\tilde{y}}}\cdot \frac{1}{ \sinh(\dist(x, y))}\\
     &\circled{3}[=] \frac{\sin(\tilde{\alpha})}{\sqrt{1-\norm{\tilde{x}}^2+\norm{\tilde{x}-\tilde{y}}(-2\norm{\tilde{x}}\cos(\tilde{\alpha})+\norm{\tilde{x}-\tilde{y}})}} \cdot\frac{\norm{\tilde{x}-\tilde{y}}}{\sinh(\dist(x, y))} \\
     &\circled{4}[=] \frac{\sin(\tilde{\alpha})}{\sqrt{1-\norm{\tilde{x}}^2}} \lim_{\dist(x, y)\to 0} \norm{\tilde{x}-\tilde{y}}\frac{1}{\sinh(\dist(x,y))} \\
        &\circled{5}[=] \frac{\sin(\tilde{\alpha})}{\sqrt{1-\norm{\tilde{x}}^2}} \lim_{\dist(x, y)\to 0} \frac{(e^{2\dist(x, y)}-1)(\norm{\tilde{a}-\tilde{x}}\cdot \norm{\tilde{b}-\tilde{x}})}{e^{2\dist(x, y)}(\norm{\tilde{a}-\tilde{x}} + \norm{\tilde{b}-\tilde{x}})}\cdot \frac{2e^{\dist(x, y)}}{e^{2\dist(x, y)}-1}\\
    &=\frac{\sin(\tilde{\alpha})}{\sqrt{1-\norm{\tilde{x}}^2}}\cdot\frac{2 \norm{\tilde{a}-\tilde{x}} \cdot \norm{\tilde{b}-\tilde{x}}}{\norm{\tilde{a}-\tilde{x}} + \norm{\tilde{b}-\tilde{x}}} \\
     &\circled{6}[=]\frac{\sin(\tilde{\alpha})}{\sqrt{1-\norm{\tilde{x}}^2}}\cdot \frac{2(1-\norm{\tilde{x}}^2\sin^2(\tilde{\alpha})-\norm{\tilde{x}}^2\cos^2(\tilde{\alpha}))}{2\sqrt{1-\norm{\tilde{x}}^2\sin^2(\tilde{\alpha})}} = \sin(\tilde{\alpha})\sqrt{\frac{1-\norm{\tilde{x}}^2}{1-\norm{\tilde{x}}^2\sin^2(\tilde{\alpha})}}.
\end{align*}
\endgroup
    In $\circled{1}$ we used the hyperbolic sine theorem. In $\circled{2}$ we used the expression above regarding segments that pass through the origin, and the Euclidean sine theorem. In $\circled{3}$, we simplify and use that the coordinates of $\tilde{y}$ are $(-\sin(\tilde{\alpha})\norm{\tilde{x}-\tilde{y}}, \norm{\tilde{x}}-\cos(\tilde{\alpha})\norm{\tilde{x}-\tilde{y}})$. Then, in $\circled{4}$, since $\sin(\alpha)$ does not depend on $\norm{\tilde{x}-\tilde{y}}$, we can take the limit when $\dist(x, y) \to 0$, by which we mean we take the limit $\tilde{y}\to\tilde{x}$ by keeping the angle $\tilde{\alpha}$ constant. Since a posteriori the limit of each fraction exists, we compute them one at a time. $\circled{5}$ uses \eqref{eq:distances_hyper_alt} and the definition of $\sinh(\dist(x, y))$ for the last factor. The equality for the other factor can be checked with \eqref{eq:distances_hyper_alt}. In $\circled{6}$ we substitute $\norm{\tilde{a}-\tilde{x}}$ and $\norm{\tilde{b}-\tilde{x}}$ by their values.

    The spherical case is similar to the hyperbolic case. We also assume without loss of generality that the dimension is $\n=2$. Define $\tilde{y}$ as a point such that $\angle \xInit\tilde{x}\tilde{y} = \tilde{\alpha} $. We can assume without loss of generality that the coordinates of $\tilde{x}$ are $(0,x_2)$, that $\tilde{y}=(y_1, y_2)$, for $y_1\leq 0$, and $\tilde{\delta} \defi \angle \tilde{y}\xInit\tilde{x} \in [0, \pi/2]$, since we can make the distance $\norm{\tilde{x}-\tilde{y}}$ as small as we want. We recall that by \eqref{eq:R_tilde_vs_R} we have $\dist(\xInit[\notilde],x) = \arctan(\norm{\xInit-\tilde{x}})$ and we note that $\sin(\arctan(t)) = \frac{t}{1+t^2}$, so that $\sin(\dist(\xInit[\notilde], x)) = \norm{\xInit-\tilde{x}}/\sqrt{1+\norm{\xInit-\tilde{x}}^2}$, for any $\tilde{x}\in\X$. We will apply the spherical and Euclidean law of sines \cref{thm:law_of_sines} in order to compute the value of $\sin(\alpha)$ with respect to $\tilde{\alpha}$. We have

\begingroup
\allowdisplaybreaks
\begin{align*}
     \sin(\alpha) &\circled{1}[=] \frac{\sin(\dist(\xInit[\notilde], y))\sin(\tilde{\delta})}{\sin(\dist(x, y))} \circled{2}[=] \frac{\norm{\xInit-\tilde{y}}}{\sqrt{1+\norm{\xInit-\tilde{y}}^2}} \cdot\frac{\norm{\tilde{x}-\tilde{y}} \sin(\tilde{\alpha})}{\norm{\xInit-\tilde{y}}} \frac{1}{ \sin(\dist(x, y))}\\
     &\circled{3}[=] \frac{\sin(\tilde{\alpha})\norm{\tilde{x}-\tilde{y}}}{\sqrt{1+\norm{\xInit-\tilde{y}}^2}\sqrt{1-\frac{(1-\norm{x}\cos(\tilde{\alpha})\norm{\tilde{x}-\tilde{y}}+\norm{\tilde{x}}^2)^2}{(1+\norm{\tilde{x}}^2)(1+\norm{\xInit-\tilde{y}}^2)}}}\\
     &\circled{4}[=] \frac{\sin(\tilde{\alpha})\norm{\tilde{x}-\tilde{y}}}{\sqrt{\norm{\tilde{x}-\tilde{y}}^2(1+\norm{\tilde{x}}^2-\norm{\tilde{x}}^2\cos^2(\tilde{\alpha}))/(1+\norm{\tilde{x}}^2)}} \circled{5}[=] \sin(\tilde{\alpha})\sqrt{\frac{1+\norm{\tilde{x}}^2}{1+\norm{\tilde{x}}^2\sin^2(\tilde{\alpha})}}.
\end{align*}
\endgroup

    In $\circled{1}$ we used the spherical sine theorem. In $\circled{2}$ we used the expression above regarding segments that pass through the origin, and the Euclidean sine theorem. In $\circled{3}$, we use the fact that the coordinates of $\tilde{y}$ are $(-\sin(\tilde{\alpha})\norm{\tilde{x}-\tilde{y}}, \norm{\tilde{x}}^2-\cos(\tilde{\alpha})\norm{\tilde{x}-\tilde{y}})$, use the distance formula \eqref{eq:distances} and the trigonometric equality $\sin(\arccos(x)) = \sqrt{1-x^2}$. Then, in $\circled{4}$ and $\circled{5}$, we multiply and simplify. 

    Finally, in both cases, the cosine formula is derived from the identity $\sin^2(\alpha)+\cos^2(\alpha)=1$ after noticing that the sign of $\cos(\alpha)$ and the sign of $\cos(\tilde{\alpha})$ are the same. The latter fact can be seen to hold true by noticing that $\alpha$ is monotonous with respect to $\tilde{\alpha}$ and  the fact that $\tilde{\alpha}=\pi/2$ implies $\sin(\alpha)=0$.
\end{proof}

\begin{fact}[Constant Curvature non-Euclidean Law of Sines and Law of Cosines]\label{thm:law_of_sines}
    Let $\K\neq 0$ and let $\newtarget{def:special_sine}{\sk[\cdot]}$ and $\ck[\cdot]$ denote the special sine and cosine, respectively, defined as $\sk[t] = \sin(\sqrt{\K}t)$ and $\ck[t] = \cos(\sqrt{\K}t)$ if $\K>0$, and as $\sk[t] = \sinh(\sqrt{-\K}t)$ and $\ck[t] = \cosh(\sqrt{-\K}t)$ if $\K<0$. Let $a, b,c$ be the lengths of the sides of a geodesic triangle defined on a manifold of constant sectional curvature $\K$. Let $\alpha, \beta, \gamma$ be the angles of the geodesic triangle, that are opposite to the sides $a, b, c$. The following holds:
    \begin{itemize}
        \item Law of sines:
            \[
                    \frac{\sin(\alpha)}{\sk[a]} =  \frac{\sin(\beta)}{\sk[b]} = \frac{\sin(\gamma)}{\sk[c]}.
            \] 
        \item Law of cosines:
            \[
                \ck[a] = \ck[b]\ck[c] + \cos(\alpha)\sk[b]\sk[c].
            \] 
    \end{itemize}
\end{fact}
    We refer to \citep{greenberg1993euclidean} for a proof of these classical theorems.

\subsection[Gradient deformation and gradient Lipschitzness of f]{Gradient deformation and gradient Lipschitzness of ${\protect\f}$} \label{sec:app_gradient_deformation_and_smoothness}

\cref{lemma:angle_deformation}, with $\tilde{\alpha} = \pi/2$, shows that $e_1 \perp e_j$, for $j\neq 1$. The rotational symmetry implies $e_i\perp e_j$ for $i\neq j$ and $i,j>1$. As in \cref{lemma:deformations}, let $x\in\BR$ be a point and assume without loss of generality that $\tilde{x}\in\operatorname{span}\{\hyperlink{def:e_i_canonical_basis}{\color{black}\tilde{e}_1}\}$ and $\nabla \f(\tilde{x})\in \operatorname{span}\{\hyperlink{def:e_i_canonical_basis}{\color{black}\tilde{e}_1}, \hyperlink{def:e_i_canonical_basis}{\color{black}\tilde{e}_2}\}$. It can be assumed without loss of generality because of the symmetries. So we can assume the dimension is $\n=2$. Using \cref{lemma:deformations} we obtain that $\tilde{\alpha}=0$ implies $\alpha=0$. Also $\tilde{\alpha} = \pi/2$ implies $\alpha=\pi/2$, so the adjoint of the differential of $\h^{-1}$ at $x$, $(\mathrm{d} \h^{-1})^\ast_x$ diagonalizes and has $e_1$ and $e_2$ as eigenvectors. We only need to compute the eigenvalues. The computation of the first one uses that the geodesic passing from $\xInit[\notilde]$ and $x$ can be parametrized as $\h^{-1}(\xInit+\arctan(\tilde{\lambda}\hyperlink{def:e_i_canonical_basis}{\color{black}\tilde{e}_1}))$ if $\K=1$ and $\h^{-1}(\xInit+\arctanh(\tilde{\lambda}\hyperlink{def:e_i_canonical_basis}{\color{black}\tilde{e}_1}))$ if $\K=-1$, by \eqref{eq:appendix_characterization_of_geodesic_map_and_metric}. The derivative of $\arctan(\cdot)$ or $\arctanh(\cdot)$ reveals that the first eigenvector, the one corresponding to $e_1$, is $1/(1+\K\norm{\tilde{x}^2})$, i.e., $\nabla \f(\tilde{x})_1 = \nabla \F(x)_1 / (1+\K\norm{\tilde{x}^2})$. For the second one, let $x=(x_1, 0)$ and $y=(y_1, y_2)$, with $y_1=x_1$ the second eigenvector results from the computation, for $\K=-1$: 
\begin{align*} 
 \begin{aligned}
     \lim_{y_2 \to 0} \frac{\dist(x,y)}{y_2}  &=\lim_{y_2 \to 0} \frac{1}{2y_2} \log\left(1+\frac{2y_2}{\sqrt{1-x_1^2}-y_2}\right) \\ 
    &\circled{1}[=] \lim_{y_2 \to 0}  \frac{\frac{2}{\sqrt{1-x_1^2}-y_2} + \frac{2y_2}{(\sqrt{1-x_1^2}-y_2)^2}}{2+\frac{4y_2}{\sqrt{1-x_1^2}-y_2}} \\
    &= \frac{1}{\sqrt{1-x_1^2}},
   \end{aligned}
\end{align*}
and for $\K=1$: 
\begin{align*} 
 \begin{aligned}
     \lim_{y_2 \to 0} \frac{\dist(x,y)}{y_2}  &=\lim_{y_2 \to 0} \frac{1}{y_2} \arccos\left(\frac{\sqrt{1+x_1^2}}{\sqrt{1 + x_1^2 + y_2^2}}\right) \\ 
    &\circled{2}[=] \lim_{y_2 \to 0} \frac{\sqrt{1 + x_1^2}}{1 + x_1^2 + y_2^2}\\
    &= \frac{1}{\sqrt{1+x_1^2}}.
   \end{aligned}
\end{align*}
So, since $x_1=\norm{\tilde{x}}$, we have $\nabla \f(\tilde{x})_2 = \nabla \F(x)_2/\sqrt{1+\K\norm{\tilde{x}}^2}$ for $\K\in\{1, -1\}$. We used L'Hôpital's rule in  $\circled{1}$ and $\circled{2}$. 

Also note that if $v \in \Tansp{x}\Mk$ is a vector normal to $\nabla \F(x)$, then $\tilde{v}$ is normal to $\nabla \f(x)$. It is easy to see this geometrically: Indeed, no matter how $\h$ changes the geometry, since it is a geodesic map, a geodesic in the direction of first-order constant increase of $\F$ is mapped via $\h$ to a geodesic in the direction of first-order constant increase of $\f$. And the respective gradients must be perpendicular to all these directions. Alternatively, this can be seen algebraically. Suppose first $\n=2$, then $v$ is proportional to $(\nabla \F(x)_2, -\nabla \F(x)_1) = (\sqrt{1+\K\norm{\tilde{x}}^2} \nabla \f(\tilde{x})_2, -(1+\K\norm{\tilde{x}}^2)\nabla \f(\xI[1]))$. And a vector $\tilde{v}'$ normal to $\nabla \f(x)$ must be proportional to $(-\nabla \f(x)_2, \nabla \f(x)_1)$. Let $\alpha$ be the angle formed by $v$ and $-e_1$, $\tilde{\alpha}$ the corresponding angle formed between $\tilde{v}$ and $-\hyperlink{def:e_i_canonical_basis}{\color{black}\tilde{e}_1}$, and $\tilde{\alpha}'$ the angle formed by $\tilde{v}'$ and $-\hyperlink{def:e_i_canonical_basis}{\color{black}\tilde{e}_1}$. Then we have, using the expression for the vectors proportional to $v$ and $\tilde{v}'$:
\[
    \sin(\alpha) = \frac{-\f(x)_2}{\sqrt{\nabla \f(x)_2^2+(1+\norm{x}^2)\nabla \f(x)_1^2}} \text{ and } \sin(\tilde{\alpha}') = \frac{-\f(x)_2}{\sqrt{\nabla \f(x)_2^2+\nabla \f(x)_1^2}} 
\] 
and using one equation on the other yields $\sin(\alpha) = \sin(\tilde{\alpha}')\sqrt{(1+\K\norm{\tilde{x}^2})/(1+\K\norm{\tilde{x}^2}\sin^2(\tilde{\alpha}'))}$, which after applying \cref{lemma:angle_deformation} we obtain $\sin(\tilde{\alpha}') = \sin(\tilde{\alpha})$ from which we conclude that $\tilde{\alpha}' = \tilde{\alpha}$ given that the angles are in the same quadrant. So $\tilde{v} \perp \nabla \f(x)$. In order to prove this for $\n\geq 3$ one can apply the reduction \eqref{claim:same_direction_when_projecting} to the case $\n=2$ that we obtain in the next section.  

Combining the results obtained so far in \cref{app:geometric_results}, we can prove \cref{lemma:deformations}. We continue by proving \cref{lemma:smoothness_of_transformed_function}, which will generalize the computations we just performed, in order to analyze the Hessian of $\f$ and provide smoothness. Then, in the next section, we combine the results in \cref{lemma:deformations} to prove \cref{prop:bounding_hyperplane}.

\begin{proof}\textbf{of \cref{lemma:deformations}.}\linkofproof{lemma:deformations}
    The lemma follows from Lemmas \ref{lemma:distances_hyperbolic}, \ref{lemma:distances_spherical}, \ref{lemma:angle_deformation} and the previous reasoning in this Section \ref{sec:app_gradient_deformation_and_smoothness}.
\end{proof}

\begin{proof}\textbf{of \cref{lemma:smoothness_of_transformed_function}.}\linkofproof{lemma:smoothness_of_transformed_function}
    Recall $\F:\Mk\to\R$ is a function defined on a manifold of constant sectional curvature with a point $\xast[\notilde] \in\BR$ such that $\nabla \F(\xast[\notilde]) = 0$ and we call $\RR$ an upper bound on $\dist(\xInit[\notilde], \xast[\notilde])$, for an initial point $\xInit[\notilde]$. We first note that for $\F$, which is assumed to be twice differentiable, g-convex and $L$-smooth in $\Mk$, we have that $\norm{\nabla^2 \F(x)} \leq L$ for any $x\in\Mk$. Indeed, for twice differentiable Euclidean convex and $L$-smooth functions, this fact is know. By restricting $\F(x)$ to a geodesic, we obtain one such Euclidean function in dimension $1$ to which we can apply the aforementioned fact. After applying it to all geodesics going through $x$, we obtain the Riemannian fact.

    We will compute the Hessian of $\f=\F\circ \h^{-1}$ and we will bound its spectral norm for any point $\tilde{x}\in\X$. We can assume without loss of generality that $\n=2$ and $\tilde{x} = (\tilde{\ell}, 0)$, for $\tilde{\ell} > 0$ (the case $\tilde{\ell}=0$ is trivial), since there is a rotational symmetry with $e_1$ as axis. This means that by rotating we could align the top eigenvector of the Hessian at a point so that it is in $\operatorname{span}\{e_1, e_2\}$. Let $\tilde{y} = (y_1, y_2) \in \X$ be another point, with $y_1=\tilde{\ell}$. We can also assume that $y_2>0$ without loss of generality, because of our symmetry. Our approach will be the following. We know by \cref{lemma:deformations}.b and by the beginning of this Section \ref{sec:app_gradient_deformation_and_smoothness} that the adjoint of the differential of $\h^{-1}$ at $y$, $(\mathrm{d} \h^{-1})^\ast_y$ has $\exponinv{y}(\xInit[\notilde])$ and a normal vector to it as eigenvectors. Their corresponding eigenvalues are $1/(1+\K\norm{\tilde{y}}^2)$ and $1/\sqrt{1+\K\norm{\tilde{y}}^2}$, respectively. Consider the basis $\{e_1, e_2\}$ of $\Tansp{x}\Mk$ as defined at the beginning of this section, i.e., where $e_1$ is a unit vector proportional to $-\exponinv{x}(\xInit[\notilde])$ and $e_2$ is the normal vector to $e_1$ that makes the basis orthonormal. Consider this basis being transported to $y$ using parallel transport and denote the result $\{v_y, v_y^\perp\}$. Assume we have the gradient $\nabla \F(y)$ written in this basis. Then we can compute the gradient of $\f$ at $y$ by applying $(\mathrm{d}\h^{-1})^\ast_y$ to $\nabla \F(y)$. In order to do that, we compose the change of basis from $\{v_y, v_y^\perp\}$ to the basis of eigenvectors of $(\mathrm{d} \h^{-1})^\ast_y$, then we apply a diagonal transformation given by the eigenvalues and finally we change the basis to $\{\hyperlink{def:e_i_canonical_basis}{\color{black}\tilde{e}_1}, \hyperlink{def:e_i_canonical_basis}{\color{black}\tilde{e}_2}\}$. Once this is done, we can differentiate with respect to $y_2$ in order to compute a column of the Hessian.  Let $\tilde{\alpha}$ be the angle formed by the vectors $\tilde{y}$ and $\tilde{x}$. Note that $\tilde{\alpha} = \arctan(y_2/y_1)$. Let $\tilde{\gamma}$ be the angle formed by the vectors $(\tilde{y}-\tilde{x})$ and $-\tilde{y}$. That is, the angle $\tilde{\gamma}=\pi-\angle \tilde{x} \tilde{y} \xInit$. Since $v_y^\perp$ is the parallel transport of $e_2^\perp$, the angle between $v_y^\perp$ and the vector $\exponinv{y}(\xInit[\notilde])$ is $\gamma$. Note we use the same convention as before for the angles, i.e., $\gamma$ is the corresponding angle to $\tilde{\gamma}$, meaning that if $\gamma$ is the angle between two intersecting geodesics in $\BR$, then $\tilde{\gamma}$ is the angle between the respective corresponding geodesics in $\X$. Note the first change of basis is a rotation and that the angle of rotation is $\gamma-\pi/2$. The last change of basis is a rotation with angle equal to the angle formed by a vector $\tilde{v}$  normal to $-\tilde{y}$ ( $\tilde{v}$ is the one such that $-\tilde{y}\times \tilde{v} >0$) and the vector $\hyperlink{def:e_i_canonical_basis}{\color{black}\tilde{e}_2}$. This vector is equal to $\tilde{\alpha}$. So we have
\begin{equation}\label{eq:equality_gradients}
    \nabla \f(\tilde{y}) = 
\begin{pmatrix}
\cos(\tilde{\alpha}) & -\sin(\tilde{\alpha}) \\ \sin(\tilde{\alpha}) & \cos(\tilde{\alpha})
\end{pmatrix}
\begin{pmatrix}
    \frac{1}{1+\K(y_1^2+y_2^2)} & 0   \\ 0 & \frac{1}{\sqrt{1+\K(y_1^2+y_2^2)}}
\end{pmatrix}
\begin{pmatrix}
    \sin(\gamma) & -\cos(\gamma) \\ \cos(\gamma) & \sin(\gamma)
\end{pmatrix}
\nabla \F(y)
\end{equation}

    We want to take the derivative of this expression with respect to $y_2$ and we want to evaluate it at $y_2=0$. Let the matrices above be $A$, $B$ and $C$ so that $\nabla \f(\tilde{y}) = ABC\nabla \F(y)$. Using \cref{lemma:deformations}.b we have
\begin{align} \label{eq:aux_comp_angle_gamma}
   \begin{aligned}
       \sin(\gamma) &= \sin(\tilde{\gamma})\sqrt{\frac{1+\K(y_1^2+y_2^2)}{1+\K(y_1^2+y_2^2)\sin^2(\tilde{\gamma})}} \circled{1}[=] \cos(\tilde{\alpha})\sqrt{\frac{1+\K(y_1^2+y_2^2)}{1+\K(y_1^2+y_2^2)\cos^2(\tilde{\alpha})}}, \\
       \cos(\gamma) &= -\sin(\tilde{\alpha})\sqrt{\frac{1}{1+\K(y_1^2+y_2^2)\cos^2(\tilde{\alpha})}},
   \end{aligned}
\end{align}
where $\circled{1}$ follows from $\sin(\tilde{\gamma}) = \sin(\tilde{\alpha}+\pi/2) = \cos(\tilde{\alpha})$. Now we compute some quantities
\[
    \left.A\right|_{y_2=0} = I,   \left.B\right|_{y_2=0} = 
\begin{pmatrix}
    \frac{1}{1+\K y_1^2} & 0 \\ 0 & \frac{1}{\sqrt{1+\K y_1^2}}
\end{pmatrix},
\left.C\right|_{y_2=0} = I,
\]        
\[ 
        \left.\frac{\partial A}{\partial y_2}\right|_{y_2=0} = \left.\frac{\partial \tilde{\alpha}}{\partial y_2}\right|_{y_2=0} \cdot
\begin{pmatrix}
    0 & -1 \\ 1 & 0
\end{pmatrix}
\circled{1}[=]
\begin{pmatrix}
    0 & \frac{-1}{y_1} \\ \frac{1}{y_1} & 0
\end{pmatrix}
,
\]
\[
        \left.\frac{\partial B}{\partial y_2}\right|_{y_2=0} =
\left.
\begin{pmatrix}
    \frac{2\K y_2}{(1+\K(y_1^2+y_2^2))^2} & 0 \\ 0 & \frac{2\K y_2}{2(1+\K(y_1^2+y_2^2))^{3/2}} 
\end{pmatrix}
\right|_{y_2=0} =
\begin{pmatrix}
    0 & 0 \\ 0 & 0
\end{pmatrix}
,
\]
\[
    \left.\frac{\partial C}{\partial y_2}\right|_{y_2=0} \circled{2}[=]
\begin{pmatrix}
    0 & \frac{1}{y_1\sqrt{1+\K y_1^2}} \\ \frac{-1}{y_1\sqrt{1+\K y_1^2}} & 0
\end{pmatrix}
.
\]
Equalities $\circled{1}$ and $\circled{2}$ follow after using \eqref{eq:aux_comp_angle_gamma}, $\tilde{\alpha} = \arctan(\frac{y_2}{y_1})$ and taking derivatives. Now we differentiate \eqref{eq:equality_gradients} with respect to $y_2$ and evaluate to $y_2=0$ using the chain rule. The result is 
\begin{align*}
   \begin{aligned}
\begin{pmatrix}
     \nabla^2 \f(\tilde{x})_{12} \\ \nabla^2 \f(\tilde{x})_{22} 
\end{pmatrix}
        &= \left.\left(\frac{\partial A}{\partial y_2} BC\nabla \F(x) + A\frac{\partial B}{\partial y_2} C\nabla \F(x) +AB\frac{\partial C}{\partial y_2} \nabla \F(x) + ABC\frac{\partial \nabla \F(x)}{\partial y_2}  \right)\right|_{y_2=0} \\
        &=
\begin{pmatrix}
    \frac{-\nabla \F(x)_{2}}{y_1\sqrt{1+\K y_1^2}}  \\ \frac{\nabla \F(x)_{1}}{y_1(1+\K y_1^2)}
\end{pmatrix}
+
\begin{pmatrix}
    0 \\ 0
\end{pmatrix}
+
\begin{pmatrix}
 \frac{\nabla \F(x)_{2}}{y_1(1+\K y_1^2)^{3/2}}  \\ \frac{-\nabla \F(x)_{1}}{y_1(1+\K y_1^2)} 
\end{pmatrix}
+
\begin{pmatrix}
   \frac{\nabla^2 \F(x)_{12}}{(1+\K y_1^2)^{3/2}}  \\ \frac{\nabla^2 \F(x)_{22}}{1+\K y_1^2} 
\end{pmatrix}
   \end{aligned}
\end{align*}
Computing the other column of the Hessian is easier. We can just consider \eqref{eq:equality_gradients} with $\tilde{\alpha} = 0$, $\gamma=\pi/2$ and vary $y_1$. Taking derivatives with respect to $y_1$ gives
\[
\begin{pmatrix}
    \nabla^2 \f(\tilde{x})_{11}   \\ \nabla^2 \f(\tilde{x})_{21}
\end{pmatrix}
=
\begin{pmatrix}
   \frac{-2\K y_1\nabla \F(x)_{1}}{(1+\K y_1^2)^{2}}  \\ \frac{-\K y_1\nabla \F(x)_{2}}{(1+\K y_1^2)^{3/2}} 
\end{pmatrix}
+
\begin{pmatrix}
    \frac{\nabla^2 \F(x)_{11}}{(1+\K y_1^2)^{2}} \\ \frac{\nabla^2 \F(x)_{21}}{(1+\K y_1^2)^{3/2}} 
\end{pmatrix}
.
\]
Note in the computations of both of the columns of the Hessian we have used 
\[
    \frac{\partial{\nabla \F(y)_i}}{\partial y_1} = \nabla \F(x)_{i1} \cdot \frac{1}{1+\K y_1^2} \quad \text{ and } \quad\left.\frac{\partial{\nabla \F(y)_i}}{\partial y_2}\right|_{y_2=0} = \nabla \F(x)_{i2} \cdot \frac{1}{\sqrt{1+\K y_1^2}},
\]
for $i=1,2$. The eigenvalues of the adjoint of the differential of $\h^{-1}$ appear as a factor because we are differentiating with respect to the geodesic in $\X$ which moves at a different speed than the corresponding geodesic in $\BR$. Note as well, as a sanity check, that the cross derivatives are equal, since 
    \[
    -\frac{1}{y_1\sqrt{1+\K y_1^2}} + \frac{1}{y_1(1+\K y_1^2)^{3/2}} = \frac{1}{y_1\sqrt{1+\K y_1^2}}\left(-1 + \frac{1}{1+\K y_1^2}\right) = \frac{-\K y_1}{(1+\K y_1^2)^{3/2}}.
    \] 
    Finally, we bound the new smoothness constant $\Ltilde$ by bounding the spectral norm of this Hessian. First note that using $y_1=\tilde{\ell}$ we have that $\frac{1}{\sqrt{1+\K\tilde{\ell}^2}} = \ck[\ell]$ and then $y_1\sqrt{\abs{K}}=\tilde{\ell}\sqrt{\abs{K}} = \sk[\ell]/\ck[\ell]$ so for $\K=-1$ it is $\tilde{\ell} = \tanh(\ell)$ and for $\K=1$ it is $\tilde{\ell} = \tan(\ell)$, where $\ell = \dist(x, \xInit[\notilde]) < \RR$. We have that since there is a point $\xast[\notilde] \in \BR$ such that $\nabla \F(\xast[\notilde]) = 0$ and $\F$ is $\L$-smooth, then it is $\norm{\nabla \F(x)} \leq 2L\bar{R}$, where $\bar{R} \defi \dist(\xInit[\notilde],\xast[\notilde])$. Similarly, by $\L$-smoothness $\abs{\nabla^2 \F(x)_{ij}} \leq \L$. We are now ready to prove $O(L)$-Lipschitzness of $\nabla\f$. 
\begin{align*} 
 \begin{aligned}
     \Ltilde^2& \defi \max_{\tilde{x}\in\X}\norm{\nabla^2 \f(\tilde{x})}_2^2 \\
     & \leq \max_{\tilde{x}\in\X}\norm{\nabla^2 \f(\tilde{x})}_F^2 = \max_{\tilde{x}\in\X}\{(\nabla^2 \f(\tilde{x})_{11})^2+2(\nabla^2 \f(\tilde{x})_{12})^2+(\nabla^2 \f(\tilde{x})_{22})^2\} \\
     &\leq \L^2([2\sqrt{\abs{\K}} \bar{R} \sk[\RR]\ck[\RR][3]+\ck[\RR][4]]^2+2[\sqrt{\abs{\K}}\bar{R}\sk[\RR]\ck[\RR][2]+\ck[\RR][3]]^2+\ck[\RR][4])
   \end{aligned}
\end{align*}
and this can be bounded by $ 44 \L^2 \max\{1,\abs{\K}\bar{R}^2\}$ if $\K=1$ and $44 \L^2 \max\{1,\abs{\K}\bar{R}^2\} \ck[\RR][8]$ if $\K=-1$. It is $O(\L^2(\abs{\K}\bar{R}^2+1))$ for general $\K$, where recall, we are treating $\RR\sqrt{\abs{\K}}=\bigo{1}$. The result follows, since boundedness of operator norm of $\nabla^2 f$ by $\Ltilde$ implies $\nabla f$ has $\Ltilde$-Lipschitz gradients, and by definition $\Rglobal \geq \bar{R}$.
\end{proof}

\subsection[Proof of Lemma \ref{prop:bounding_hyperplane}]{Proof of \cref{prop:bounding_hyperplane}} \label{app:sec_proof_of_bounding_hyperplane}

\begin{proof}\linkofproof{prop:bounding_hyperplane}
    Assume for the moment that the dimension is $\n=2$. We can assume without loss of generality that $\tilde{x} = (\tilde{\ell}, 0)$. We are given two vectors, that are the gradients $\nabla \F(x)$, $\nabla \f(\tilde{x})$ and a vector $w\in \Tansp{x}\Mk$. Let $\newtarget{def:auxiliary_angle_delta}{\deltatildeAux}$ be the angle between $\tilde{w}$ and $-\tilde{x}$. Let $\deltaAux$ be the corresponding angle, i.e., the angle between $w$ and $u \defi \exponinv{x}(\xInit[\notilde])$. Let $\alpha$ be the angle in between $\nabla \F(x)$ and $u$. Let $\tilde{\beta}$ be the angle in between $\nabla \f(\tilde{x})$ and $-x$. $\tilde{\alpha}$ and $\beta$ are defined similarly. We claim
\begin{equation} \label{inner_product_comparison}
    \frac{\innp{\frac{\nabla \F(x)}{\norm{\nabla \F(x)}}, \frac{w}{\norm{w}}}}{\innp{\frac{\nabla \f(\tilde{x})}{\norm{\nabla \f(\tilde{x})}}, \frac{\tilde{w}}{\norm{\tilde{w}}}}} = \sqrt{\frac{1+\K\tilde{\ell}^2}{(1+\K\tilde{\ell}^2 \sin^2(\deltatildeAux))(1+\K\tilde{\ell}^2 \cos^2(\tilde{\beta}))}}.
\end{equation}

    Let's see how to arrive to this expression. By \cref{lemma:deformations}.c we have
    \begin{equation}\label{eq:relationship_tangents_gradients}
        \tan(\alpha) = \frac{\tan(\tilde{\beta})}{\sqrt{1+\K\tilde{\ell}^2}}.
    \end{equation}
    From this relationship we deduce
    \begin{equation} \label{eq:relationship_cosines_gradients}
        \cos(\alpha)  = \cos(\tilde{\beta}) \sqrt{\frac{1+\K\tilde{\ell}^2}{1+\K\tilde{\ell}^2\cos^2(\tilde{\beta})}},
    \end{equation}
    that comes from squaring \eqref{eq:relationship_tangents_gradients}, reorganizing terms and noting that $\sign(\cos(\alpha)) =\sign(\cos(\tilde{\beta}))$ which is implied by \cref{lemma:deformations}.c. We are now ready to prove the claim \eqref{inner_product_comparison} (for $\n=2$). We have  
\begin{align*} 
   \begin{aligned}
       \frac{\innp{\frac{\nabla \F(x)}{\norm{\nabla \F(x)}}, \frac{w}{\norm{w}}}}{\innp{\frac{\nabla \f(\tilde{x})}{\norm{\nabla \f(\tilde{x})}}, \frac{\tilde{w}}{\norm{\tilde{w}}}}} &= \frac{\cos(\alpha-\deltaAux)}{\cos(\tilde{\beta}-\deltatildeAux)} \\
       &\circled{2}[=] \frac{\cos(\deltaAux)+\tan(\alpha)\sin(\deltaAux)}{\cos(\tilde{\beta})\cos(\deltatildeAux)+\sin(\tilde{\beta})\sin(\deltatildeAux)}\cos(\alpha) \\
       &\circled{3}[=] \frac{\frac{\cos(\deltatildeAux)}{\sqrt{1+\K\tilde{\ell}^2\sin^2(\deltatildeAux)}}+\frac{\tan(\tilde{\beta})}{\sqrt{1+\K\tilde{\ell}^2}}\frac{\sin(\deltatildeAux)\sqrt{1+\K\tilde{\ell}^2}}{\sqrt{1+\K\tilde{\ell}^2\sin^2(\deltatildeAux)}}}{\cos(\tilde{\beta})\cos(\deltatildeAux)+\sin(\tilde{\beta})\sin(\deltatildeAux)}\cos(\tilde{\beta}) \sqrt{\frac{1+\K\tilde{\ell}^2}{1+\K\tilde{\ell}^2\cos^2(\tilde{\beta})}} \\
       &\circled{4}[=] \sqrt{\frac{1+\K\tilde{\ell}^2}{(1+\K\tilde{\ell}^2\sin^2(\deltatildeAux))(1+\K\tilde{\ell}^2\cos^2(\tilde{\beta}))}}. \\
   \end{aligned}
\end{align*}
    Equality $\circled{1}$ follows by the definition of $\alpha, \deltaAux, \deltatildeAux$, and $\tilde{\beta}$. In $\circled{2}$, we used trigonometric identities. In $\circled{3}$ we used \cref{lemma:deformations}.b, \eqref{eq:relationship_tangents_gradients} and \eqref{eq:relationship_cosines_gradients}. By reordering the expression, the denominator cancels out with a factor of the numerator in $\circled{4}$.

In order to work with arbitrary dimension, we note it is enough to prove it for $\n=3$, since in order to bound
\[
\frac{\innp{\frac{\nabla \F(x)}{\norm{\nabla \F(x)}}, \frac{v}{\norm{v}}}}{\innp{\frac{\nabla \f(\tilde{x})}{\norm{\nabla \f(\tilde{x})}}, \frac{\tilde{v}}{\norm{\tilde{v}}}}},
\] 
it is enough to consider the following submanifold
\[
    \mathcal{M}_{\K}' \defi  \expon{x}(\operatorname{span}\{v, \exponinv{x}(\xInit[\notilde]), \nabla \F(x)\}).
\]
    for an arbitrary vector $v\in \Tansp{x}\Mk$ and a point $x$ defined as above. The case $\n=3$ can be further reduced to the case $\n=2$ in the following way. Suppose $\mathcal{M}_{\K}'$ is a three dimensional manifold (if it is one or two dimensional there is nothing to do). Define the following orthonormal basis of $\Tansp{x}\Mk$, $\{e_1, e_2, e_3\}$ where $e_1=-\exponinv{x}(\xInit[\notilde])/\norm{\exponinv{x}(\xInit[\notilde])}$, $e_2$ is a unit vector, normal to $e_1$ such that $e_2\in\operatorname{span}\{e_1, \nabla \F(x)\}$. And $e_3$ is a vector that completes the orthonormal basis. In this basis, let $v$ be parametrized by $\norm{v}(\sin(\deltaAux), \cos(\nu)\cos(\deltaAux), \sin(\nu)\cos(\deltaAux))$, where $\deltaAux$ can be thought as the angle between the vector $v$ and its projection onto the plane $\operatorname{span}\{e_2, e_3\}$ and $\nu$ can be thought as the angle between this projection and its projection onto $e_2$. Similarly we parametrize $\tilde{v}$ by $\norm{\tilde{v}}(\sin(\deltatildeAux), \cos(\tilde{\nu})\cos(\deltatildeAux), \sin(\tilde{\nu})\cos(\deltatildeAux))$, where the base used is the analogous base to the previous one, i.e., the vectors $\{\hyperlink{def:e_i_canonical_basis}{\color{black}\tilde{e}_1}, \hyperlink{def:e_i_canonical_basis}{\color{black}\tilde{e}_2}, \hyperlink{def:e_i_canonical_basis}{\color{black}\tilde{e}_3}\}$. Taking into account that $e_2 \perp e_1$, $e_3\perp e_1$, $\hyperlink{def:e_i_canonical_basis}{\color{black}\tilde{e}_2} \perp \hyperlink{def:e_i_canonical_basis}{\color{black}\tilde{e}_1}$, $\hyperlink{def:e_i_canonical_basis}{\color{black}\tilde{e}_3}\perp \hyperlink{def:e_i_canonical_basis}{\color{black}\tilde{e}_1}$, and the fact that $e_1$ is parallel to $-\expon{x}(\xInit[\notilde])$, by the radial symmetry of the geodesic map we have that $\nu = \tilde{\nu}$. Also, by looking at the submanifold $\expon{x}(\operatorname{span}\{e_1,v\})$ and using \cref{lemma:deformations}.b we have 
\[
    \sin(\deltaAux) = \sin(\deltatildeAux)\sqrt{\frac{1+\K\tilde{\ell}^2}{1+\K\tilde{\ell}^2\sin(\deltatildeAux)}}.
\]
    If we want to compare $\innp{\nabla \F(x), v}$ with $\innp{\nabla \f(\tilde{x}), \tilde{v}}$ we should be able to just zero out the third components of $v$ and $\tilde{v}$ and work in $\n=2$. But in order to completely obtain a reduction to the two-dimensional case we studied a few paragraphs above, we would need to prove that if we call $w\defi(\sin(\deltaAux), \cos(\nu)\cos(\deltaAux), 0)$ the vector $v$ with the third component made $0$, then $\tilde{w}$ is in the same direction of the vector $\tilde{v}$, when the third component is made $0$. The norm of these two vectors will not be the same, however. Let $\tilde{w}'=(\sin(\deltatildeAux), \cos(\nu) \cos(\deltatildeAux), 0)$ be the vector $\tilde{v}$ when the third component is made $0$. Then 
\begin{equation} \label{eq:norms_of_projections_from_3d_vectors}
    \norm{w} = \norm{v} \sqrt{\sin^2(\deltaAux)+\cos^2(\deltaAux)\cos^2(\nu)} \text{ and } \norm{\tilde{w}'} = \norm{\tilde{v}}\sqrt{\sin^2(\deltatildeAux) + \cos^2(\deltatildeAux)\cos^2(\nu)}.
\end{equation}
But indeed, we claim 
    \begin{equation}\label{claim:same_direction_when_projecting}
    \tilde{w} \text{ and } \tilde{w}' \text{ have the same direction.}
\end{equation}
     This is easy to see geometrically: since we are working with a geodesic map, the submanifolds $\expon{x}(\operatorname{span}\{v, e_3\})$ and $\expon{x}(\operatorname{span}\{e_1, e_2\})$ contain $w$. Similarly the submanifolds $x+\operatorname{span}\{\tilde{v}, \hyperlink{def:e_i_canonical_basis}{\color{black}\tilde{e}_3}\}$ and $x+\operatorname{span}\{\hyperlink{def:e_i_canonical_basis}{\color{black}\tilde{e}_1}, \hyperlink{def:e_i_canonical_basis}{\color{black}\tilde{e}_2}\}$ contain $\tilde{w}'$. If the intersections of each of these pair of manifolds is a geodesic then the geodesic map must map one intersection to the other one, implying $\tilde{w}$ is proportional to $\tilde{w}'$. If the intersections are degenerate the case is trivial. Alternatively, one can prove this fact algebraically after some computations. It will be convenient for the rest of the proof so we will also include it here. If we call $\newtarget{def:auxiliary_delta_ast}{\deltaastAux}$ and $\newtarget{def:auxiliary_delta_prime}{\deltaptildeAux}$ the angles formed by, respectively, the vectors $e_2$ and $w$, and the vectors $\hyperlink{def:e_i_canonical_basis}{\color{black}\tilde{e}_2}$ and $\tilde{w}'$, then we have $\tilde{w}'$ is proportional to $\tilde{w}$ if $\deltaptildeAux=\deltaasttildeAux$, or equivalently $\deltapAux=\deltaastAux$. Using the definitions of $w$ and $\tilde{w}'$ we have
\begin{align*} 
   \begin{aligned}
       \sin(\deltaastAux) = \sin\left(\arctan\left(\frac{\sin(\deltaAux)}{\cos(\nu)\cos(\deltaAux)}\right)\right) = \frac{\tan(\deltaAux)/\cos(\nu)}{(\tan(\deltaAux)/\cos(\nu))^2+1} \\
       = \frac{\sin(\deltaAux)}{\sqrt{\sin^2(\deltaAux)+\cos^2(\nu)\cos^2(\deltaAux)}},
   \end{aligned}
\end{align*}
and analogously
\begin{align} \label{eq:aux_sine_of_angle_3d}
   \begin{aligned}
       \sin(\deltaptildeAux) = \sin\left(\arctan\left(\frac{\sin(\deltatildeAux)}{\cos(\nu)\cos(\deltatildeAux)}\right)\right) = \frac{\tan(\deltatildeAux)/\cos(\nu)}{(\tan(\deltatildeAux)/\cos(\nu))^2+1} \\
       = \frac{\sin(\deltatildeAux)}{\sqrt{\sin^2(\deltatildeAux)+\cos^2(\nu)\cos^2(\deltatildeAux)}}.
   \end{aligned}
\end{align}
Using \cref{lemma:deformations}.b for the pairs $\deltapAux$, $\deltaptildeAux$ and $\deltaastAux$, $\deltaasttildeAux$, and the equations above we obtain
\begin{align*} 
   \begin{aligned}
       \sin(\deltaastAux) = \frac{\sin(\deltatildeAux)\sqrt{\frac{1+\K\tilde{\ell}^2}{1+\K\tilde{\ell}^2\sin^2(\deltatildeAux)}}}{\sqrt{\sin^2(\deltatildeAux)\frac{1+\K\tilde{\ell}^2}{1+\K\tilde{\ell}^2\sin^2(\deltatildeAux)}+\cos^2(\nu)\frac{\cos^2(\deltatildeAux)}{1+\K\tilde{\ell}^2\sin^2(\deltatildeAux)}}} = \frac{\sin(\deltatildeAux)\sqrt{1+\K\tilde{\ell}^2}}{\sqrt{\sin^2(\deltatildeAux)(1+\K\tilde{\ell}^2)+\cos^2(\nu)\cos^2(\deltatildeAux)}},
   \end{aligned}
\end{align*}
and 
\begin{align*} 
   \begin{aligned}
       \sin(\deltapAux) = \frac{\sin(\deltatildeAux)}{\sqrt{\sin^2(\deltatildeAux)+\cos^2(\nu)\cos^2(\deltatildeAux)}}\sqrt{\frac{1+\K\tilde{\ell}^2}{1+\K\tilde{\ell}^2\left( \frac{\sin^2(\deltatildeAux)}{\sin^2(\deltatildeAux)+\cos^2(\nu)\cos^2(\deltatildeAux)}\right)}
},
   \end{aligned}
\end{align*}
The two expressions on the right hand side are equal. This implies $\sin(\deltapAux) = \sin(\deltaastAux)$. Since the angles were in the same quadrant we have $\deltapAux =\deltaastAux$.

We can now come back to the study of $\frac{\innp{\nabla \F(x), v}}{\innp{\nabla \f(\tilde{x}), \tilde{v}}}$. By  \eqref{eq:norms_of_projections_from_3d_vectors} we have
\begin{align*} 
   \begin{aligned}
   \frac{\innp{\nabla \F(x), v}}{\innp{\nabla \f(\tilde{x}), \tilde{v}}} = \frac{\norm{\nabla \F(x)}}{\norm{\nabla \f(\tilde{x})}}\frac{\norm{v}}{\norm{\tilde{v}}}\frac{\innp{\frac{\nabla \F(x)}{\norm{\nabla \F(x)}}, \frac{w}{\norm{w}}}}{\innp{\frac{\nabla \f(\tilde{x})}{\norm{\nabla \f(\tilde{x})}}, \frac{\tilde{w}}{\norm{\tilde{w}}}}} \frac{\sqrt{\sin^2(\deltaAux)+\cos^2(\deltaAux)\cos^2(\nu)}}{\sqrt{\sin^2(\deltatildeAux)+\cos^2(\deltatildeAux)\cos^2(\nu)}} \end{aligned}
\end{align*}
We now operate the last two fractions. Using \eqref{inner_product_comparison} and \eqref{eq:norms_of_projections_from_3d_vectors} we get that the product of the last two fractions above is equal to
\begin{align*} 
   \begin{aligned}
       \sqrt{\frac{1+\K\tilde{\ell}^2}{(1+\K\tilde{\ell}^2\sin^2(\deltaasttildeAux))(1+\K\tilde{\ell}^2\cos^2(\tilde{\beta}))}} \frac{\sqrt{\sin^2(\deltatildeAux)\frac{1+\K\tilde{\ell}^2}{(1+\K\tilde{\ell}^2\sin^2(\deltatildeAux))} +\cos^2(\nu)\frac{\cos^2(\deltatildeAux)}{1+\K\tilde{\ell}^2\sin(\deltatildeAux)}}}{\sin^2(\deltatildeAux)+\cos^2(\deltatildeAux)\cos^2(\nu)}
   \end{aligned}
\end{align*}
which after using \eqref{eq:aux_sine_of_angle_3d} (recall $\deltaasttildeAux = \deltaptildeAux$), and simplifying it yields
\[
\sqrt{\frac{1+\K\tilde{\ell}^2}{(1+\K\tilde{\ell}^2\sin^2(\deltatildeAux))(1+\K\tilde{\ell}^2\cos^2(\tilde{\beta}))}}.
\] 
So finally we have
\[
\frac{\innp{\nabla \F(x), v}}{\innp{\nabla \f(\tilde{x}), \tilde{v}}} = \frac{\norm{\nabla \F(x)}}{\norm{\nabla \f(\tilde{x})}}\frac{\norm{v}}{\norm{\tilde{v}}}
\sqrt{\frac{1+\K\tilde{\ell}^2}{(1+\K\tilde{\ell}^2\sin^2(\deltatildeAux))(1+\K\tilde{\ell}^2\cos^2(\tilde{\beta}))}}.
\] 
In order to bound the previous expression, we now use \cref{lemma:deformations}.c and \cref{lemma:deformations}.a, and bound $\sin^2(\deltatildeAux)$ and $\cos^2(\tilde{\beta})$ by $0$ or $1$ depending on the inequality. Recall that, by \eqref{eq:R_tilde_vs_R} we have $1/\sqrt{1+\K\tilde{\ell}^2} = \ck[\ell]$, for $\ell=\dist(x, \xInit[\notilde]) \leq \RR$. And $\tilde{\ell}=\norm{\tilde{x}}$. Let's proceed. We obtain, for $\K=-1$
\[
    \cosh^{-3}(\RR)\leq \frac{1}{\cosh^2(\ell)} \cdot 1 \cdot \frac{1}{\cosh(\ell)} \leq \frac{\innp{\nabla \F(x), v}}{\innp{\nabla \f(\tilde{x}), \tilde{v}}}  \leq \frac{1}{\cosh(\ell)} \cdot \cosh^2(\ell) \cdot \cosh(\ell) \leq \cosh^2(\RR).
\] 
and for $\K=1$ it is 
\[
    \cos^{2}(\RR)\leq \frac{1}{\cos(\ell)} \cdot \cos^2(\ell) \cdot \cos(\ell) \leq \frac{\innp{\nabla \F(x), v}}{\innp{\nabla \f(\tilde{x}), \tilde{v}}}  \leq \frac{1}{\cos^2(\ell)} \cdot 1 \cdot \frac{1}{\cos(\ell)} \leq \cos^{-3}(\RR).
\] 

The first part of \cref{prop:bounding_hyperplane} follows, for $\newtarget{def:gamma_p}{\gammap} = \cosh^{-3}(\RR)$ and $\newtarget{def:gamma_n}{\gamman} = \cosh^{-2}(\RR)$ when $\K=-1$, and $\gammap = \cos^2(\RR)$ and $\gamman = \cos^3(\RR)$ when $\K=1$.

The second part of \cref{prop:bounding_hyperplane} follows readily from the first one and g-convexity of $\F$, as in the following. It holds
\[
    \f(\tilde{x}) + \frac{1}{\gamman}\innp{\nabla \f(\tilde{x}), \tilde{y}- \tilde{x}} \circled{1}[\leq] \F(x) + \innp{\nabla \F(x), y\riemMinus x} \circled{2}[\leq] \F(y) = \f(\tilde{y}),
\] 
and
\[
    \f(\tilde{x}) + \gammap\innp{\nabla \f(\tilde{x}), \tilde{y}-\tilde{x}} \circled{3}[\leq] \F(x) + \innp{\nabla \F(x), y\riemMinus x} \circled{4}[\leq] \F(y) = \f(\tilde{y}),
\] 
where $\circled{1}$ and $\circled{3}$ hold if $\innp{\nabla \f(\tilde{x}), \tilde{y}-\tilde{x}} \leq 0$ and $\innp{\nabla \f(\tilde{x}), \tilde{y}-\tilde{x}} \geq 0$, respectively, by the first part of this theorem. Inequalities $\circled{2}$ and $\circled{4}$ hold by g-convexity of $\F$.

\end{proof}

\section[Constants depending on R and K, and comparisons]{Constants depending on {$\protect\RR$} and {$\protect\K$}, and comparisons}\label{app:constants}
We discuss the value of the constants of our algorithms in \cref{remark:constants} and discuss recent hardness results in \cref{remark:related_work_hardness_hyperbolic}. But we start by proving a relevant result that says that the condition number of an $\L$-smooth and $\mu$-strongly g-convex function $\F:\BR\to\R$ is lower bounded by a term depending on $\RR$ and $\K$, where the condition number is defined by $\L/\mu$. This is unlike in the Euclidean case, for which there are functions with condition number $1$.

In particular, we show that the function $x\mapsto \frac{1}{2}\dist(x,\xInit[\notilde])^2$ has minimum condition number on $\BR$, and is $(\RR\sqrt{\abs{\K}}\cotk[\RR])^{-\sign(\K)}$, where $\newtarget{def:special_cotangent}{\cotk[\RR]}$ is the special cotangent that is $\cot(\sqrt{\abs{\K}}\RR)$ if $\K>0$ and $\coth(\sqrt{\abs{\K}}\RR)$ if $\K<0$. And $\newtarget{def:sign_of_a_number}{\sign}(\K)$ is $\K/\abs{\K}$ for $\K\neq 0$. The fact about the condition number of $\frac{1}{2}\dist(x,\xInit[\notilde])^2$ can be obtained from the proof of \cref{strong_convexity_and_smoothness_of_ell_2}, and actually the fact per se as a comparison geometry theorem that uses that the inequality there is satisfied with equality in the constant curvature case. However, we recover the computation of this condition number while proving the proposition. 

\begin{proposition}\label{prop:lower_bound_on_condition_number}
    Let $F:\Mk\to\R$ be an $\L$-smooth and $\mu$-strongly convex function on $\BR\subset\Mk$. Assume $F$ is twice differentiable with continuous Hessian. Then, the condition number $\L/\mu$ of $F$ on $\BR$ is at least the condition number of the function $\frac{1}{2}\dist(x,\xInit[\notilde])^2$ on $\BR$.
\end{proposition}

\begin{proof}
    As we have done before, we can assume $\K\in\{1,-1\}$ because the other cases can be reduced to this one by a rescaling, cf. \cref{remark:rescaling_of_K}. Recall that by definition of $\Mk$ and $\BR$, for $\K>0$, we have that $\RR<\pi/2\sqrt{\K}$.

    We start by noting that given $F$, we can obtain another function $G$ whose condition number is at most the one of $F$ and such that it is symmetric with respect to every rotation whose axis goes through $\xInit[\notilde]$. Formally, $G = G\circ \exp_{\xInit[\notilde]} \circ\ \sigma \circ \exp_{\xInit[\notilde]}^{-1}$ for a rotation $\sigma\in \operatorname{SO}(\n)$. Equivalently, the function $G(x)$ depends on $\norm{\exp_{\xInit[\notilde]}^{-1}(x)}$ only. Indeed, an average of $F$ and itself after performing an arbitrary rotation $\sigma$, that is \@ $(F+F\circ \exp_{\xInit[\notilde]} \circ\ \sigma \circ \exp_{\xInit[\notilde]}^{-1})/2$, has a condition number that is at most the condition number of $F$. This is due to the Hessian being linear and its maximum and minimum eigenvalues over the domain determining the condition number. That is, the smoothness constant can only decrease or stay the same after performing the average. It would only be the same if, at some point, the Hessian matrices of each of the two added functions both have the same eigenvector with maximum eigenvalue and it equals the smoothness constant. The argument for the minimum is analogous. This argument extends to the case in which we integrate the function, pointwise, over $\operatorname{SO}(\n)$ after applying a rotation. That is, defining $g(x) = \int_{SO(\n)} F \circ \exp_{\xInit[\notilde]} \circ\ \sigma \circ \exp_{\xInit[\notilde]}^{-1}(x) d\sigma$ we obtain a symmetric function with condition number that is at most the condition number of $F$. So without loss of generality we can solely study symmetric functions $G$ and in fact, due to the symmetries we do not lose generality if we work in dimension $\n=2$. 

    Denote $y_x = \norm{\exp_{\xInit[\notilde]}^{-1}(x)} \in \R$. We will express the condition number of $G$ by using the function $g :\R \to \R$, defined as $g(y_x) = G(x)$ for any point $x \in \Mk$. Note the function is well defined by the symmetry property on $G$. A basis formed by (two) eigenvectors of $\nabla^2 G(x)$ can be chosen to have vectors in the direction of $\exp_{x}^{-1}(\xInit[\notilde])$ and its normal. Indeed, either every vector is an eigenvector associated to the same eigenvalue, which satisfies the above, or by the symmetry of $\nabla^2 G(x)$, there exists a base $\{v_1,v_2\}$ of orthonormal eigenvectors, associated with different eigenvalues $\lambda_1 > \lambda_2$. By the symmetry of $G$ we have that $\lambda_1 = v_1^\top \nabla^2 G(x) v_1 = v_1'^{\top} \nabla^2 G(x) v_1'$, where $v_1'$ is the symmetric vector to $v_1$ with respect to $\exp_{x}^{-1}(\xInit[\notilde])$. However, since $\lambda_1 \neq \lambda_2$ then the only unit vectors $v$ that can satisfy $\lambda_1 = v^\top \nabla^2 G(x) v$ are $\pm v_1$, so $v_1=v_1'$ and therefore $v_1$ and $v_2$ can be taken to be in the direction of $\exp_{x}^{-1}(\xInit[\notilde])$ and its normal. Consequently, one eigenvalue of $\nabla^2 G(x)$ is $g''(y_x)$. We can compute the other eigenvalue by using the non-Euclidean cosine theorem, cf. \cref{thm:law_of_sines}. In order to do this, first note that $\nabla G(x)$ must be in the direction of $\exp_{x}^{-1}(\xInit[\notilde])$ by the symmetry of $G$ and it must be $\norm{\nabla G(x)} =  g'(y_x)$. Now given $x\in\M$ and small enough $\eta\in\R$, we consider a right geodesic triangle with vertices $\xInit[\notilde]$, $x$ and $z_\eta$, where $z_\eta = \exp_x(\eta v_2)$ for $v_2$ defined above. Recall it is a unit vector that is normal to $\exp_x^{-1}(\xInit[\notilde])$ and it is an eigenvector of $\nabla^2 G(x)$. The definition of $z_\eta$ implies that the angle between $\exp_x^{-1}(\xInit[\notilde])$ and $\exp_x^{-1}(x_2)$ is $\pi/2$ and $\dist(x,z_\eta) = \eta$. Let $\alpha(\eta)$ be the angle between $\exp_{z_\eta}^{-1}(\xInit[\notilde])$ and $\exp_{z_\eta}^{-1}(x)$. Since we are only interested about the eigenvalue of $\nabla^2 G(x)$ associated to the eigenvector $v_2$ we can project $\nabla G(z_\eta)$ onto $\exp_{z_\eta}^{-1}(x)$, which has norm $\norm{\nabla G(z_\eta) \cos(\alpha(\eta))}$. We compute the eigenvalue as
\begin{align*} 
 \begin{aligned}
     \lim_{\eta\to 0} \frac{\norm{\nabla G(z_{\eta})}\cos(\alpha(\eta))}{\eta} &= \norm{\nabla G(x)}\lim_{\eta\to 0} \frac{\cos(\alpha(\eta))}{\eta} \\
     & \circled{1}[=] g'(y_x) \lim_{\eta\to 0} \frac{\ck[\dist(x_0,x)]-\ck[\eta]\ck[\dist(x_0,z_\eta)]}{\K\sk[\dist(x_0,z_\eta)] \eta \sk[\eta]} \\
     & \circled{2}[=] g'(y_x) \lim_{\eta\to 0}\frac{\ck[\dist(x_0,x)](1-\ck[\eta][2])/\K}{\sk[\dist(x_0,z_\eta)]\eta\sk[\eta]}  \\
     & \circled{3}[=] g'(y_x)\cotk[\dist(x_0,x)] = g'(y_x)\cotk[y_x].
  \end{aligned}
\end{align*}
    Above, $\circled{1}$ uses the cosine theorem, cf. \cref{thm:law_of_sines}, applied as
    \[
        \ck[\dist(x_0,x)] = \ck[\dist(x,z_\eta)]\ck[\dist(x_0,z_\eta)] + K\cos(\alpha(\eta))\sk[\dist(x_0,z_\eta)]\sk[\dist(x,z_\eta)].
    \] 
    Recall that we have $\eta = \dist(x, z_{\eta})$ by definition. Equality $\circled{2}$ uses the cosine theorem again, with a different ordering of the sides so we obtain
    \[
    \ck[\dist(x_0,z_\eta)] = \ck[\dist(x_0,x)]\ck[\dist(x,z_\eta)],
    \] 
    by using the right angle of the geodesic triangle. Finally $\circled{3}$ simplifies some terms, since $(1-\ck[\eta][2])/\K = \sk[\eta][2]$ and uses that $\dist(x_0,z_\eta)$ and $\sk[\eta]/\eta$ tend to $\dist(x_0,x)$ and $1$, respectively, when $\eta\to 0$. We conclude that the condition number of $G$ is
    \begin{equation}\label{eq:condition_number_riemannian_simplified}
        \kappa_G = \frac{\max_{y\in[0,\RR]}\{ g''(y), g'(y)\cotk[y] \}}{\min_{y\in[0,\RR]}\{ g''(y), g'(y)\cotk[y] \}}.
    \end{equation}
    We only need to prove that for any twice differentiable function $g: [0,\RR]\to\R$ with continuous second derivative, the quotient above is at least the value of the quotient that we obtain for $g(y) = y^2/2$, which is $(\RR\cotk[\RR])^{-\K}$. This is computed by noticing that, for that choice of $g$, we have that $g''(y)=1$, that if $\K=1$ then $g'(y)\cotk[y] \leq 1$ and it reaches its minimum at $y=\RR$. If $\K=-1$ then $g'(y)\cotk[y] \geq 1$ and it is maximum at $y=\RR$. Note this implies that the condition number of $\frac{1}{2}\dist(\xInit[\notilde], x)^2$ on $\BR$ is $(\RR\cotk[\RR])^{-\K}$, as it was advanced before.

    Given $g$, let $a,b$ be tight constants such that $g''(y) \in [a,b]$ for $y \in [0,\RR]$. Such constants must exist since $g''$ is a continuous function defined on a compact. We have ${g'(y) \leq by}$, since by the symmetry and differentiability of $G$ it must be $g'(0)=0$. We obtain a lower bound on $\kappa_G$ if we lower bound the numerator of \eqref{eq:condition_number_riemannian_simplified} by $\max_{y\in[0,\RR]}\{g''(y)\}=b$ and if we upper bound the denominator by $g'(\RR)\cotk[\RR]$. We obtain
    \[
        \kappa_G \geq \frac{b}{g'(\RR)\cotk[\RR]} \geq \frac{b}{Rb \cotk[\RR]} = \frac{1}{\RR\cotk[\RR]}.
    \] 
    Similarly, if we lower bound the denominator of \eqref{eq:condition_number_riemannian_simplified} by $aR\cotk[\RR] \leq \max_{y\in[0,\RR]}\{g'(y)\cotk[y]\}$ and upper bound the denominator by $a=\min_{y\in[0,\RR]}\{g''(y)\}$ we obtain
    \[
        \kappa_G \geq \frac{aR\cotk[\RR]}{a}= \RR\cotk[\RR].
    \] 
    For each case $\K\in\{1, -1\}$, there is only one of the lower bounds above such that its right hand side is greater than $1$ and it precisely matches the value of the condition number of $\frac{1}{2}\dist(\xInit[\notilde], x)^2$ we computed above.
\end{proof}

\begin{remark}\label{remark:constants} The previous proposition intuitively suggests that it could be unavoidable to have some particular constants depending on $\RR$ in the rates of any optimization algorithm. For starters, optimizing a g-convex function by adding a strongly g-convex regularizer and optimizing the resulting strongly g-convex problem would entail rates containing a factor depending on the condition number of the regularizer, which the proposition proves it is at least the value $(\RR\sqrt{\abs{\K}}\cotk[\RR])^{-\sign(\K)}$. This implies that in the case of positive curvature $\K=1$, a $\mu$-strongly g-convex and $\L$-smooth function defined on the ball $\BR$ must have condition number that is at least $\tan(\RR)/\RR \in [\frac{2}{\pi\cos(\RR)}, \frac{1}{\cos(\RR)}]$. This grows fast with $\RR$, but it is only natural if one takes into account that no strongly g-convex function exists if $\RR\geq \frac{\pi}{2}$, due to the space containing a full geodesic circle (so the constant function is the only g-convex function in this domain). Optimization in manifolds of positive curvature only makes sense in spaces of low diameter.

    The classical domain of application of accelerated methods for strongly convex functions consists of functions with large condition number $\kappa$, due to the $\sqrt{\kappa}$-dependence of the rates. For $\K=1$, the constants of our algorithm $1/\gammap=\cos^{-2}(\RR)$ and $1/\gamman=\cos^{-3}(\RR)$ (we also have the constant $\sqrt{44\max\{1, \RR^2\}}$ coming from $\Ltilde$) might seem large but they are a small polynomial of the minimum attainable condition number. If the condition number is large or, in its limit to infinity, whenever the function is g-convex, then acceleration is beneficial. For the case $\K=-1$ the previous proposition shows that the minimum condition number is $\RR\tanh(\RR)$. In this case, our constants are $1/\gammap=\cosh^{3}(\RR)$ and $1/\gamman=\cosh^{2}(\RR)$, and a constant of a similar nature coming from $\tilde{L}$ (cf, \hyperlink{proof:lemma:smoothness_of_transformed_function}{proof of \cref{lemma:smoothness_of_transformed_function}}), which do not present an analogous dependency with respect to the minimum attainable condition number as in the previous case. This exponential dependence could be due to the exponential volume that a ball contains in the hyperbolic space. Studying if these constants are necessary for a global full accelerated method is interesting open problem and future direction of research. Regardless, the essence of our results for $\mu$-strongly g-convex functions is that we can optimize at a full accelerated rate \textit{globally} as opposed to essentially fully accelerating in a small neighborhood of radius $O((\mu/\L)^{3/4})$ around the minimizer (this is explicit in \citep{zhang2018towards} and implicit in \citep{ahn2020nesterov} since the rates of \AGD{} are nearly achieved only after a number of steps that is what \RGD{} needs to reach the neighborhood). Note that, additionally, we can achieve acceleration in the g-convex case, which was not possible before. In any case, we note that in machine learning applications, it has been observed that the iterates do not get far from initialization \citep{nagarajan2019generalization}, especially in overparametrized models. Consequently, in such regime, $\RR$ being a small constant is not a strong assumption and the constants of our algorithms do not become significant. 
\end{remark}

In the sequel, we comment on the work of \citet{hamilton2021no, criscitiello2021negative}, that show a hardness result in this direction. Our intuition is that, due to the geometry, it is necessary to have an additive and/or multiplicative constant depending on $\RR$ on the optimal rates of convergence, similarly to the multiplicative constant $\RR$ that one has in the lower bound for the class of smooth and convex functions in the Euclidean space. And for \textit{easy} strongly g-convex functions (low condition number), this hardness could dominate convergence. However, when the condition number is large, which is the traditional regime of application of accelerated methods, or in its limit to infinity, that is in the case of g-convexity, acceleration becomes again a very useful tool.

\subsection{Comment on hardness results} \label{app:hardness_results}

\begin{remark}\label{remark:related_work_hardness_hyperbolic}
    After this work was publicly available, two lower bounds have been constructed, \citep{hamilton2021no, criscitiello2021negative}. For the first one \citep{hamilton2021no}, in the noisy setting, in the hyperbolic plane, the authors claim ``{\emph{[to have] dashed these hopes [of having Nesterov-like accelerated algorithms] by showing that acceleration is impossible even in the simplest of settings where we want to minimize a smooth and strongly geodesically convex function over the hyperbolic plane}}''. We argue here that this is not the case.

    \citet{hamilton2021no} essentially argue that, in their setting with a noisy oracle in the hyperbolic plane, one needs $\gtrsim \RR/\log(\RR)$ noisy queries to the gradient or function value for optimizing functions of the form $\dist(x, \xast[\notilde])^2$, while their condition numbers are $\L/\mu \approx \RR$ so obtaining rates $\lesssim \sqrt{\L/\mu}$ is impossible in general. But it does not preclude to have an algorithm with rates that are, for instance, $\lesssim \RR+\sqrt{\L/\mu}\log(1/\epsilon)$. Or a similar expression that involves some other additive or multiplicative constants depending on $\RR$. In fact, they are able to show that ``{\normalfont acceleration is impossible even in the simplest of settings}'' precisely because they study the simplest of settings! That is, they show there is some hardness depending on the geometry. In particular, when the condition number is low this hardness can dominate the convergence. For instance, for rates $\RR+\sqrt{\L/\mu}\log(1/\epsilon)$ the $\RR$ can dominate convergence unless $\L/\mu \geq \RR^2$ or $\epsilon$ is small enough. The lower bound does not mean that acceleration is doomed to fail. In fact, the problems for which acceleration gets the most improvement are ill-conditioned problems and for those one would expect to still have acceleration in their noisy setting. In particular, acceleration is of importance when $\L/\mu$ is large or in the limit to infinity, that is, when the function is g-convex.

    We note that \citet{hamilton2021no} independently proved a similar result as our \cref{prop:lower_bound_on_condition_number}, limited to the hyperbolic plane. In particular they show that the condition number for an $\L$-smooth $\mu$-strongly g-convex function $\F$ defined on the hyperbolic disk of curvature ${\K=-1}$ must be $\kappa_F = \L/\mu \geq \Omega(\RR)$, which is similar to the precise result that we found that had optimal constant $\RR\cot(\RR)$. 

    The lower bound of \citet{criscitiello2021negative} is a generalization of the previous one, in which they show that the $\widetilde{\Omega}(\RR)$ lower bound still holds for optimization with a deterministic first-order oracle, in hyperbolic spaces and some other more general negatively curved manifolds. Similarly, this lower bound applies when the condition number is small enough ($\approx \RR$), close to the lower bound on the condition number that we proved in \cref{prop:lower_bound_on_condition_number}. They also provide a lower bound in the case of smooth and only g-convex functions. Their lower bound requires to have $\RR$ growing as $\widetilde{\Theta}(\frac{1}{\epsilon})$ and in that case they show that \RGD{} is optimal. This result is in the same spirit as the previous one: The geometry causes some hardness, and if we allow to grow the feasible space enough, this hardness can dominate convergence. On the other hand, this does not mean that one cannot accelerate when, for instance, $\RR$ is fixed and $\epsilon$ is small enough.
\end{remark}

\subsection[Comment on the rates of eventually accelerated algorithms (AS20)]{Comment on the rates of eventually accelerated algorithms \citep{ahn2020nesterov}}

\begin{remark}\label{remark:comparison_rates_riemannian}
    The local algorithm in \citep{zhang2018towards} requires starting $O((\L/\mu)^{-3/4})$ close to the optimum and it finds an $\epsilon$-minimizer in $\bigo{\sqrt{\L/\mu}\log(1/\epsilon)}$. On the other hand \RGD{} has a convergence rate of $\bigo{\L/\mu \log(1/\epsilon)}$. Hence, we could run both algorithms in parallel and restart them every few iterations from the best of the two points that both algorithms yielded. In that case we would obtain the convergence rate $\bigotilde{\L/\mu + \sqrt{\L/\mu}\log (\mu/\epsilon)}$. Indeed, note that we would just compute twice as many gradients as if we run \RGD{} but we perform as well as if we first run \RGD{} until it gets into the desired neighborhood and then we run the local accelerated algorithm. And by $\mu$-strong g-convexity we can guarantee we are $\mu \bar{\epsilon}^2/2$-close to a minimizer if we are at an $\bar{\epsilon}$-minimizer so if we set $\bar{\epsilon}$ so that $\mu \bar{\epsilon}^2/2 = O((\L/\mu)^{-3/4})$ and run \RGD{} we reach the neighborhood after $O((\L/\mu)\log(\L/\mu))$ iterations.

    We note that this mix of \RGD{} and the local algorithm in \citep{zhang2018towards} enjoys the same worse case guarantee of \citep{ahn2020nesterov}. This latter work is a generalization of \citep{zhang2018towards} that \textit{eventually accelerates}. The proofs of this paper reveal that in order for their bound to reach accelerated rates the algorithm needs as much time as \RGD{} takes to reach the accelerating neighborhood of \citep{zhang2018towards}. Indeed, they can guarantee that for their iterates $y_t$, their algorithm converges at an accelerated rate $f(y_t) - f(x^\ast) \leq O(f(y_{t-1})-f(x^\ast))(1-\sqrt{\mu/\L}))$ when $t=\Omega^\ast(\frac{1}{\log(1/\lambda)}) = \Omega^\ast(\L/\mu)$, where $\lambda = \Omega(1-\mu/\L)$. A summary of rates is presented in \cref{table:comparisons:riemannian}, including this fact.
\end{remark}

\end{document}